\documentclass[12pt]{article}
\usepackage{amsthm}

\theoremstyle{theorem}
\newtheorem{theorem}{Theorem}

\newtheorem{proposition}[theorem]{Proposition}
\newtheorem{corollary}[theorem]{Corollary}
\newtheorem{lemma}{Lemma}

\theoremstyle{remark}%

\theoremstyle{definition}%
\usepackage{amsmath}
\usepackage{amsfonts}
\usepackage{amssymb}
\usepackage{mathtools}

\usepackage{cases}
\usepackage{xcolor}
\usepackage{enumitem}

\usepackage{pifont}% http://ctan.org/pkg/pifont
\newcommand{\cmark}{\ding{51}}
\newcommand{\xmark}{\ding{55}}

\newcommand{\dss}{\displaystyle}

\DeclareMathOperator{\spn}{\mathrm{Span}}

\DeclareMathOperator{\proj}{\mathrm{Proj}}
\newcommand{\mrm}[1]{\mathrm{#1}}

\newcommand{\dt}{\partial_t}

\newcommand{\idx}[0]{\,d\vec{x}}

\newcommand{\nablah}[0]{\nabla_h}
\newcommand{\deltah}[0]{\Delta_h}
\newcommand{\dvh}[0]{\mathrm{div}_h\,}
\newcommand{\dz}[0]{\partial_z}

\newcommand{\D}[2]{\partial_{#1}^{#2}}
\newcommand{\stHnorm}[2]{H_{#1}^{#2}}
\newcommand{\ttene}[2]{\mathcal E_{#1,#2}}
\newcommand{\subeqref}[2]{$\eqref{#1}_{#2}$}

\newcommand{\norm}[2]{\bigl\Arrowvert #1 \bigr\Arrowvert_{#2}}
\newcommand{\Lnorm}[1]{L^{#1}(\mathbb T^3)}
\newcommand{\Hnorm}[1]{H^{#1}(\mathbb T^3)}

\title{The soundproof model of an acoustic--internal waves system with low stratification}

\author{Didier Bresch\footnote{
Laboratoire de Math\'ematiques UMR5127 CNRS, Batiment le Chablais, {Universit\'e Savoie Mont-Blanc},  {{Le Bourget du Lac}, 
{Le Bourget du Lac}, {73376}, %\state{}, 
{France}}, {didier.bresch@univ-smb.fr}. 
}

\qquad
Rupert Klein\footnote{
{Institut f\"ur Mathematik}, {Freie Universit\"at Berlin}, {{Berlin}, {14195},  
{Germany}}, {rupert.klein@math.fu-berlin.de}.
}

\qquad
Xin Liu\footnote{
{Department of Mathematics}, {Texas A\&M University}, {%\street{Street}, 
{College Station}, {77843}, %\state{State}, 
{Texas, U.S.}}, 
{Weierstrass-Institut f\"ur Angewandte Analysis und Stochastik}, {Leibniz-Institut im Forschungsverbund Berlin}, {%\street{Street}, 
{Berlin}, {10117}, %\state{State}, 
{Germany}},
{{Department of Applied Mathematics and Theoretical Physics}, {University of Cambridge}, {%\street{Street}, 
{Cambridge}, {CB3 0WA}, %\state{State}, 
{United Kingdom}},
{stleonliu@gmail.com}.}
}
}

\begin{document}

\allowdisplaybreaks
\maketitle

\begin{abstract}
	This work is devoted to {investigating} a compressible fluid system with low stratification, which is driven by fast acoustic waves and internal waves. The approximation using a soundproof model is justified. More precisely, the soundproof model captures the dynamics of both the non-oscillating mean flows and the oscillating internal waves, while filters out the fast acoustic waves, of the compressible system with or without initial acoustic waves.
	Moreover, the fast-slow oscillation structure is investigated.
	
	
\end{abstract}

\section{Introduction}

\subsection{Motivations}\label{subsec:motivation}

The rigorous justification of the anelastic and pseudo-incompressible models for atmospheric flows \cite{OguraPhillips1962,LippsHemler1982,Durran1989} in the inviscid case remains a challenge for at least three technical reasons: First, under realistic conditions for the troposphere, the compressible flow model involves three asymptotically separated time scales, associated with advection (slow), internal gravity waves (intermediate), and acoustics (fast), respectively. The two sound-proof models still involve the slow and intermediate scales, see \cite{Klein2010}, and thus still depend on the scale separation parameter. In other words, the anelastic and pseudo-incompressible models are not ``limit models'' in the classical sense, e.g., of low Mach number analysis. The technical question to be rigorously answered therefore is: What is the relation between the compressible three-scale and the sound-proof pseudo-incompressible (or anelastic) two-scale models.  

Secondly, realistic atmospheric background states feature temperatures and local Brunt-V\"ais\"al\"a or buoyancy frequencies that depend on the spatial position. This leaves the fast linear system describing acoustic and internal wave modes with non-constant, space-dependent coefficients. The control of derivatives for non-constant coefficient systems using techniques of energy estimates is substantially more difficult than it is in the constant coefficient case. 

Thirdly, problems on the torus or in $ \mathbb T^d $  ($ d \in {2,3}$) are often technically easier to handle than bounded domain problems, except when the bounded domain problem has a natural extension through certain symmetries to the infinite or toroidal domain case. Owing to the presence of gravity, realistic atmospheric flows always include a bottom boundary of the critical type that does not lend itself to domain extensions that would preserve smoothness of solutions across the eliminated domain boundary. 

In this paper we make progress in addressing the first issue, i.e., model reduction from three to two asymptotically separated scales, while we avoid the non-constant coefficient problem and irregular behavior of solutions near the (bottom) boundary of the domain by introducing judicious simplifications in the original model, designed to render the physics of the scale interactions largely intact: Let us denote $R$ the gas constant of the fluid, $ T $ the background temperature and $ g $ the typical gravitational acceleration, respectively. Then by (i) considering a fluid layer much thinner than the pressure scale height $h_{sc} = R T / g$, we guarantee that the leading-order temperature, and with it the leading order speed of sound, are constant. By (ii) assuming a particular vertical stratification of entropy (or potential temperature), we guarantee that the buoyancy (or Brunt-V\"ais\"al\"a-) frequency is constant as well (see \ref{H3} in page \pageref{H3}). This renders the linear fast system describing acoustic and internal waves with constant coefficients. Finally, by (iii) letting the gravitational acceleration decay to zero smoothly towards the top and bottom domain boundaries, while maintaining a constant buoyancy frequency by choice of the entropy stratification, we obtain a problem that has a regular extension to a vertically periodic domain problem (see \ref{H1} and \ref{H2} in page \pageref{H1}). 

%Under these conditions, we prove that, without initial acoustic waves, solutions of the compressible and incompressible models remain asymptotically close over the slowest (advective) time scale as the small parameter representative of the Mach and Froude numbers vanishes.  More details are given in the following section describing the relation between the compressible model \eqref{eq:EEq} and the pseudo-incompressible system \eqref{eq:sndprf} through system \eqref{eq:rf-EEq}. Moreover, in the case with initial acoustic waves, we show that the solutions of the pseudo-incompressible model capture the dynamics of the mean flows and the internal waves in the compressible model. 

Under these conditions, our main results can be stated in an informal fashion as follows:

{
%{\setcounter{theorem}{-1}}
\begin{theorem}
	Consider the full compressible model (\eqref{eq:rf-EEq} in page \pageref{eq:rf-EEq}) with both acoustic and internal waves, and the pseudo-incompressible model (\eqref{eq:sndprf} in page \pageref{eq:sndprf}). 
	\begin{itemize}
		\item Without initial acoustic waves, solutions of the compressible and pseudo-incompressible models remain asymptotically close over the slowest (advective) time scale as the small parameter representative of the Mach and Froude numbers vanishes. See Theorem \ref{thm:sp-approximation} in page \pageref{thm:sp-approximation} for the detailed statement; 
		\item Moreover, in the case with initial acoustic waves, the solutions of the pseudo-incompressible model capture the dynamics of the mean flows and the internal waves in the compressible model. See Theorem \ref{thm:sp-ill-prepared} in page \pageref{thm:sp-ill-prepared} for more details. 
	\end{itemize}
\end{theorem}

%The detailed description of our results can be found in Theorem \ref{thm:sp-approximation} and Theorem \ref{thm:sp-ill-prepared} in pages \pageref{thm:sp-approximation} and \pageref{thm:sp-ill-prepared}, respectively.
More details are given in the following section describing the relation between the compressible model \eqref{eq:EEq} and the pseudo-incompressible system \eqref{eq:sndprf} through system \eqref{eq:rf-EEq}. 
}

%\todo{Pseudo-incompressible or incompressible model?}
{An explanatory remark regarding our use of the notion of the ``pseudo-incompressible'' model is in order: By the assumption of a shallow domain, the formal leading order divergence constraint emerging from the pressure evolution equation is the \emph{incompressibility constraint} $\nabla\cdot v = 0$ rather than the pseudo-incrompressibility constraint $\nabla\cdot (P v) = 0$, where $P(z)$ is a function of the vertical coordinate only. We nevertheless speak of the pseudo-incompressible model in the last paragraph because we show in section~4 that first order pseudo-incompressibility effects are important when closeness of the compressible and soundproof approximations are to be maintained over the slow advective time scale. In fact, in that section we study the intermediate model \eqref{eq:mid-s} {in page \pageref{eq:mid-s}}, which we anticipate here in a notation similar to that of \cite{Klein2010} which is likely more familiar to readers of the meteorological literature,
\begin{subequations}\label{eq:mid-s-met-notation}
\begin{align}
\frac{Dv}{Dt} + \nablah \pi 
  & = 0,
    \label{eq:mid-s-met-notation-v}\\
\frac{Dw}{Dt} + \dz \pi  
  & = \frac{\theta}{\varepsilon^\nu},
    \label{eq:mid-s-met-notation-w}\\
\frac{D\theta}{Dt} 
  & = \frac{S^{\varepsilon}}{\varepsilon^{\nu}} w, 
    \label{eq:mid-s-met-notation-theta}\\
\dvh (P^{\varepsilon}v) + \dz (P^{\varepsilon}w) 
  & = 0.
    \label{eq:mid-s-met-notation-P}
\end{align}
\end{subequations}
%
%
%\begin{equation}\label{eq:mid-s-met-notation}
%\begin{split}
%\dvh v + \dz w 
%  & = \varepsilon \mathcal A (\mathcal C G + \varepsilon^\mu \overline{\mathcal H}_0) w, 
%    \\
%\dt \widetilde{\mathcal H} + v\cdot\nablah \widetilde{\mathcal H} + w\dz \widetilde{\mathcal H} - \dfrac{1}{\mathcal B \, \varepsilon^{\nu}} w 
%  & = \widetilde G \cdot \widetilde{\mathcal H} w, 
%    \\
%\dt v + v \cdot \nablah v +  w\dz v + \frac{1}{\mathcal C}\nablah p 
%  & = 0,
%    \\
%\dt w +  v \cdot \nablah w +  w\dz w + \frac{1}{\mathcal C}\dz p + \dfrac{1}{\mathcal C\, \varepsilon^\nu} \widetilde{\mathcal H} 
%  & = 0,
%\end{split}
%\end{equation}
%%
where
\begin{equation}
\frac{D}{Dt} = \dt + v\cdot\nablah + w\,\dz\,,  
\end{equation}
and 
\begin{equation}
P^{\varepsilon}(z) = 1 - \varepsilon \widetilde{P}^{\varepsilon}(z)\,,
\qquad
S^{\varepsilon}(z,\theta) = S^{\varepsilon}_0 + \varepsilon^\nu S^{\varepsilon}_\nu(z) \theta\,.
\end{equation}
The system in \eqref{eq:mid-s-met-notation} captures part of the difference between incompressible and pseudo-incompressible dynamics. Thus, the divergence control in \eqref{eq:mid-s-met-notation-P} represents weak deviations from the constraint $\dvh v = 0$ of an incompressible flow that are due to the small but finite height of the flow domain. It is not equivalent to the pseudo-incompressible system, however, as it does not include its baroclinic nonlinearity which would be represented by pressure gradient terms $(\theta^{\varepsilon}_0(z) + \varepsilon^{\mu+\nu} \theta) (\nablah\pi, \dz\pi)$ in \eqref{eq:mid-s-met-notation-v} and \eqref{eq:mid-s-met-notation-w}. Our main point in section~\ref{sec:sound-proof-app} will be to show that the weak deviation from incompressibility, even though small, significantly improves the system's agreement with the full compressible model relative to the incompressible model.
}

\subsection{Description of the problem}

To model a compressible flow under the influence of an external force (e.g., earth gravity), the compressible Euler equations is considered. With low stratifications, the dimensionless system can be written as (see, e.g., \cite{FeireislSingularLimits}),
%Let $ \varepsilon \in (0,1) $. 
%The compressible Euler equations with low stratification
\begin{equation}\label{eq:EEq}
	\begin{cases}
		\dt \rho + \dvh(\rho v) + \dz (\rho w) = 0,  \\
		\dt(\rho v) + \dvh (\rho v\otimes v) + \dz (\rho w v) + \dfrac{1}{\varepsilon^2} \nablah p =0,  \\
		\dt (\rho w) + \dvh (\rho v w) + \dz (\rho w w ) + \dfrac{1}{\varepsilon^2} \dz p + \dfrac{1}{\varepsilon} \rho G(z) = 0,\\
		\dt p + v\cdot \nablah p + w \dz p + \gamma p (\dvh v + \dz w) = 0,
	\end{cases}
\end{equation}
where $ \varepsilon \in (0,1) $ denotes the small Mach number, and $ \rho,\, p,\, v $, and $ w $ are the scalar density, the pressure potential, the horizontal velocity field, and the vertical velocity, respectively. Here $ G(z) $ is the external force, causing stratification. As $ \varepsilon \rightarrow 0^+ $, system \eqref{eq:EEq} describes flows in the low Mach number region with low stratification, i.e., the Boussinesq scale. The external force $ \rho G(z) $ causes the flow to form stratification as $ \varepsilon \rightarrow 0^+ $. One particular stratification profile considered in this paper is characterised by
\begin{equation}\label{def:stratification}
	\dz \theta = \mathcal O(\varepsilon^\mu), \qquad \mu \in (0,1),
\end{equation}
where $ \theta $ is the potential temperature defined by
\begin{equation}\label{def:potential-temperature}
	\theta := p^{1/\gamma} \rho^{-1}.
\end{equation}
In addition, the Exner pressure, defined by
\begin{equation}\label{def:exner-pressure}
	\varpi := \dfrac{\gamma}{\gamma-1} p^{\frac{\gamma-1}{\gamma}},
\end{equation}
is commonly used in meteorological study (\cite{Klein2010,Klein2009} etc.). 
%Next, denote by Exner pressure $ \pi := \dfrac{\gamma}{\gamma-1} p^{\frac{\gamma-1}{\gamma}} $, and the potential temperature $ \theta := p^{1/\gamma} \rho^{-1} $.
Then \eqref{eq:EEq} is equivalent to, described by the new unknowns $ (\varpi, \theta, v, w) $,
\begin{equation}\label{eq:EEq-2}
	\begin{cases}
		\dt \varpi + v \cdot\nablah \varpi + w \dz \varpi + (\gamma-1) \varpi (\dvh v + \dz w ) = 0, \\
		\dt \theta^{-1} + v \cdot\nablah \theta^{-1} + w \dz \theta^{-1} = 0,\\
	\theta^{-1} (\dt v + v \cdot \nablah v + w \dz v ) + \dfrac{1}{\varepsilon^2} \nablah \varpi = 0, \\
	\theta^{-1} (\dt w + v \cdot \nablah w + w \dz w) + \dfrac{1}{\varepsilon^2} \dz \varpi + \dfrac{1}{\varepsilon} \theta^{-1} G(z) = 0. 
	\end{cases} 
\end{equation}
In order the investigate the stratification with \eqref{def:stratification}, the following ansatz is introduced:
%Next, we introduce the leading order and the perturbation variables. Denote by
\begin{equation}\label{def:perturbation}
	\varpi : = \varpi_0 + \varepsilon \widetilde \varpi, \quad \theta^{-1} := \theta_0^{-1} + \varepsilon^\mu G^{-1} \overline{\mathcal H}_0  + \varepsilon^{\mu + \nu} G^{-1} \widetilde{\mathcal H},
\end{equation}
where $ \varpi_0, \theta_0 $ are constant, and $ \overline{\mathcal H}_0 = \overline{\mathcal H}_0(z) $.
Then from \eqref{eq:EEq-2}, one can derive, with
\begin{equation} \label{def:mu-nu}
	\mu + 2 \nu = 1, 
\end{equation}
\begin{equation}\label{eq:EEq-2-1}
	\begin{cases}
		\dfrac{1}{(\gamma-1) \varpi} (\dt \widetilde \varpi + v \cdot \nablah \widetilde \varpi + w \dz \widetilde \varpi ) + \dfrac{1}{\varepsilon} (\dvh v + \dz w) = 0, \\
		\dfrac{1}{G \dz (G^{-1} \overline{\mathcal H}_0)} (\dt \widetilde{\mathcal H} + v \cdot \nablah \widetilde{\mathcal H} + G w \dz (G^{-1} \widetilde{\mathcal H})) + \dfrac{1}{\varepsilon^{\nu}} w = 0,\\
		\theta^{-1} ( \dt v + v\cdot \nablah v + w \dz v) + \dfrac{1}{\varepsilon} \nablah \widetilde \varpi = 0, \\
		\theta^{-1} (\dt w + v \cdot \nablah w + w \dz w ) + \dfrac{1}{\varepsilon} ( \dz \widetilde \varpi + \theta_0^{-1} G(z) + \varepsilon^\mu \overline{\mathcal H}_0 ) + \dfrac{1}{\varepsilon^{\nu}} \widetilde{\mathcal H} = 0. 
	\end{cases}
\end{equation} 
After denoting by 
\begin{equation}\label{def:perturbation-q}
	\widetilde q := \widetilde \varpi + \theta_0^{-1} \int_0^z G(z') \,dz' + \varepsilon^{\mu} \int_0^z \overline{\mathcal H}_0(z') \,dz'
\end{equation}
and multiplying the first equation of \eqref{eq:EEq-2-1} with $ \varpi / \varpi_0 $,
we arrive at
\begin{equation}\label{eq:EEq-3}
	\begin{cases}
		\dfrac{1}{(\gamma-1) \varpi_0} (\dt \widetilde q + v \cdot \nablah \widetilde q + w \dz \widetilde q - \theta_0^{-1} G w - \varepsilon^\mu \overline{\mathcal H}_0 w ) \\
		\qquad \qquad + \dfrac{1}{\varepsilon} (\dvh v + \dz w) = - \varpi_0^{-1} \widetilde q (\dvh v + \dz w) \\
		\qquad\qquad  + \varpi_0^{-1} (\dvh v + \dz w) ( \theta_0^{-1} \int_0^z G(z') \,dz' + \varepsilon^{\mu} \int_0^z \overline{\mathcal H}_0(z') \,dz' ) , \\
		- \dfrac{1}{G\dz(G^{-1} \overline{\mathcal H}_0)} (\dt \widetilde{\mathcal H} + v \cdot \nablah \widetilde{\mathcal H} + w \dz \widetilde{\mathcal H} - \dfrac{\dz G}{G} \widetilde{\mathcal H} w % + G w \dz (G^{-1} \widetilde{\mathcal H})
		) - \dfrac{1}{\varepsilon^{\nu}} w = 0,\\
		\theta^{-1} ( \dt v + v\cdot \nablah v + w \dz v) + \dfrac{1}{\varepsilon} \nablah \widetilde q = 0, \\
		\theta^{-1} (\dt w + v \cdot \nablah w + w \dz w ) + \dfrac{1}{\varepsilon} \dz \widetilde q + \dfrac{1}{\varepsilon^{\nu}} \widetilde{\mathcal H} = 0. 
	\end{cases}
\end{equation}

On the other hand, denote by
\begin{equation}\label{def:vec-opt}
	\begin{gathered}
		U := \left( \begin{array}{c}
			\widetilde q \\ \widetilde{\mathcal H} \\ v \\ w
		\end{array} \right), \quad \mathcal L_a U := \left( \begin{array}{c}
			\dvh v + \dz w \\ 0 \\ \nablah \widetilde q \\ \dz \widetilde q
		\end{array} \right), \quad \text{and} \,\, \mathcal L_g U := \left( \begin{array}{c}
			0 \\ - w \\ 0 \\ \widetilde{\mathcal H}
		\end{array}\right). 
	\end{gathered}
\end{equation}
Notice that operators $ \mathcal L_a $ and $ \mathcal L_g $ are anti-symmetric with respect to the $ L^2 $-inner product and induce oscillations, corresponding to acoustic waves and internal waves of solutions to system \eqref{eq:EEq-3}, respectively. 

Unfortunately, in general, the anti-symmetry property does not hold for general boundary conditions and systems with non-constant coefficients in more regular Sobolev space, for instance $ H^s $, $ s > 0 $. %The reason being the anti-symmetry property of operators $ \mathcal L_a $ and $ \mathcal L_g $  does not survive after applying derivatives; that is, $ \partial \mathcal L_a $ and $ \partial \mathcal L_g $ are no longer anti-symmetric with respect to the $ L^2 $-inner product in general scenario. 
This is a major difficulty in the study of asymptotic limit of fast oscillation systems (see, e.g., \cite{Majda1984}). To resolve this difficulty is beyond the scope of this paper. Instead, we will introduce a system closely related to system \eqref{eq:EEq-3}, which still captures the acoustic waves and the internal waves driven by $ \mathcal L_a $ and $ \mathcal L_g $, respectively. 

\bigskip

\noindent  As explained in section \ref{subsec:motivation}, we will assume the following hypothesis in this first work.
\newcounter{hypothesis}
\begin{enumerate}[label = {\bf H\arabic{enumi})}, ref = {\bf H\arabic{enumi})}]
\setcounter{enumi}{\value{hypothesis}}
\item \label{H1} If one considers \eqref{eq:EEq-3} in $  \lbrace (x,y,z) \in  \mathbb T^2 \times 0.5 \mathbb T \rbrace $, and assumes that $ G, \overline{\mathcal H}_0 $ are odd in the $ z $-variable, then the following symmetry invariance holds:
\begin{equation*}\label{SYM}\tag{SYM}
\begin{gathered}
\text{$ \widetilde q, \widetilde{\mathcal H}, v, $ and $ w $ are even, odd, even, and odd, respectively,}\\
	\text{with respect to the $ z $-variable.}
\end{gathered}
\end{equation*}
Therefore, by, in addition, assuming $ G, \overline{\mathcal H}_0 $ to be smooth enough in $ \mathbb T^3 $, one can consider \eqref{eq:EEq-3} in $ \mathbb T^3 $. 
%In fact, we will assume that
%\begin{equation}\label{asmpt:gravity-potential}
%	G = \sin (2\pi z), \qquad \overline{\mathcal H}_0 = \dfrac{\text{constant}}{2\pi} \sin (2\pi z) \cos (2\pi z).
%\end{equation}
%contains a singular term $ \dz G / G $, which, unfortunately, we don't know how to deal with it. 
\item \label{H2} Noticing that in \subeqref{eq:EEq-3}{2}, the term	
$ \widetilde{\mathcal H} w \dz G/G$ becomes singular when $ G $ approaches $ 0 $. The function $ \frac{\dz G}{G} $ is replaced by another function $ \widetilde G $, which is odd with respect to the $ z $-variable and smooth in $ \mathbb T^3 $. For the same reason, we replace $ G^{-1} $ in \eqref{def:perturbation} by $ \widetilde G $. 

%{\par\noindent\bf Remark: 
%Assuming $ G, \overline H_0 $ are odd in the $ z $-variable, then \eqref{eq:EEq-3} is invariable respect to the following symmetry:
%\begin{center}
%	$ \widetilde q, \widetilde{\mathcal H}, v, $ and $ w $ are even, odd, even, and odd, respectively, with respect to the $ z $-variable.
%\end{center}
%}
\item \label{H3} The Brunt-V\"ais\"al\"a frequency $ \mathfrak N $, defined by,
\begin{equation}\label{assm:constant-gravity-wave-freq}
%	\mathcal N = \text{constant}, \qquad \text{i.e.}, 
%	\quad 
	\mathfrak N^2 := - G\dz (G^{-1} \overline{\mathcal H}_0), %  = \text{constant}.
\end{equation}
is constant. 
\setcounter{hypothesis}{\value{enumi}}
\end{enumerate}

Then, after denoting the positive constants
\begin{equation}\label{def:constants}
	\begin{gathered}
		\mathcal A := \dfrac{1}{(\gamma-1)\varpi_0}, \quad \mathcal B := - \dfrac{1}{G\dz(G^{-1}\overline{\mathcal H}_0)}, \quad \mathcal C:= \theta_0^{-1},
	\end{gathered}
\end{equation}
%Then \eqref{eq:EEq-3} can be written as, after replacing the singular term $ \dz G / G $ in \subeqref{eq:EEq-3}{2} with $ \widetilde G $ (see \eqref{drop-singular}),
we introduce the following system: in $ \mathbb T^3 $, with $ \mu + 2\nu = 1 $, 
\begin{equation}\label{eq:rf-EEq}
	\begin{cases}
	\mathcal A \dt \widetilde q + \mathcal A v \cdot \nablah \widetilde q + \mathcal A w \dz \widetilde q  + \dfrac{1}{\varepsilon} (\dvh v + \dz w) \\
	\qquad = \mathcal A \mathcal C G  w + \varepsilon^{\mu} \mathcal A \overline{\mathcal H}_0  w
		- \varpi_0^{-1} \widetilde q (\dvh v + \dz w) \\
		\qquad\qquad  + \varpi_0^{-1} (\dvh v + \dz w) ( \mathcal C \int_0^z G(z') \,dz' + \varepsilon^{\mu} \int_0^z \overline{\mathcal H}_0(z') \,dz' ) 
		, \\
		\mathcal B \dt \widetilde{\mathcal H} + \mathcal B v \cdot \nablah \widetilde{\mathcal H} + \mathcal B w \dz \widetilde{\mathcal H} - \dfrac{1}{\varepsilon^{\nu}} w = \mathcal B  \widetilde G \cdot \widetilde{\mathcal H} w ,\\
		\vartheta  \dt v + \vartheta v\cdot \nablah v + \vartheta w \dz v + \dfrac{1}{\varepsilon} \nablah \widetilde q = 0, \\
		\vartheta \dt w + \vartheta v \cdot \nablah w + \vartheta w \dz w + \dfrac{1}{\varepsilon} \dz \widetilde q + \dfrac{1}{\varepsilon^{\nu}} \widetilde{\mathcal H} = 0,
	\end{cases}
\end{equation}
where $ \widetilde q, \, \widetilde{\mathcal H}, \, v, \,  w $ admit the symmetry \eqref{SYM}, $ G, \, \overline{\mathcal H}_0, \, \widetilde G $ are odd in the $ z $-variable and smooth enough in $ \mathbb T^3 $, and $ \vartheta = \mathcal C + \mathcal O(\varepsilon^{\mu}) $ are given by
\begin{equation}\label{def:perturbation-theta}
	\vartheta := \mathcal C + \varepsilon^\mu \widetilde G \overline{\mathcal H}_0  + \varepsilon^{\mu + \nu} \widetilde G \widetilde{\mathcal H}.
\end{equation}
System \eqref{eq:rf-EEq} is complemented with initial data
\begin{equation}\label{def:initial-EEq}
	(\widetilde{q},  \widetilde{\mathcal{H}}, v, w)\vert_{t=0} = (\widetilde{q}_\mrm{in}, \widetilde{\mathcal{H}}_\mrm{in}, v_\mrm{in}, w_\mrm{in} ).
\end{equation}
Accordingly, $ (\lbrack \dt^\alpha \widetilde q\rbrack_\mrm{in}, \lbrack \dt^\alpha \widetilde{\mathcal H}\rbrack_\mrm{in}, \lbrack \dt^\alpha v\rbrack_\mrm{in}, \lbrack \dt^\alpha w\rbrack_\mrm{in}), \, \alpha \in \mathbb N^+ $, are defined inductively after shifting spatial derivatives to temporal derivatives using equations of \eqref{eq:rf-EEq}.

\bigskip

Before stating our results, we would like to make a few perspective remarks. As one can see, in system \eqref{eq:rf-EEq}, the linear oscillator is given by
$$
\dfrac{1}{\varepsilon} \mathcal L_a + \dfrac{1}{\varepsilon^\nu} \mathcal L_g,
$$
i.e., a combination of the acoustic oscillator and the internal wave oscillator. Moreover, since $ \nu \in (0,1) $, as $ \varepsilon \rightarrow 0^+ $, the oscillation induced by $ \frac{1}{\varepsilon} \mathcal L_a $ is much faster than that of $ \frac{1}{\varepsilon^\nu} \mathcal L_g $. This means that the acoustic waves will be averaged out (or filtered out) before the internal waves. Owing to such a phenomena, we propose a pseudo-incompressible/soundproof model, similar to \cite{Klein2010}:
\begin{equation}\label{eq:sndprf}
	\begin{cases}
		\dvh v_\mrm{sp} + \dz w_\mrm{sp} = 0, \\
		\mathcal B \dt \widetilde{\mathcal H}_\mrm{sp} + \mathcal B v_\mrm{sp}\cdot\nablah \widetilde{\mathcal H}_\mrm{sp} + \mathcal B w_\mrm{sp}\dz \widetilde{\mathcal H}_\mrm{sp} - \dfrac{1}{\varepsilon^{\nu}} w_\mrm{sp} = \mathcal B \widetilde G \cdot \widetilde{\mathcal H}_\mrm{sp} w_\mrm{sp}, \\
		\mathcal C \dt v_\mrm{sp} + \mathcal C v_\mrm{sp} \cdot \nablah v_\mrm{sp} + \mathcal C w_\mrm{sp}\dz v_\mrm{sp} + \nablah p_\mrm{sp} = 0,\\
		\mathcal C \dt w_\mrm{sp} + \mathcal C v_\mrm{sp} \cdot \nablah w_\mrm{sp} + \mathcal C w_\mrm{sp}\dz w_\mrm{sp} + \dz p_\mrm{sp} + \dfrac{1}{\varepsilon^\nu} \widetilde{\mathcal H}_\mrm{sp} = 0,
	\end{cases}
\end{equation}
whose solutions will be an approximation to the solutions to system \eqref{eq:rf-EEq} minus the acoustic waves, with or without initial acoustic waves.

Aside from the soundproof approximation, we would like to investigate how the mixture of acoustic waves and internal waves with different frequencies affects the total oscillation of the system. To do so, we will first consider a linear system associated with \eqref{eq:rf-EEq} and the corresponding eigenvalue problem. By comparing the distribution of eigenvalues with that of eigenvalues associated with $ \mathcal L_a $, we have a more precise description of how the internal waves intertwine with the acoustic waves at the level of eigenvalues. Based on the understanding of the linear theory, we will discuss the fast-slow wave interaction of system \eqref{eq:rf-EEq} in the end. 

We would like to mention, our current study is strongly motivated by previous study on flows with strong stratification (see, e.g., \cite{Klein2010,Klein2009}), to which we refer readers for more metrological perspectives. A recent paper \cite{DavidLannes2019} focuses on the soundproof model with stratification to better understand the internal waves. 

The justification of singular limits is rooted back to:
\begin{itemize}
	\item Fast oscillation limit with only one parameter can be found in \cite{Schochet1988,Schochet1994,SchochetCMP1986}. For geophysical purposes, see for instance \cite{Klainerman1981,Klainerman1982,Majda1984}. 
%	\item We also refer interested readers to \cite{Schochet1988,Schochet1994,SchochetCMP1986} for the theory of fast oscillation limit. 
		
	\item Fast oscillation limit with several parameters linked together can be found in \cite{SchochetXu2020}. For geophysical purposes, see, for instance, \cite{feireisl2012multi,FeireislSingularLimits} for weak solutions, and \cite{BrMe} for strong solutions.
%	\item 
%	 Multi-scale limit, including stratification, can be found in \cite{feireisl2012multi,FeireislSingularLimits} for weak solutions, and \cite{SchochetXu2020,BrMe} for strong solutions.
%	\item 
%	 See \cite{FeireislSingularLimits} for a comprehensive monograph in the case of viscous flows with or without stratification. 
	
\end{itemize}

\bigskip

In this work, we do {\bf not} perform fast oscillation limit. Instead, we want to prove that the non-oscillating mean flows and the oscillating internal waves of solutions of two singular systems (the compressible and pseudo-incompressible/soundproof models) remain asymptotically close over the slowest time scale. The main theorems in this paper consider the initial data of the following types in the full compressible system \eqref{eq:rf-EEq}, as in table \ref{tb:waves}:
\begin{center}
\begin{table}[h]
\begin{center}
\begin{minipage}{174pt}
\caption{Waves in the initial data}\label{tb:waves}
\begin{tabular}{@{}lll@{}}
\hline
 {Waves in the initial data} & Theorem \ref{thm:sp-approximation}  & Theorem \ref{thm:sp-ill-prepared} \\
\hline
Mean flows & \cmark & \cmark \\ 
Internal waves & \cmark & \cmark \\ 
Acoustic waves & \xmark  & \cmark \\ 
\hline
\end{tabular}
\end{minipage}
\end{center}
\end{table}
\end{center}
and compare the solutions to those of the soundproof model \eqref{eq:sndprf}. In both cases, we justify the rigidity of capturing the dynamics of the mean flows and the internal waves of the full compressible system using the soundproof approximation. 

\bigskip

More precisely, our first result provides the comparison of solutions to the two singular systems in the well-prepared data (without acoustic waves) case:
%%
%\begin{theorem}[\chgd{qualitative statement}]\label{thm:sp-approximation-qualitative} 
%\chgd{For two sound-proof systems, the incompressible and a pseudo-incompressible model that is intermediate between the incompressible and the fully compressible systems, we demonstrate that their fast internal wave modes are asymptotically close over the slow advective time scale.}
%\end{theorem}
%
\begin{theorem}[{Mean flows + Internal waves}]\label{thm:sp-approximation}
Let $ 0 < 2 \nu < 1 $. 
Denote the initial data to the intermediate model \eqref{eq:mid-s}, below {in page \pageref{eq:mid-s}}, as
\begin{equation*}
	(\widetilde{\mathcal H}_{\mrm{ms},\mrm{in}}, v_{\mrm{ms},\mrm{in}},w_{\mrm{ms},\mrm{in}}) \quad \in H^3(\mathbb T^3),
\end{equation*}
and the initial data to the soundproof model \eqref{eq:sndprf} as
\begin{equation*}
	(\widetilde{\mathcal H}_{\mrm{sp},\mrm{in}},v_{\mrm{sp},\mrm{in}}, w_{\mrm{ps},\mrm{in}}) \quad \in H^3(\mathbb T^3),
\end{equation*}
satisfying the pseudo-incompressible and incompressible conditions \subeqref{eq:mid-s}{1} and \subeqref{eq:sndprf}{1}, respectively. Then there exist local-in-time solutions to the intermediate model \eqref{eq:mid-s}, below, and the soundproof model \eqref{eq:sndprf}, denoted as
\begin{equation*}
	( p_\mrm{ms}(s), \widetilde{\mathcal H}_\mrm{ms}(s), v_\mrm{ms}(s), w_\mrm{ms}(s) ) \quad \text{and} \quad ( p_\mrm{sp}(s), \widetilde{\mathcal H}_\mrm{sp}(s), v_\mrm{sp}(s), w_\mrm{sp}(s) ),
\end{equation*}
respectively, in $ L^\infty ((0,T_\mrm{ms+sp}), H^3(\mathbb T^3))\cap C([0,T_\mrm{ms+sp}), H^2(\mathbb T^3)) $ for some $ T_\mrm{ms+sp} \in (0,\infty) $. 

Meanwhile, denote by 
\begin{equation*}
	( \widetilde q(s), \widetilde{\mathcal H}(s), v(s), w(s) )
\end{equation*}
the solution to {compressible} system \eqref{eq:rf-EEq} in Proposition \ref{thm:uniform-est}, below, with initial data satisfying \eqref{cnt-thm:uniform-est-1} for any fixed $ \sigma \in (0,\mu] $. %with $ \sigma = \mu $. 
Then there exist $ T_\mrm{app} \in (0,\infty) $ and $ \mathcal C_\mrm{app}\in (0,\infty) $, depending only on the initial data above, such that, for $ \varepsilon \in (0,1) $,
\begin{equation}\label{sp-ap-est}
	\begin{gathered}
	\sup_{\mathclap{0\leq s \leq T_\mrm{app}}} \norm{\widetilde q(s) - \varepsilon p_\mrm{sp}(s), \widetilde{\mathcal H}(s) - \widetilde{\mathcal H}_\mrm{sp}(s), v(s) - v_\mrm{sp}(s), w(s) - w_\mrm{sp}(s) }{\Lnorm{2}} \\
	\leq \mathcal C_\mrm{app} \biggl( \varepsilon^{\max\lbrace \mu - \nu,  \mu - \sigma \rbrace} \\
	+ \norm{\widetilde{\mathcal H}_{\mrm{ms},\mrm{in}}-\widetilde{\mathcal H}_{\mrm{sp},\mrm{in}},
	v_{\mrm{ms},\mrm{in}}-v_{\mrm{sp},\mrm{in}},
	w_{\mrm{ms},\mrm{in}}-w_{\mrm{sp},\mrm{in}}}{\Lnorm{2}} \\
	 + \norm{\widetilde q_\mrm{in}- \varepsilon  p_{\mrm{ms},\mrm{in}}, \widetilde{\mathcal H}_\mrm{in} - \widetilde{\mathcal H}_{\mrm{ms},\mrm{in}}, v_{\mrm{in}} - v_{\mrm{ms},\mrm{in}}, w_{\mrm{in}} - w_{\mrm{ms},\mrm{in}}}{\Lnorm{2}} 
	\biggr).
	\end{gathered}
\end{equation}
Here $ p_\mrm{ms}, \, p_{\mrm{ms},\mrm{in}} $, and $ p_\mrm{sp}  $ are given by solutions to elliptic problems \eqref{eq:pressure-ms} and \eqref{eq:pressure-sp}, and can be estimated as in \eqref{est:pressure-ms-1} and \eqref{est:pressure-1}, below {in pages \pageref{eq:pressure-ms}, \pageref{eq:pressure-sp}, \pageref{est:pressure-ms-1}, and \pageref{est:pressure-1}, respectively}. 
%, and $ \varepsilon \in (0,\varepsilon_{0,\mu}) $ with $ \varepsilon_{0,\mu} \in (0,1) $ as in theorem \ref{thm:uniform-est}.
Here, recall that $ \mu + 2\nu = 1 $. 
%In particular, with $ \mu - \nu  = 1 - 3\nu >0 $ and proper initial data, \eqref{sp-ap-est} provides the error estimates of the soundproof approximation.
\end{theorem}%

The uniform-in-$\varepsilon$ estimate of solutions to \eqref{eq:sndprf} and \eqref{eq:mid-s} can be found in \eqref{est:uni-total-sp}, \eqref{est:pressure-1}, \eqref{est:uni-total-ms}, and \eqref{est:pressure-ms-1}, respectively.  
In particular, with $ \max\lbrace \mu - \nu, \mu -\sigma\rbrace  = \max \lbrace 1 - 3\nu, 1-2\nu - \sigma \rbrace >0 $ and proper initial data (so that the initial data on the right hand side of \eqref{sp-ap-est} is small), \eqref{sp-ap-est} provides the error estimates and convergence rate of the soundproof approximation with ``well-prepared'' initial data. 

The term $ \varepsilon^{\max\lbrace \mu - \nu,  \mu - \sigma \rbrace} $ in the error estimate \eqref{sp-ap-est} results from the comparison between the terms $ (\mathcal C+ \mathcal O(\varepsilon^\mu)) (\dt v, \dt w) $ and $ \mathcal C (\dt v_\mrm{ms}, \dt w_\mrm{ms}) $ in section \ref{subsec:itmd-ap}, which can be either written as
$$
(\mathcal C+ \mathcal O(\varepsilon^\mu)) (\dt (v-v_\mrm{ms}), \dt (w-w_\mrm{ms})) + \mathcal O(\varepsilon^\mu) (\dt v_\mrm{ms}, \dt w_\mrm{ms})
$$
or 
$$
\mathcal C  (\dt (v-v_\mrm{ms}), \dt (w-w_\mrm{ms})) + \mathcal O(\varepsilon^\mu) (\dt v, \dt w).
$$
See \eqref{eq:ms-perturb} and \eqref{eq:ms-perturb-2}, respectively, for details. Since $ (\dt v_\mrm{ms}, \dt w_\mrm{ms}) \simeq (\dt v_\mrm{sp}, \dt w_\mrm{sp}) \simeq \mathcal O(\varepsilon^{-\nu}) $ (as can be seen through \eqref{eq:sndprf}) and $ (\dt v, \dt w) \simeq \mathcal O(\varepsilon^{-\sigma} )$ thanks to Proposition \ref{thm:uniform-est}, this results in our freedom of choice in the error estimate \eqref{sp-ap-est}. For more details, we refer readers to the estimate of $ \mathfrak I_3 $ in \eqref{est:I3-ms} in page \pageref{est:I3-ms} and \eqref{rm:error-1} in page \pageref{rm:error-1} of the proof of the theorem. 

%\todo{more physical remark?}
%\klein{[RK: I think that the next paragraph, which starts with "Heuristically...", IS rather physical. If you agree with my insertion of ``acoustic'', then I would leave it at that.]}

Heuristically speaking, for larger $ \nu \in [1/3,1/2) $, the oscillating rates of the internal gravity waves and the acoustic waves (if non-trivial, of $ \mathcal O(\varepsilon^{-\nu}) $ and $ \mathcal O(\varepsilon^{-1}) $, respectively) are closer to each other.  
In order to control 
the error $ \varepsilon^{\max\lbrace \mu - \nu,  \mu - \sigma \rbrace}= \varepsilon^{\max\lbrace 1 - 3\nu,  1-2\nu - \sigma \rbrace} $, we need $ \sigma < \mu = 1-2\nu $, i.e., smaller value of $ \sigma $ (hence weaker {acoustic} waves in the full compressible system), to avoid strong interaction between acoustic waves and internal waves in the full compressible system.

We would also like to point out that the constraint $ 0 < 2 \nu < 1 $ is physical (see the formal deviation between \eqref{eq:EEq-2} and \eqref{eq:EEq-3}). %  while $ 3\nu < 1 $ is due to the mathematical analysis.

\bigskip

The second result will provide the convergence in the ill-prepared data (with acoustic waves) case:
\begin{theorem}[Mean flows + Internal waves + Acoustic waves]\label{thm:sp-ill-prepared}
	Under the same assumptions as in Theorem \ref{thm:sp-approximation}, denote by $ U=(\widetilde q, \widetilde{\mathcal H}, v, w) $, the solution to \eqref{eq:rf-EEq}, and write $ U = U^\mrm{mf}_\varepsilon + U^\mrm{gw}_\varepsilon + U^\mrm{aw}_\varepsilon $ as the summation of the mean flows, the internal waves, and the acoustic waves. Let $ (p_\mrm{sp}, U_\mrm{sp}) = (p_\mrm{sp}, \widetilde{\mathcal H}_\mrm{sp}, v_\mrm{sp}, w_\mrm{sp}) $ be the solution to the soundproof approximation \eqref{eq:sndprf} with initial data capturing the initial mean flows and internal waves of the full compressible system \eqref{eq:rf-EEq} (see \eqref{est:proj-mg-sp-102} for the exact meaning of this statement). Let $ \mathcal P_\mrm{rd}: (\widetilde q, \widetilde{\mathcal H}, v, w) \mapsto  (\widetilde{\mathcal H}, v, w) $ and $ \lbrace T_k \rbrace_{k \in \mathbb N} $ be the vector-dimension reduction and finite dimension truncation defined in \eqref{def:dr-prjtn} and \eqref{def:k-truncation} {of pages \pageref{def:dr-prjtn} and \pageref{def:k-truncation}}, respectively. Then for any positive integer $ K $, one has
	\begin{equation}
		\sup_{\mathclap{0<t<T_{\sigma, \mrm{mg}}}}\norm{T_K \mathcal P_\mrm{rd} (U^\mrm{mf}_\varepsilon +  U^\mrm{gw}_\varepsilon)(t) - T_K U_\mrm{sp}(t)}{L^2}^2 \leq C_K (\mathcal O(\varepsilon^{2\mu - 2\sigma}) + \mathcal O(\varepsilon) ) + Err,
	\end{equation}
	where $ T_{\sigma,\mrm{mg}}\in (0,\infty) $ is the time of existence of solutions independent of $ \varepsilon $ and $ K $, and $ Err $ is the truncation error which vanishes uniformly-in-$\varepsilon $ as $ K \rightarrow \infty $. 
\end{theorem}
{The physical rationale for the need to project out the pressure variable in the course of this estimate is as follows: By the non-dimensionalization underlying the full compressible system in \eqref{eq:rf-EEq}, the small parameter $\varepsilon$ is proportional to the Mach number. Then, under the assumption of initial velocities of order unity, acoustic pressure amplitudes will be of order $\mathcal O(\varepsilon)$ for otherwise general initial data, see \eqref{def:perturbation} and, e.g., \cite{Schneider1978,Klainerman1982,Klein1995}. Similarly, internal waves inducing velocities of $\mathcal O(1)$ come with pressure perturbation amplitudes of order $\mathcal O(\varepsilon^{2-\nu})$, see \cite{Klein2010}, while slow, purely advective dynamics implies pressure amplitudes of $\mathcal O(\varepsilon^2)$ according to the classical scaling for incompressible flows. Therefore, when the contributions of the superimposed acoustic, gravity wave, and mean flow modes to the velocity field are comparable (e.g., of order unity), then their contributions to the pressure field have decidedly different amplitude scaling with $\delta p_{\rm aw} \gg \delta p_{\rm gw} \gg \delta p_{\rm mf}$. That is, there are scaling regimes within which the influence of acoustics on the flow velocity and advected scalars is neglible compared to that of gravity waves and mean flow, although the pressure perturbations are still dominated by the acoustic modes. In these regimes, the projected variables $(\mathcal{H}, v,w)$ in the full compressible and pseudo-incompressible solutions are asymptotically close, whereas the pressure fields are not. Our theorem then states that the net effect of the larger acoustic pressure fluctuations rigorously average out at leading order and over the pertinent advective time scale. This generalizes related statements regarding acoustic averaging in the absence of gravity by Klainerman and Majda, \cite{Klainerman1982}.}

\bigskip

To get existence of solutions to \eqref{eq:rf-EEq}, we need uniform-in-$ \varepsilon $ {\it a priori} estimate, namely:
\begin{proposition}\label{thm:uniform-est} 
Let $ 0 < 2 \nu < 1 $ and $ 0 < \varepsilon < 1 $. Suppose that $ ( \widetilde q_\mrm{in}, \widetilde{\mathcal H}_\mrm{in}, v_\mrm{in}, w_\mrm{in} ) $ in \eqref{def:initial-EEq} satisfies
\begin{equation}\label{cnt-thm:uniform-est-1}
\begin{gathered}
	\sum_{{\substack{\alpha,\beta \in \mathbb N, \, \alpha+\beta \leq 3,\\ \partial \in \lbrace \partial_x, \partial_y, \partial_z \rbrace}}} \biggl( \norm{\lbrack \partial^\beta  (\varepsilon^\sigma \dt)^\alpha  \widetilde q\rbrack_{\mrm{in}}, \lbrack \partial^\beta (\varepsilon^\sigma \dt)^\alpha  \widetilde{\mathcal H}\rbrack_\mrm{in}}{\Lnorm{2}}^2\\ 
	+ \norm{\lbrack \partial^\beta (\varepsilon^\sigma \dt)^\alpha  v\rbrack_\mrm{in}, \lbrack \partial^\beta (\varepsilon^\sigma \dt)^\alpha  w\rbrack_\mrm{in} }{\Lnorm{2}}^2 \biggr) 
		\leq \mathcal C_\mrm{in}, 
	\end{gathered}
\end{equation}
for some $ \mathcal C_\mrm{in}  \in (0,\infty) $ and $ \sigma \in (0,\mu\rbrack $, where $ (\lbrack\dt^\alpha \widetilde q\rbrack_\mrm{in}, \lbrack \dt^\alpha \widetilde{\mathcal H}\rbrack_\mrm{in}, \lbrack \dt^\alpha v\rbrack_\mrm{in}, \lbrack\dt^\alpha w\rbrack_\mrm{in}), \, \alpha \in \mathbb N^+ $, are defined inductively after shifting spatial derivatives to temporal derivatives using equations of \eqref{eq:rf-EEq}. Let $ (\widetilde q(s), \widetilde{\mathcal H}(s), v(s), w(s)) $ be the smooth solution to \eqref{eq:rf-EEq} with initial data $ ( \widetilde q_\mrm{in}, \widetilde{\mathcal H}_\mrm{in}, v_\mrm{in}, w_\mrm{in} ) $. Then there exist $ T_\sigma \in (0,\infty) $, depending only on $ \mathcal C_\mrm{in} $, such that
\begin{equation*}%\label{est:uni-total-general}
	\begin{gathered}
	\sup_{0\leq s \leq T_\sigma}\sum_{{\substack{\alpha,\beta \in \mathbb N, \, \alpha+\beta \leq 3,\\ \partial \in \lbrace \partial_x, \partial_y, \partial_z \rbrace }}} \biggl( \norm{(\varepsilon^\sigma \dt)^\alpha \partial^\beta \widetilde q(s), (\varepsilon^\sigma \dt)^\alpha \partial^\beta \widetilde{\mathcal H}(s)}{\Lnorm{2}}^2\\ + \norm{(\varepsilon^\sigma \dt)^\alpha \partial^\beta v(s), (\varepsilon^\sigma \dt)^\alpha \partial^\beta w(s) }{\Lnorm{2}}^2 \biggr) 
		\leq \mathcal C \mathcal C_\mrm{in},
		\end{gathered}
\end{equation*}
with  some constant $ \mathcal C \in (0,\infty) $, independent of $ \varepsilon $. 
\end{proposition}
We would like to mention that, with the {\it a priori} estimate, one can construct solutions locally in time to \eqref{eq:rf-EEq}, and also show the well-posedness, i.e., uniqueness and continuous dependency on initial data. The construction and proof are standard, and we leave the details to readers. 

%\todo{to check with Rupert}
% After rescaling the time at the same order, a uniform-in-$ \varepsilon $ estimate for a soundproof system similar to \eqref{eq:sndprf} was obtained by the authors of \cite[Theorem 2]{DavidLannes2019} under the assumption that the Brunt-V\"ais\"al\"a frequency $ \mathfrak N $ is constant, i.e., \ref{H3}. In the case when $ \mathfrak N $ is not constant, and depending on the $ z $-variable, the existence time is of $ \mathcal O(\varepsilon^\nu) $ as shown in \cite[Theorem 1]{DavidLannes2019} for the soundproof system. Moreover, by performing vertical wave number/modal decomposition, the authors in \cite{DavidLannes2019} formally calculate the wave interaction via a series of PDEs in two space-dimensions, and show that when $ \mathfrak N $ is not constant, different vertical wave modals have strong interaction and dispersive mixing (see Proposition 4 in \cite{DavidLannes2019}). The case when $ \mathfrak N $ is constant is calculated in Proposition 5 of the same paper. Similar modal decomposition can be done for our soundproof model \eqref{eq:sndprf} and compressible model \eqref{eq:rf-EEq}, at least for initial data close to each other, thanks to Theorem \ref{thm:sp-approximation}. Following the scheme of \cite{Klein2010}, one can also perform modal decomposition in the full compressible system \eqref{eq:rf-EEq} and formally investigate the wave interaction in the fashion of \cite{DavidLannes2019}. 

{After rescaling time at the same order, a uniform-in-$ \varepsilon $ estimate for a soundproof system similar to \eqref{eq:sndprf} was obtained by the authors of \cite[Theorem 2]{DavidLannes2019} under the assumption that the Brunt-V\"ais\"al\"a frequency $ \mathfrak N $ is constant, i.e., \ref{H3}. In the case when $ \mathfrak N $ is not constant but depends on the vertical coordinate, $ z $, the existence time is of $ \mathcal O(\varepsilon^\nu) $ as shown in \cite[Theorem 1]{DavidLannes2019} for the soundproof system. Moreover, a vertical mode decomposition based on modes obtained from the eigenfunctions of a Sturm-Liouville equation associated with the background stratification is introduced, and a formal derivation of (partial differential) evolution equations for these modes is provided. It is shown that the modes interact strongly with dispersive mixing when $ \mathfrak N $ is not constant (see Proposition 4 in \cite{DavidLannes2019}). In contrast, when $ \mathfrak N $ \emph{is} constant, the vertical modes decouple (see their Proposition 5).}

{Notably, the vertical mode decomposition in \cite{DavidLannes2019} is not an eigenmode decomposition of the fast linear system describing its internal wave dynamics as developed in section~\ref{subsec:FourierRepresentations} of the present paper (see also \cite{Klein2010}). In fact the eigenmodes of the fast system are sinusoidal in the horizontal direction and satisfy a Sturm-Liouville equation that is parameterized by the horizontal wave number. For non-constant $ \mathfrak N $, the resulting vertical modes are not sinusoidal and their structure depends non-trivially on the horizontal wave number. As a consequence, the projections of the solution onto just the eigenmodes of the hydrostatic background in \cite{DavidLannes2019} will themselves be linear combinations of the eigenmodes of the full system and must reveal dispersive behavior. Moreover, in this case the modes of the background system will also generally be coupled, because their projection onto the eigenmodes of the full system will depend on the time evolving horizontal structure of the solution. }

{The present analysis for the pseudo-incompressible model reduces to that of the incompressible system studied in \cite{DavidLannes2019} for $P^{\varepsilon} = 1$ in \eqref{eq:mid-s-met-notation}. It would be interesting to compare the detailed analytical steps and accessible results when the solution decomposition in terms of a single family of vertical modes as invoked by Desjardins et al.~\cite{DavidLannes2019} is replaced with a decomposition in terms of the full set of eigenmodes of the fast system as worked out here.  As demonstrated in \cite{Klein2010}, that approach could also be transferred to the full compressible system \eqref{eq:rf-EEq} in which case the additional family of (even faster) acoustic eigenmodes and their potential interactions with the internal wave and advective modes will have to be accounted for.}

\bigskip

To prove Theorem \ref{thm:sp-ill-prepared}, we need to understand the distribution of eigenvalues and need to have comparison of eigenvectors, that is:
\begin{proposition}\label{thm:eg-values}
	The eigenvalues of operator $ \mathcal L_a + \varepsilon^{1-\nu} \mathcal L_g $ lie within the neighborhood of radius $ \varepsilon^{1-\nu} $ of the eigenvalues of operator $ \mathcal L_a $. More precisely, let $ i \omega $ be an eigenvalue of $ \mathcal L_a + \varepsilon^{1-\nu} \mathcal L_g $, then there exists $ m \in \lbrace 0,1,2,\cdots \rbrace $, such that
	\begin{equation*}
		\lvert \omega_{\mrm{ac},m}^\pm \rvert^2 \leq \lvert \omega \rvert^2 \leq \lvert \omega_{\mrm{ac},m}^\pm \rvert^2 + \varepsilon^{2-2\nu},
	\end{equation*}
	where $ \lbrace i \omega_{\mrm{ac},m}^\pm \rbrace_{m\in \lbrace 0,1,2,\cdots \rbrace } $ are the eigenvalues of $ \mathcal L_a $. Therefore, the eigenvalues of the linear oscillating operator 
	$$ \dfrac{1}{\varepsilon}\mathcal L_a + \dfrac{1}{\varepsilon^\nu} \mathcal L_g, $$ 
	to system \eqref{eq:rf-EEq}, with $ \mathcal A = \mathcal B = \mathcal C = 1 $, can be classified into three families: {mean flow frequency} $ \lvert \iota_\mrm{mf}\rvert = 0 $; {internal wave frequency} $ \lvert\iota_\mrm{gw}\rvert = \mathcal O(\varepsilon^{-\nu}) $; {perturbed acoustic wave frequency} $ \lvert \iota_\mrm{aw} \rvert = \mathcal O(\varepsilon^{-1}) $.
	
	In addition, with Fourier representations, one can obtain more detailed and sharper comparison on the eigenvalues and eigenvectors, which are presented in Corollary \ref{cor:cmprs-gw} in page \pageref{cor:cmprs-gw}. 
\end{proposition}
We refer readers to the representation of eigenvalue-eigenvector pairs to Proposition \ref{prop:wave-bases}, below.

The rest of this paper is organized as follows. Section \ref{sec:preliminaries} will introduce some notations {that} have been and will be used in this paper, as well as some classic nonlinear and commutator estimates. Section \ref{sec:uniform-est} is devoted to uniform-in-$\varepsilon $ energy estimates of solutions to \eqref{eq:rf-EEq}, and thus proves Proposition \ref{thm:uniform-est}. In section \ref{sec:sound-proof-app}, the rigidity of soundproof approximation is established, which proves Theorem \ref{thm:sp-approximation}. Notice that due to the stratification, we will introduce an intermediate model, i.e., \eqref{eq:mid-s}, to establish the soundproof approximation. The aforementioned linear oscillating system is introduced in section \ref{sec:fast-slow-linear}, where the eigenvalue problem is investigated. Using the Fourier representation, the eigenvalue-eigenvector pairs are identified. Thus Proposition \ref{thm:eg-values} is proved. In section \ref{sec:gw-sp}, we further investigate the internal waves in the soundproof model \eqref{eq:sndprf} and compare them with those in the compressible system \eqref{eq:rf-EEq}.
In section \ref{sec:fast-slow-nonlinear}, we discuss the fast-slow wave interactions of nonlinear system \eqref{eq:rf-EEq}, and establish Theorem \ref{thm:sp-ill-prepared}. %with application of the result from section \ref{sec:fast-slow-linear}.

\section{Preliminaries}\label{sec:preliminaries}

 We assume that we are in $ \mathbb T^3 $ all the time. We use the notation $ \partial \in \lbrace \partial_x,\partial_y,\partial_z \rbrace $ throughout the rest of the paper.  The horizontal gradient, the horizontal divergence, and the horizontal laplacian operators are defined by
 $$ 
 \nablah := \biggl( \begin{array}{c}
 	\partial_x \\ \partial_y
 \end{array}\biggr), \quad \dvh : = \nablah \cdot, \quad \text{and} \quad \deltah := \dvh \nablah,
 $$
 respectively. By adding a subscript $ { }_\mrm{in} $ to any function $ u $, we mean the initial data of $ u $, i.e., $ u\big\vert_{t=0} = u_\mrm{in} $. By $ A \lesssim B $, it means there exists a generic constant $ C \in (0,\infty) $, different from lines to lines, such that $ A \leq C B $. Whenever we would like to emphasize the dependency of the generic constant $ C $ on certain quantities, the depending quantities will be added as subscript, i.e., $ C_{g} $ means a constant depending on $ g $. For any norm $ \norm{}{X} $, we shorten the notation for norms of multiple functions as
 \begin{equation*}
 	\norm{A,B}{X} = \norm{A}{X} + \norm{B}{X}. 
 \end{equation*} 
 
 First, we introduce some nonlinear estimates, which are classic in the literature. 
  \begin{lemma}\label{lm:nonlinear-1} For $ s \in \mathbb N^+ $, 
  	\begin{equation}\label{est:nonlinear-1}
  	\norm{u v}{\Hnorm{s}} \leq K \norm{u}{\Hnorm{\eta}} \norm{v}{\Hnorm{s}} + \norm{u}{\Hnorm{s}} \norm{v}{\Hnorm{\eta}},
  	\end{equation}
  	where
  	\begin{equation}\label{est:nonlinear-1-1}
  		\eta:= \max\lbrace [s/2], 2 \rbrace,
  	\end{equation}
  	and $ K \in (0,\infty) $ depends on $ s $. 
  \end{lemma}
  \begin{proof}
  	The proof is straightforward, after applying Leibniz's formula, H\"older's inequality, and the Sobolev embedding inequality. Details are omitted here.
  \end{proof}
  
  \begin{lemma}\label{lm:nonlinear-2} For $ s > 3/2 $, $ \sigma_1,\sigma_2 \in [0,s], \sigma_1 + \sigma_2 \leq s $, one has
  \begin{equation}\label{est:nonlinear-2}
  	\norm{u v}{\Hnorm{s-\sigma_1-\sigma_2}} \leq K \norm{u}{\Hnorm{s-\sigma_1}} \norm{v}{\Hnorm{s-\sigma_2}},
  \end{equation}
  where $ K \in (0,\infty) $ depends on $ s, \sigma_1, \sigma_2 $. 
  \end{lemma}
\begin{proof}
	We sketch the estimate of $ \norm{\partial^\alpha u \partial^\beta v}{\Lnorm{2}} $, with $ \alpha + \beta \leq s-\sigma_1-\sigma_2 $. After applying H\"older's inequality and the Sobolev embedding inequality, one has
	\begin{equation}\label{est:nonlinear-2-1}
		\begin{aligned}
			& \norm{\partial^\alpha u \partial^\beta v}{\Lnorm{2}} \lesssim \norm{\partial^\alpha u}{\Lnorm{p}} \norm{\partial^\beta v}{\Lnorm{q}} \\
			& \qquad \qquad \lesssim \norm{u}{\Hnorm{m}} \norm{v}{\Hnorm{n}},
		\end{aligned}
	\end{equation}
	with certain
	\begin{equation}\label{est:nonlinear-2-2}
		\begin{gathered}
			\dfrac{1}{2} = \dfrac{1}{p} + \dfrac{1}{q}, \quad p,q \in (2,\infty] \\
			\dfrac{1}{p} - \dfrac{\alpha}{3} \geq \dfrac{1}{2} - \dfrac{m}{3}, \quad \dfrac{1}{q} - \dfrac{\beta}{3} \geq \dfrac{1}{2} - \dfrac{n}{3}.
		\end{gathered}
	\end{equation}
	In order to have a non-empty set of $ (p,q) $ in \eqref{est:nonlinear-2-2}, we require further that
	\begin{equation}\label{est:nonlinear-2-3}
		\alpha \leq m, \quad \beta \leq n, \quad m+n \geq \alpha + \beta + \dfrac{3}{2}.
	\end{equation}
	One can check, with $ m = s-\sigma_1 $ and $ n = s - \sigma_2 $, \eqref{est:nonlinear-2-3} are satisfied with $ s > 3/2 $. Therefore,  \eqref{est:nonlinear-2} follows after taking the sum over $ \alpha,\beta $ of \eqref{est:nonlinear-2-1}.
\end{proof}

%\todo{probably not needed}
%\begin{lm}\label{lm:nonlinear-3} Let $ \mathcal F(\cdot) \in C^\infty(-\delta,\delta) $ and $ \norm{u}{\Lnorm{\infty}} \leq \delta/2 $, $ \delta \in (0,\infty) $. For $ s \geq 1 $, 
%\begin{equation}\label{est:nonlinear-3}
%	\norm{\partial^s \mathcal F(u)}{\Lnorm{2}} \leq C_{\mathcal{F},s}(\norm{u}{\Lnorm{\infty}})\norm{u}{\Hnorm{s}}^s,
%\end{equation}
%for some constant $ C_{\mathcal{F},s} \in (0,\infty) $. 
%%where
%%\begin{equation}\label{est:nonlinear-3-1}
%%	\eta := \max\lbrace s, 2 \rbrace.
%%\end{equation}
%\end{lm}
%\begin{proof}
%	For $ s =1,2 $, \eqref{est:nonlinear-3} can be checked easily. For $ s \geq 3 $, one has, directly,
%	\begin{align*}
%		& \norm{\partial^s \mathcal F(u)}{\Lnorm{2}} \lesssim \sum_{0\leq \alpha \leq s-3}\norm{\partial^{s-1-\alpha}\mathcal F'(u)}{\Lnorm{2}}\norm{\partial^\alpha \nabla u}{\Lnorm{\infty}} \\
%		& \quad + \norm{\partial \mathcal F'(u)}{\Lnorm{\infty}}\norm{\partial^{s-2}\nabla u}{\Lnorm{2}} + \norm{\mathcal F'(u)}{\Lnorm{\infty}} \norm{\partial^{s-1} \nabla u}{\Lnorm{2}} \\
%		& \lesssim \sum_{0\leq \alpha \leq s-3}\norm{\partial^{s-1-\alpha}\mathcal F'(u)}{\Lnorm{2}}\norm{\partial^\alpha \nabla u}{\Lnorm{\infty}} \\
%		& \qquad + \norm{\mathcal F''(u)}{\Lnorm{\infty}} \norm{\partial u}{\Lnorm{\infty}} \norm{\partial^{s-2} \nabla u}{\Lnorm{2}}\\
%		& \qquad + \norm{\mathcal F'(u)}{\Lnorm{\infty}} \norm{\partial^{s-1}\nabla u}{\Lnorm{2}}
%		.
%	\end{align*}
%	Thus \eqref{est:nonlinear-3} follows by performing induction on $ s $. 
%\end{proof}

Next, we will introduce the some functional setups, and commutator estimates. 

Let \begin{equation}\label{def:dd-opt}
	\D{\sigma}{\alpha,\beta}:= \sum_{\partial \in \lbrace \partial_x,\partial_y,\partial_z\rbrace} (\varepsilon^\sigma\dt)^\alpha\partial^\beta. %, \quad \text{with} \quad \partial \in \lbrace \partial_x,\partial_y,\partial_z\rbrace.
	\end{equation} 
Denote by
\begin{equation}\label{def:h-norm}
	\norm{\cdot}{\stHnorm{\alpha,\sigma}{\beta}}:= 
	\sum_{\iota\leq \beta} \norm{\D{\sigma}{\alpha,\iota}(\cdot)}{\Lnorm{2}}.
\end{equation}
The hyperbolic energy is defined as
\begin{equation}\label{def:total-norm}
	\ttene{\sigma}{s}(\cdot):= \sum_{\alpha+\beta \leq s}\norm{\cdot}{\stHnorm{\alpha,\sigma}{\beta}}.
\end{equation}
%In addition, the homogeneous energy is defined as, for $ s \geq 1 $,
%\begin{equation}\label{def:hom-norm}
%	\httene{\sigma}{s}(\cdot):= \sum_{1 \leq \alpha + \beta \leq s} \norm{\D{\sigma}{\alpha,\beta}(\cdot)}{\Lnorm{2}}.
%\end{equation}
%Notice that, for $ s \geq 1 $, 
%\begin{equation}\label{est:cp-norms}
%	\httene{\sigma}{s}(\cdot) + \norm{\cdot}{\Lnorm{2}} \leq \ttene{\sigma}{s}(\cdot) \leq K \bigl(\httene{\sigma}{s}(\cdot) + \norm{\cdot}{\Lnorm{2}}\bigr),
%\end{equation}
%for some constant $ K \in (0,\infty) $, depending only on $ s $.

Now, we are ready to establish some commutator estimates. The following lemma presents the estimate of $ [\D{\sigma}{\alpha,\beta}, f_1 \dt ] $. 
\begin{lemma}\label{lm:cmt-est-1}
For $ \alpha + \beta \geq 3 $, 
	\begin{equation}\label{est:cmt-est-1}
		\norm{[\D{\sigma}{\alpha,\beta}, f_1 \dt ] g_1}{\Lnorm{2}} \leq K \varepsilon^{-\sigma} \ttene{\sigma}{\alpha+\beta}(f_1) \ttene{\sigma}{\alpha+\beta}(g_1),
	\end{equation}
	where $ K\in (0,\infty) $ depends only on $ \alpha,\beta $. 
\end{lemma}
\begin{proof}
It suffices to consider the estimate of
\begin{equation}\label{est:cmt-est-1-1}
	\begin{gathered}
	\norm{\D{\sigma}{\alpha_1,\beta_1} f_1 \D{\sigma}{\alpha_2,\beta_2}\dt g_1}{\Lnorm{2}} = \varepsilon^{-\sigma}\norm{\D{\sigma}{\alpha_1,\beta_1} f_1 \D{\sigma}{\alpha_2+1,\beta_2} g_1}{\Lnorm{2}},\\
	\text{with} \quad 
	 \alpha_1 + \alpha_2 = \alpha, \quad \beta_1 + \beta_2 = \beta, \quad \alpha_1 + \beta_1 \geq 1.
	\end{gathered}
\end{equation}
	Since $ \alpha + \beta - 1 \geq 2 > 3/2 $, applying \eqref{est:nonlinear-2} with 
	\begin{gather*}
		u = \D{\sigma}{\alpha_1,\beta_1} f_1, \quad v = \D{\sigma}{\alpha_2+1,\beta_2} g_1,\\
		s = \alpha+\beta - 1, \quad \sigma_1 = \alpha_1 + \beta_1 - 1, \quad \sigma_2 = \alpha_2 + \beta_2,
	\end{gather*}
	leads to
	\begin{equation}\label{est:cmt-est-1-2}
	\begin{gathered}
		\norm{\D{\sigma}{\alpha_1,\beta_1} f_1 \D{\sigma}{\alpha_2+1,\beta_2} g_1}{\Lnorm{2}} \lesssim \norm{\D{\sigma}{\alpha_1,\beta_1} f_1}{\Hnorm{\alpha_2+\beta_2}} \\
		\times \norm{\D{\sigma}{\alpha_2+1,\beta_2} g_1}{\Hnorm{\alpha_1+\beta_1-1}} %\leq \norm{f_1}{\stHnorm{\alpha_1,\sigma}{\alpha_2 + \beta}}\norm{g_1}{\stHnorm{\alpha_2+1,\sigma}{\alpha_1 + \beta - 1}}\\
		\leq \ttene{\sigma}{\alpha+\beta}(f_1) \ttene{\sigma}{\alpha+\beta}(g_1).
	\end{gathered}
	\end{equation}
	Therefore \eqref{est:cmt-est-1} follows after summing over $ (\alpha_1,\alpha_2,\beta_1,\beta_2) $ in \eqref{est:cmt-est-1-1} and \eqref{est:cmt-est-1-2}.
\end{proof}

Next, we are going to establish the estimate of $ [\D{\sigma}{\alpha,\beta}, f_2 \partial ] $, with $ \partial \in \lbrace \partial_x,\partial_y,\partial_z \rbrace $. 
\begin{lemma}\label{lm:cmt-est-2}
With $ \partial \in \lbrace \partial_x,\partial_y,\partial_z \rbrace $ and $ \alpha+\beta \geq 3 $, 
	\begin{equation}\label{est:cmt-est-2}
		\norm{[\D{\sigma}{\alpha,\beta}, f_2 \partial ] g_2}{\Lnorm{2}} \leq K \ttene{\sigma}{\alpha+\beta}(f_2) \ttene{\sigma}{\alpha+\beta}(g_2),
	\end{equation}
	where $ K\in (0,\infty) $ depends only on $ \alpha,\beta $. 
\end{lemma}
\begin{proof}
	The proof is very much similar to that of Lemma \ref{lm:cmt-est-1}. Therefore we leave it to readers. 
%	It suffices to consider
%	\begin{gather*}
%		\norm{\D{\sigma}{\alpha_1,\beta_1} f_2 \D{\sigma}{\alpha_2,\beta_2}\partial g_2}{\Lnorm{2}} = \norm{\D{\sigma}{\alpha_1,\beta_1} f_2 \D{\sigma}{\alpha_2,\beta_2+1} g_2}{\Lnorm{2}},\\
%		\text{with} \quad \alpha_1 + \alpha_2 = \alpha, \beta_1 + \beta_2 = \beta, \alpha_1+\beta_1 > 0.
%	\end{gather*}
%	Since $ \alpha + \beta - 1 > 3/2 $, applying \eqref{est:nonlinear-2} with
%	\begin{gather*}
%		u = \D{\sigma}{\alpha_1,\beta_1} f_2, \quad v = \D{\sigma}{\alpha_2,\beta_2 + 1} g_2,\\
%		s = \alpha+\beta-1, \quad \sigma_1 = \alpha_1+\beta_1 - 1, \quad \sigma_2 = \alpha_2 + \beta_2,
%	\end{gather*}
%	leads to
%	\begin{gather*}
%		\norm{\D{\sigma}{\alpha_1,\beta_1} f_2 \D{\sigma}{\alpha_2,\beta_2+1} g_2}{\Lnorm{2}} \lesssim \norm{\D{\sigma}{\alpha_1,\beta_1}f_2}{\Hnorm{\alpha_2 + \beta_2}}\\
%		\times \norm{\D{\sigma}{\alpha_2,\beta_2 + 1} g_2}{\Hnorm{\alpha_1+\beta_1 - 1}} \leq \ttene{\sigma}{\alpha+\beta}(f_2) \ttene{\sigma}{\alpha+\beta}(g_2). 
%	\end{gather*}
\end{proof}

In addition, we would like to provide some nonlinear estimates.

\begin{lemma}\label{lm:nonlinear-5}
For $ \iota \geq 2 $,
	\begin{equation}\label{est:nonlinear-5}
		\ttene{\sigma}{\iota}(f g) \leq K \ttene{\sigma}{\iota}(f) \ttene{\sigma}{\iota}(g),
	\end{equation}
	where $ K\in (0,\infty) $ depends only on $ \iota $. 
\end{lemma}
\begin{proof}
	Consider
	\begin{equation}\label{est:nonlinear-5-1}
		\begin{gathered}
		\norm{(\varepsilon^\sigma\dt)^{\alpha_1} \partial^{\beta_1} f \cdot (\varepsilon^\sigma\dt)^{\alpha_2} \partial^{\beta_2} g}{\Lnorm{2}},\\
		\text{with} \quad 2 \leq \alpha_1 + \alpha_2 + \beta_1 + \beta_2 \leq  \iota. 
		\end{gathered}
	\end{equation}
	Applying \eqref{est:nonlinear-2} with 
	\begin{gather*}
		u = (\varepsilon^\sigma\dt)^{\alpha_1} \partial^{\beta_1} f, \quad v =(\varepsilon^\sigma\dt)^{\alpha_2} \partial^{\beta_2} g, \\
		s =\alpha_1 + \alpha_2 + \beta_1 + \beta_2, \quad \sigma_1 = \alpha_1 + \beta_1, \quad \sigma_2 = \alpha_2 + \beta_2,
	\end{gather*}
	leads to
	\begin{equation}\label{est:nonlinear-5-2}
		\begin{gathered}
			\norm{(\varepsilon^\sigma\dt)^{\alpha_1} \partial^{\beta_1} f \cdot (\varepsilon^\sigma\dt)^{\alpha_2} \partial^{\beta_2} g}{\Lnorm{2}} \lesssim \norm{(\varepsilon^\sigma\dt)^{\alpha_1} \partial^{\beta_1} f }{\Hnorm{\alpha_2 + \beta_2}}\\
			\times \norm{(\varepsilon^\sigma\dt)^{\alpha_2} \partial^{\beta_2} g}{\Hnorm{\alpha_1+\beta_1}} \leq \ttene{\sigma}{\iota}(f) \ttene{\sigma}{\iota}(g).
		\end{gathered}
	\end{equation}
	The estimates of the case when $ \alpha_1 + \alpha_2 + \beta_1 + \beta_2 = 0,1 $ is straightforwards and thus is omitted here. Therefore \eqref{est:nonlinear-5} follows after summing over $ ( \alpha_1,\alpha_2,\beta_1,\beta_2) $. 
\end{proof}

\section{Uniform \textit{a priori} estimates}\label{sec:uniform-est}

We are in the place to perform {\it a priori} energy estimates. That is, we will establish the proof of Proposition \ref{thm:uniform-est} in this section. 
In particular, we will focus on the estimates of
\begin{equation*}
	\norm{\D{\sigma}{\alpha,\beta} (\widetilde q, \widetilde{\mathcal H}, v, w)}{\Lnorm{2}}, \quad \text{with} \quad \alpha+\beta 
	 = 3.
\end{equation*}
The case when $ \alpha+\beta = 0,1,2 $ can be calculated in a similar, if not simpler, manner. 

Applying $ \D{\sigma}{\alpha,\beta} $, $ \alpha+\beta = 3 $, to \eqref{eq:rf-EEq} leads to
\begin{equation}\label{eq:dd-EEq}
	\begin{cases}
		\mathcal A \dt \D{\sigma}{\alpha,\beta} \widetilde q + \mathcal A  v \cdot \nablah \D{\sigma}{\alpha,\beta} \widetilde q + \mathcal A w \dz \D{\sigma}{\alpha,\beta} \widetilde q  \\
		\qquad \qquad + \dfrac{1}{\varepsilon}(\dvh \D{\sigma}{\alpha,\beta} v + \dz \D{\sigma}{\alpha,\beta} w) = \mathcal I_1 + \mathcal J_1 , \\
		\mathcal B \dt \D{\sigma}{\alpha,\beta} \widetilde{\mathcal H} + \mathcal B v \cdot \nablah \D{\sigma}{\alpha,\beta} \widetilde{\mathcal H} + \mathcal B w \dz \D{\sigma}{\alpha,\beta} \widetilde{\mathcal H} - \dfrac{1}{\varepsilon^\nu} \D{\sigma}{\alpha,\beta} w = \mathcal I_2 + \mathcal J_2,\\
		\vartheta \dt \D{\sigma}{\alpha,\beta} v + \vartheta v \cdot\nablah \D{\sigma}{\alpha,\beta} v + \vartheta w \dz \D{\sigma}{\alpha,\beta} v + \dfrac{1}{\varepsilon}\nablah \D{\sigma}{\alpha,\beta} \widetilde q = \mathcal I_3, \\
		\vartheta \dt \D{\sigma}{\alpha,\beta} w + \vartheta v \cdot \nablah \D{\sigma}{\alpha,\beta} w + \vartheta w \dz \D{\sigma}{\alpha,\beta} w \\
		\qquad\qquad + \dfrac{1}{\varepsilon} \dz \D{\sigma}{\alpha,\beta}\widetilde q + \dfrac{1}{\varepsilon^\nu} \D{\sigma}{\alpha,\beta} \widetilde{\mathcal H} = \mathcal I_4, 
	\end{cases}
\end{equation}
where
\begin{align*}
\mathcal I_1 := & - \mathcal A [\D{\sigma}{\alpha,\beta},  v\cdot \nablah ] \widetilde q - \mathcal A [\D{\sigma}{\alpha,\beta},  w \dz ] \widetilde q ,\\
\mathcal J_1:= & \mathcal A \mathcal C \D{\sigma}{\alpha,\beta} (G w) + \varepsilon^\mu \mathcal A \D{\sigma}{\alpha,\beta}( \overline{\mathcal H}_0 w) 
- \varpi_0^{-1} \D{\sigma}{\alpha,\beta}\lbrack \widetilde q (\dvh v + \dz w) \rbrack \\
		& \quad  + \varpi_0^{-1} \D{\sigma}{\alpha,\beta}\lbrack (\dvh v + \dz w) ( \mathcal C \int_0^z G(z') \,dz' + \varepsilon^{\mu} \int_0^z \overline{\mathcal H}_0(z') \,dz' )  \rbrack , \\
\mathcal I_2:= & - \mathcal B [\D{\sigma}{\alpha,\beta},v\cdot\nablah] \widetilde{\mathcal H} - \mathcal B [\D{\sigma}{\alpha,\beta},w\dz] \widetilde{\mathcal H},\\
\mathcal J_2:= & \mathcal B \D{\sigma}{\alpha,\beta} (\widetilde{G} \cdot \widetilde{\mathcal H} w), \\
\mathcal I_3:= & - [\D{\sigma}{\alpha,\beta}, \vartheta \dt ]v - [\D{\sigma}{\alpha,\beta},\vartheta v \cdot\nablah ] v - [\D{\sigma}{\alpha,\beta}, \vartheta w \dz] v,\\
\mathcal I_4:= & - [\D{\sigma}{\alpha,\beta}, \vartheta \dt ]w - [\D{\sigma}{\alpha,\beta},\vartheta v \cdot\nablah ] w - [\D{\sigma}{\alpha,\beta}, \vartheta w \dz] w.
\end{align*}
After taking the $ L^2 $-inner product of \eqref{eq:dd-EEq} with $ 2\D{\sigma}{\alpha,\beta} \widetilde q, 2\D{\sigma}{\alpha,\beta} \widetilde{\mathcal H}, 2\D{\sigma}{\alpha,\beta} v, 2\D{\sigma}{\alpha,\beta} w $, respectively, applying integration by parts to the resultant equations, summing the resultant equations, and applying H\"older's inequality and the Sobolev embedding inequalities, one can write down
\begin{equation}\label{est:uni-1}
	\begin{aligned}
		& \dfrac{d}{dt} \biggl( \mathcal A \norm{\D{\sigma}{\alpha,\beta} \widetilde q}{\Lnorm{2}}^2 + \mathcal B \norm{\D{\sigma}{\alpha,\beta}\widetilde{\mathcal H}}{\Lnorm{2}}^2 + \norm{\vartheta^{1/2} \D{\sigma}{\alpha,\beta} v}{\Lnorm{2}}^2\\
		& \qquad\qquad\qquad + \norm{\vartheta^{1/2} \D{\sigma}{\alpha,\beta} w}{\Lnorm{2}}^2 \biggr) \\
%		= \mathcal A \int \dt(\pi^{-1}) \lvert\D{\sigma}{\alpha,\beta} \widetilde q\rvert^2 \idx \\
%		& \quad + \mathcal A \int \bigl( \dvh(\pi^{-1} v) + \dz(\pi^{-1} w) \bigr) \lvert\D{\sigma}{\alpha,\beta} \widetilde q\rvert^2 \idx+ 2 \int (\mathcal I_1 +\mathcal J_1 ) \D{\sigma}{\alpha,\beta} \widetilde q \idx\\
%		& \quad + \mathcal B \int \bigl( \dvh v + \dz w \bigr) \lvert \D{\sigma}{\alpha,\beta} \widetilde{\mathcal H}\rvert^2 \idx + 2 \int (\mathcal I_2 + \mathcal J_2 ) \D{\sigma}{\alpha,\beta}\widetilde{\mathcal H} \idx \\
%		& \quad + \int \bigl(\dt \theta^{-1} + \dvh(\theta^{-1}v) + \dz(\theta^{-1} w) \bigr) \bigl( \lvert \D{\sigma}{\alpha,\beta} v\rvert^2 + \lvert \D{\sigma}{\alpha,\beta} w\rvert^2 \bigr) \idx \\
%		& \quad + 2 \int \bigl( \mathcal I_3 \D{\sigma}{\alpha,\beta} v +  \mathcal I_4 \D{\sigma}{\alpha,\beta} w \bigr) \idx \\
		& \quad \lesssim \bigl( \norm{\dvh v + \dz  w}{\Hnorm{2}} \bigr) (\norm{\D{\sigma}{\alpha,\beta} \widetilde q}{\Lnorm{2}}^2 + 
%		\\
%		& \qquad + \bigl( \norm{\dvh v + \dz w}{\Hnorm{2}} \bigr) 
		\norm{\D{\sigma}{\alpha,\beta} \widetilde{\mathcal H}}{\Lnorm{2}}^2)  \\
		& \qquad + \bigl( \norm{\dt \vartheta}{\Hnorm{2}} + \norm{\dvh(\vartheta v)}{\Hnorm{2}} + \norm{\dz(\vartheta w)}{\Hnorm{2}} \bigr) \\ & \qquad\qquad \times \bigl( \norm{\D{\sigma}{\alpha,\beta} v}{\Lnorm{2}}^2 + \norm{\D{\sigma}{\alpha,\beta} w}{\Lnorm{2}}^2 \bigr) \\
		& \qquad +  \bigl( \norm{\mathcal I_1}{\Lnorm{2}} + \norm{\mathcal  J_1}{\Lnorm{2}} \bigr) \norm{\D{\sigma}{\alpha,\beta}\widetilde q}{\Lnorm{2}}\\
		& \qquad + \bigl( \norm{\mathcal I_2}{\Lnorm{2}} + \norm{\mathcal J_2}{\Lnorm{2}} \bigr) \norm{\D{\sigma}{\alpha,\beta}\widetilde{\mathcal H}}{\Lnorm{2}} \\
		& \qquad + \norm{\mathcal I_3}{\Lnorm{2}} \norm{\D{\sigma}{\alpha,\beta} v}{\Lnorm{2}} +  \norm{\mathcal I_4}{\Lnorm{2}} \norm{\D{\sigma}{\alpha,\beta} w}{\Lnorm{2}}.
	\end{aligned}
\end{equation}

Now we present the estimates of $ \mathcal I_j $'s and $ \mathcal J_j $'s. 
%Indeed, applying \eqref{est:cmt-est-1} and \eqref{est:nonlinear-4} yields that, for $ \varepsilon \norm{\widetilde q}{\Lnorm{\infty}} \ll 1 $ small enough,
%\begin{align*}
%	& \norm{[\D{\sigma}{\alpha,\beta}, \varpi^{-1} \dt] \widetilde q}{\Lnorm{2}} \lesssim \varepsilon^{-\sigma} \httene{\sigma}{3}(\varpi^{-1}) \httene{\sigma}{3}(\widetilde q) \\
%	& \quad \lesssim \varepsilon^{-\sigma} \bigl( (\ttene{\sigma}{3}(\varepsilon \widetilde \varpi) )^4 + \ttene{\sigma}{3}(\varepsilon \widetilde \varpi) \bigr) \ttene{\sigma}{3}(\widetilde q)  \\
%	& \quad \lesssim \varepsilon^{1-\sigma} \bigl( (\ttene{\sigma}{3}(\widetilde q) )^5 + \ttene{\sigma}{3}(\widetilde q) \bigr),
%\end{align*}
%where we have used, and will use, the definitions in \eqref{def:perturbation} and \eqref{def:perturbation-q}, and 
%the fact, following from the norm-property of $ \ttene{\sigma}{s}(\cdot) $, that
%\begin{equation*}
%	\ttene{\sigma}{s}(a f_1 + b f_2) \leq \lverta\lvert \ttene{\sigma}{s}(f_1) + \lvertb\lvert \ttene{\sigma}{s}(f_2), \quad \text{with} \quad a, b \in \mathbb R.
%\end{equation*}
Applying \eqref{est:cmt-est-2} yields that, 
\begin{equation}\label{est:I-1}
\begin{aligned}
	& \norm{\mathcal I_1}{\Lnorm{2}}\lesssim \norm{[\D{\sigma}{\alpha,\beta}, v \cdot \nablah] \widetilde q}{\Lnorm{2}} + \norm{[\D{\sigma}{\alpha,\beta}, w \dz] \widetilde q}{\Lnorm{2}} \\
	& \quad \lesssim  \ttene{\sigma}{3}(v,w) \ttene{\sigma}{3}(\widetilde q).
\end{aligned}
\end{equation}
The estimates of $ \mathcal I_{j} $, $ j \in \lbrace 2,3,4\rbrace $, are similar. In fact, since $ \vartheta  = \mathcal C + \mathcal O(\varepsilon^\mu) $ as in \eqref{def:perturbation-theta}, applying \eqref{est:cmt-est-1}, \eqref{est:cmt-est-2}, and \eqref{est:nonlinear-5}, one will arrive at
\begin{equation}\label{est:I-234}
	\norm{\mathcal I_2,\mathcal I_3,\mathcal I_4}{\Lnorm{2}} \lesssim ( \varepsilon^{\mu - \sigma} + 1 ) \bigl( (\ttene{\sigma}{3} (\widetilde{\mathcal H},v,w) )^{3} + \ttene{\sigma}{3}(\widetilde{\mathcal H},v,w)  \bigr).
\end{equation}
On the other hand, the estimates of $ \mathcal J_{j} $, $ j \in \lbrace 1,2\rbrace $, are straightforward, thanks to \eqref{est:nonlinear-5}, which are
\begin{equation}\label{est:J-12}
	\norm{\mathcal J_1,\mathcal J_2}{\Lnorm{2}} \lesssim (\ttene{\sigma}{3} (\widetilde q,\widetilde{\mathcal H},v,w) )^{2} + \ttene{\sigma}{3}(\widetilde q, \widetilde{\mathcal H},v,w)  .
\end{equation}
%However, the estimate of $ \mathcal J_2 $ will require careful calculation due to the degeneracy of $G $ at $ z \in \pi \mathbb Z $. 
%Indeed,
%\begin{align*}
%	\mathcal J_2 = \mathcal B 
%\end{align*}
%
%
%\pagebreak
The rest terms on the right hand side of \eqref{est:uni-1} can be handled in a similar manner. We record the estimates below: 
\begin{equation}\label{est:rh-rest}
	\begin{gathered}
		 \norm{\dvh v + \dz w}{\Hnorm{2}} + \norm{\dt \vartheta}{\Hnorm{2}} \\
		 + \norm{\dvh(\vartheta v)}{\Hnorm{2}} + \norm{\dz(\vartheta w)}{\Hnorm{2}}\\
		 \lesssim ( \varepsilon^{\mu+\nu - \sigma} +1) \bigl( (\ttene{\sigma}{3}  (\widetilde q,  \widetilde{\mathcal H}, v,w))^{2} + \ttene{\sigma}{3}(\widetilde q,  \widetilde{\mathcal H}, v,w) \bigr).
	\end{gathered}
\end{equation}

Consequently, integrating \eqref{est:uni-1} in the temporal variable yields, 
\begin{equation}\label{est:uni-2}
\begin{aligned}
	& \mathcal A \norm{\D{\sigma}{\alpha,\beta} \widetilde q}{\Lnorm{2}}^2(t) + \mathcal B \norm{\D{\sigma}{\alpha,\beta}\widetilde{\mathcal H}}{\Lnorm{2}}^2(t) + \norm{\vartheta^{1/2} \D{\sigma}{\alpha,\beta} v}{\Lnorm{2}}^2(t)\\
		& \qquad + \norm{\vartheta^{1/2} \D{\sigma}{\alpha,\beta} w}{\Lnorm{2}}^2(t) \leq 
		\mathcal A \norm{\D{\sigma}{\alpha,\beta} \widetilde q_\mrm{in}}{\Lnorm{2}}^2 \\
		&\qquad + \mathcal B \norm{\D{\sigma}{\alpha,\beta}\widetilde{\mathcal H}_\mrm{in}}{\Lnorm{2}}^2
		+ \norm{\vartheta^{1/2}_\mrm{in} \D{\sigma}{\alpha,\beta} v_\mrm{in}}{\Lnorm{2}}^2
		+ \norm{\vartheta^{1/2}_\mrm{in} \D{\sigma}{\alpha,\beta} w_\mrm{in}}{\Lnorm{2}}^2 \\
		& + (\varepsilon^{\mu-\sigma} + 1 )\int_0^t \biggl( \bigl\lbrack\ttene{\sigma}{3}(\widetilde q,\widetilde{\mathcal H}, v, w)(s) \bigr\rbrack^{4} +\bigl\lbrack \ttene{\sigma}{3}(\widetilde q,\widetilde{\mathcal H}, v, w)(s) \bigr\rbrack^{2} \biggr) \,ds.
\end{aligned}
\end{equation}
While we only show \eqref{est:uni-2} with $ \alpha+\beta = 3 $, it holds with $ \alpha + \beta = 0,1,2 $, which can be shown in a similar, if not simpler, way. Therefore, one can conclude from \eqref{est:uni-2} that, 
\begin{equation}\label{est:uni-3}
	\begin{aligned}
	& \sup_{0\leq s \leq t}\lbrack \ttene{\sigma}{3}(\widetilde q,\widetilde{\mathcal H}, v, w)(s)\rbrack ^2 \leq C \lbrack \ttene{\sigma}{3}(\widetilde q_\mrm{in},\widetilde{\mathcal H}_\mrm{in}, v_\mrm{in}, w_\mrm{in}) \rbrack^2 \\
	&\qquad + C ( \varepsilon^{\mu-\sigma} + 1) \int_0^t \biggl( \bigl\lbrack\ttene{\sigma}{3}(\widetilde q,\widetilde{\mathcal H}, v, w)(s) \bigr\rbrack^{4} \\
	& \qquad\qquad\qquad + \bigl\lbrack \ttene{\sigma}{3}(\widetilde q,\widetilde{\mathcal H}, v, w)(s) \bigr\rbrack^2 \biggr) \,ds,
	 \end{aligned}
\end{equation}
for some constant $ C $, independent of $ \varepsilon $. Recall that $ \mu \in (0,1) $. 
Consequently, for $ \sigma \in (0,\mu] $, after applying Gr\"onwall's inequality to \eqref{est:uni-3},
there exists $ T_\sigma \in (0,\infty) $, depending only on $ \ttene{\sigma}{3}(\widetilde q_\mrm{in},\widetilde{\mathcal H}_\mrm{in}, v_\mrm{in}, w_\mrm{in}) $, such that
\begin{equation}\label{est:uni-total-general}
	\sup_{0\leq s \leq T_\sigma}\ttene{\sigma}{3}(\widetilde q,\widetilde{\mathcal H}, v, w)(s) \leq C \ttene{\sigma}{3}(\widetilde q_\mrm{in},\widetilde{\mathcal H}_\mrm{in}, v_\mrm{in}, w_\mrm{in}),
\end{equation}
with  some constant $ C \in (0,\infty) $, independent of $ \varepsilon $. 

In particular, let $\sigma = \mu $, and denote by $ T := T_\mu $. We have shown that
%There exist $ T \in (0,\infty) $ and $ \varepsilon_0 $, depending only on $ \ttene{\mu}{3}(\widetilde q_\mrm{in},\widetilde{\mathcal H}_\mrm{in}, v_\mrm{in}, w_\mrm{in}) $, such that
\begin{equation}\label{est:uni-total}
	\sup_{0\leq s \leq T}\ttene{\mu}{3}(\widetilde q,\widetilde{\mathcal H}, v, w)(s) \leq C \ttene{\mu}{3}(\widetilde q_\mrm{in},\widetilde{\mathcal H}_\mrm{in}, v_\mrm{in}, w_\mrm{in}),
\end{equation}
with  some constant $ C \in (0,\infty) $, independent of $ \varepsilon $. Proposition \ref{thm:uniform-est} follows from \eqref{est:uni-total-general}.

\section{The soundproof approximation}\label{sec:sound-proof-app}

{In this section, we focus on the proof of our first main theorem, i.e., Theorem \ref{thm:sp-approximation}.}
As mentioned in the introduction, the motivation of the soundproof approximation is due to the fact that the acoustic oscillator $ \mathcal L_a $ induces a faster oscillation than the internal wave oscillator $ \mathcal L_g $ in system \eqref{eq:rf-EEq}, which leads to faster averaging of acoustic waves. Our soundproof model \eqref{eq:sndprf} preserves the internal gravity waves while filtering out the acoustic waves. In particular, if initial data do not carry any acoustic waves, solutions driven by \eqref{eq:rf-EEq} and \eqref{eq:sndprf} with the same initial data should produce solutions close to each other. {Proving this statement is the main objective of this section.}

However, to achieve our goal, we will need to introduce an intermediate model in section \ref{subsec:itmd-md}. This is to handle the terms on the right hand side of \eqref{eq:rf-EEq} due to stratification, in contrast to \cite{Klainerman1981,Klainerman1982}. Therefore, the soundproof approximation is done in two steps: approximation by the intermediate model of \eqref{eq:rf-EEq} in sections \ref{subsec:itmd-md} and \ref{subsec:itmd-ap}; approximation by the soundproof model of the intermediate model in sections \ref{subsec:snpf-md} and \ref{subsec:snpf-ap}.

\subsection{The intermediate model}
\label{subsec:itmd-md}

%\todo{Rupert: this is another soundproof model, which captures the low stratification, remarks?}
{Here we analyse the intermediate model already introduced in the introduction in \eqref{eq:mid-s-met-notation}. In terms of the current notation, we utilise the replacements
\begin{equation}\label{eq:ms-model-relations}
\begin{array}{c}
\dss \pi = \frac{p_\mrm{ms}}{{\mathcal C}}\,,
\qquad
\theta = -\frac{\widetilde{\mathcal H}_\mrm{ms}}{\mathcal C}\,,
\qquad
S^{\varepsilon} = -\dfrac{1}{\mathcal{C}}\left(\dfrac{1}{\mathcal{B}} + \varepsilon^{\nu} \widetilde G \widetilde{\mathcal H}_\mrm{ms}\right) \,,
  \\
\dss P^{\varepsilon} = 1 - \varepsilon {\widetilde{P}}^{\varepsilon}(z)\,,
\qquad
\frac{1}{P^{\varepsilon}}\frac{d{\widetilde{P}}^{\varepsilon}}{dz} = \mathcal A (\mathcal C G + \varepsilon^\mu \overline{\mathcal H}_0)\,,
\end{array}
\end{equation}
together with the obvious replacements $(v,w) = (v_\mrm{ms},w_\mrm{ms})$, and obtain }
\begin{equation}\label{eq:mid-s}
	\begin{cases}
		\dvh v_\mrm{ms} + \dz w_\mrm{ms} = \varepsilon \mathcal A (\mathcal C G + \varepsilon^\mu \overline{\mathcal H}_0) w_\mrm{ms}, \\
		\mathcal B \dt \widetilde{\mathcal H}_\mrm{ms} + \mathcal B v_\mrm{ms}\cdot\nablah \widetilde{\mathcal H}_\mrm{ms} + \mathcal B w_\mrm{ms}\dz \widetilde{\mathcal H}_\mrm{ms} - \dfrac{1}{\varepsilon^{\nu}} w_\mrm{ms} = \mathcal B \widetilde G \cdot \widetilde{\mathcal H}_\mrm{ms} {w_\mrm{ms}}, \\
		\mathcal C \dt v_\mrm{ms} + \mathcal C v_\mrm{ms} \cdot \nablah v_\mrm{ms} + \mathcal C w_\mrm{ms}\dz v_\mrm{ms} + \nablah p_\mrm{ms} = 0,\\
		\mathcal C \dt w_\mrm{ms} + \mathcal C v_\mrm{ms} \cdot \nablah w_\mrm{ms} + \mathcal C w_\mrm{ms}\dz w_\mrm{ms} + \dz p_\mrm{ms} + \dfrac{1}{\varepsilon^\nu} \widetilde{\mathcal H}_\mrm{ms} = 0,
	\end{cases}
\end{equation}
where $ p_\mrm{ms} $ is determined by, after calculating $ \dvh ( \phi_\varepsilon  \eqref{eq:mid-s}_{3}  ) + \dz ( \phi_\varepsilon \eqref{eq:mid-s}_{4}) $,
\begin{equation}\label{eq:pressure-ms}
	\begin{gathered}
	- \biggl( \dvh (\phi_\varepsilon \nablah p_\mrm{ms}) + \dz (\phi_\varepsilon \dz p_\mrm{ms} ) \biggr) 
%	 = \mathcal C \biggl( \phi_\varepsilon (\nablah v_\mrm{sp})^\top: \nablah v_\mrm{sp} \\ + \phi_\varepsilon \dz v_\mrm{sp} \cdot \nablah w_\mrm{sp} + \dz(\phi_\varepsilon v_\mrm{sp}) \cdot \nablah w_\mrm{sp} \\ + \dz (\phi_\varepsilon w_\mrm{sp}) \dz w_\mrm{sp} - \dvh ( w_\mrm{sp} \dz \phi_\varepsilon v_\mrm{sp}) \\ - \dz ( w_\mrm{sp}^2 \dz \phi_\varepsilon ) \biggr) 
%	+ \dfrac{1}{\varepsilon^{\nu}} \dz( \phi_\varepsilon \widetilde{\mathcal H}_\mrm{sp}),
= \mathcal C \biggl( \phi_\varepsilon (\nablah v_\mrm{ms})^\top: \nablah v_\mrm{ms} \\ + 2 \phi_\varepsilon \dz v_\mrm{ms} \cdot \nablah w_\mrm{ms} + \phi_\varepsilon ( \dz w_\mrm{ms} )^2 \\ - w_\mrm{ms} \dz \phi_\varepsilon \dvh v_\mrm{ms} - (w_\mrm{ms})^2 \dz^2 \phi_\varepsilon \\ - w_\mrm{ms} \dz \phi_\varepsilon \dz w_\mrm{ms} \biggr) 
	+ \dfrac{1}{\varepsilon^{\nu}} \dz( \phi_\varepsilon \widetilde{\mathcal H}_\mrm{ms}), \\
	\int p_\mrm{ms} \idx = 0,
	\end{gathered}
\end{equation}
with
\begin{equation}\label{def:ms-weight}
	\phi_\varepsilon := \phi_\varepsilon(z) = e^{- \varepsilon \mathcal A  \int_0^z( \mathcal C G(z') + \varepsilon^\mu \overline{\mathcal H}_0(z')) \,dz' }.
\end{equation}
%Notice that in \subeqref{eq:mid-s}{1}, thanks to assumption \ref{H1}, the integral in $ \mathbb T^3$ of the right hand side vanishes. 
Notice that, \subeqref{eq:mid-s}{1} is equivalent to
\begin{equation}\label{id:ms-dv-fr}
	\dvh (\phi_\varepsilon v_\mrm{ms}) + \dz (\phi_\varepsilon w_\mrm{ms} ) = 0.
\end{equation}
We list some estimates of $ p_\mrm{ms} $, induced by the elliptic estimates on \eqref{eq:pressure-ms}:
\begin{gather}
\label{est:pressure-ms-1}
	\norm{p_\mrm{ms}}{\Hnorm{4}} \lesssim \norm{v_\mrm{ms},w_\mrm{ms}}{\Hnorm{3}}^2 + \dfrac{1}{\varepsilon^\nu}\norm{\widetilde{\mathcal H}_\mrm{ms}}{\Hnorm{3}}, \\
	\label{est:pressure-ms-2}
	\begin{gathered}
	\norm{\dt p_\mrm{ms}}{\Hnorm{3}} \lesssim \norm{v_\mrm{ms},w_\mrm{ms}}{\Hnorm{3}}\norm{\dt v_\mrm{ms},\dt w_\mrm{ms}}{\Hnorm{2}} \\ + \dfrac{1}{\varepsilon^\nu} \norm{\dt \widetilde{\mathcal H}_\mrm{ms}}{\Hnorm{2}}.
	\end{gathered} 
%	\\
%	\label{est:pressure-ms-3}
%	\begin{gathered}
%		\norm{\dt^2 p_\mrm{ms}}{\Hnorm{2}} \lesssim \norm{v_\mrm{ms},w_\mrm{ms}}{\Hnorm{3}}\norm{\dt^2 v_\mrm{ms},\dt^2 w_\mrm{ms}}{\Hnorm{1}} \\
%		+ \norm{\dt v_\mrm{ms},\dt w_\mrm{ms}}{\Hnorm{2}}^2 + \dfrac{1}{\varepsilon^\nu} \norm{\dt^2 \widetilde{\mathcal H}_\mrm{ms}}{\Hnorm{1}}.
%	\end{gathered}
\end{gather}
In addition, from \subeqref{eq:mid-s}{2}, \subeqref{eq:mid-s}{3}, and \subeqref{eq:mid-s}{4}, one can establish that
\begin{gather}
\label{est:dt-ms-1}
	\norm{\dt \widetilde{\mathcal H}_\mrm{ms}}{\Hnorm{2}} \lesssim \norm{v_\mrm{ms},w_\mrm{ms}}{\Hnorm{2}} \norm{\widetilde{\mathcal H}_\mrm{ms}}{\Hnorm{3}} + \dfrac{1}{\varepsilon^\nu} \norm{w_\mrm{ms}}{\Hnorm{2}}, \\
	\label{est:dt-ms-2}
	\norm{\dt v_\mrm{ms}}{\Hnorm{2}} \lesssim \norm{v_\mrm{ms},w_\mrm{ms}}{\Hnorm{2}} \norm{v_\mrm{ms}}{\Hnorm{3}} + \norm{p_\mrm{ms}}{\Hnorm{3}}, \\
	\label{est:dt-ms-3}
	\begin{gathered}
	\norm{\dt w_\mrm{ms}}{\Hnorm{2}} \lesssim \norm{v_\mrm{ms},w_\mrm{ms}}{\Hnorm{2}} \norm{w_\mrm{ms}}{\Hnorm{3}} + \norm{p_\mrm{ms}}{\Hnorm{3}} \\ + \dfrac{1}{\varepsilon^\nu} \norm{\widetilde{\mathcal H}_\mrm{ms}}{\Hnorm{2}}.
	\end{gathered}
\end{gather}
We point out here, the terms of $ \mathcal O(\frac{1}{\varepsilon^\nu}) $, above, although singular, will be used later together with multiplier $ \varepsilon $ or $ \varepsilon^\mu $ (for instance, see $ \mathfrak I_3$ of \eqref{est:dff-ms-1}, below), which corresponds to the error $ \mathcal O(\varepsilon^{\mu - \nu}) $ in Theorem \ref{thm:sp-approximation}.

We claim that for any initial data $ (\widetilde{\mathcal H}_\mrm{ms,in}, v_\mrm{ms,in}, w_\mrm{ms,in} ) \in H^3 $, satisfying the pseudo-incompressible condition \subeqref{eq:mid-s}{1}, there is $ T_\mrm{ms} \in (0,\infty) $, depending only on $ \norm{\widetilde{\mathcal H}_\mrm{ms,in}, v_\mrm{ms,in}, w_\mrm{ms,in}}{\Hnorm{3}} $, such that
\begin{equation}\label{est:uni-total-ms}
	\sup_{0\leq s\leq T_\mrm{ms}}\norm{\widetilde{\mathcal H}_\mrm{ms}, v_\mrm{ms}, w_\mrm{ms}}{\Hnorm{3}}(s) \leq C \norm{\widetilde{\mathcal H}_\mrm{ms,in}, v_\mrm{ms,in}, w_\mrm{ms,in}}{\Hnorm{3}}
\end{equation}
where $ C \in (0,\infty) $ is independent of $ \varepsilon $.  The proof of \eqref{est:uni-total-ms} follows from standard energy estimates. In fact, applying $ \partial^j $, $ j = 0,1,2,3 $, to \eqref{est:uni-total-ms}, after taking the $ L^2 $-inner product of the resultant equations with $$ 2 \partial^j p_\mrm{ms}, 2\partial^j  \widetilde{\mathcal H}_\mrm{ms}, 2 \partial^j v_\mrm{ms}, 2 \partial^j w_\mrm{ms}, $$ respectively, one can conclude that the summation of the resultant estimates is, thanks to \eqref{est:pressure-ms-1}, 
\begin{equation}\label{est:ms-001}
	\begin{aligned}
		& \dfrac{d}{dt} \biggl( \mathcal B \norm{\widetilde{\mathcal H}_\mrm{ms}}{\Hnorm{3}}^2 + \mathcal C \norm{v_\mrm{ms}, w_\mrm{ms}}{\Hnorm{3}}^2 \biggr) \leq C \norm{\widetilde{\mathcal H}_\mrm{ms},v_\mrm{ms},w_\mrm{ms}}{\Hnorm{3}}^3 \\
		& \qquad\qquad\qquad + 2 \varepsilon \mathcal A \sum_{j = 0}^3 \int \partial^j \bigl( (\mathcal C G + \varepsilon^\mu \overline{\mathcal H}_0) w_\mrm{ms} \bigr) \partial^j p_\mrm{ms} \\
		& \qquad \leq C \norm{\widetilde{\mathcal H}_\mrm{ms},v_\mrm{ms},w_\mrm{ms}}{\Hnorm{3}}^3 + C \norm{\widetilde{\mathcal H}_\mrm{ms},w_\mrm{ms}}{\Hnorm{3}}^2,
	\end{aligned}
\end{equation}
where we have used the fact that $ 0 < \varepsilon < 1 $ and $ 0 < \nu < 1 $. We would like to point out that, while we have omitted the details in \eqref{est:ms-001}, the quadratic terms in \eqref{eq:mid-s} are handled in the same manner as in section \ref{sec:uniform-est}. Namely, we use the following commutator estimate: for $ \beta \geq 3 $
\begin{equation}\label{est:cmt-est-3}
	\norm{[\partial^\beta,f_3\partial] g_3}{\Lnorm{2}} \leq K \norm{f_3}{\Hnorm{\beta}} \norm{g_3}{\Hnorm{\beta}},
\end{equation}
where $ K \in (0,\infty) $. The proof of \eqref{est:cmt-est-3} is similar to that of \eqref{est:cmt-est-1}, and thus is omitted here.
In addition, after the first inequality of \eqref{est:ms-001}, the coefficient $ \varepsilon $ in the second term guarantees that even though $ p_\mrm{ms} = \mathcal O( \frac{1}{\varepsilon^\nu})$ according to \eqref{est:pressure-ms-1}, there is no singular coefficient in the estimates.
Thus  \eqref{est:uni-total-ms} follows after applying Gr\"onwall's inequality. 

\subsection{Intermediate approximation}\label{subsec:itmd-ap}

We would like to compare the solutions to \eqref{eq:rf-EEq} and \eqref{eq:mid-s}. Denote by
\begin{equation}\label{def:ms-dff-vrb}
	\begin{aligned}
		\widetilde q_{\mrm{ms},\delta} := & \widetilde q - \varepsilon p_\mrm{ms},\\
		\widetilde{\mathcal H}_{\mrm{ms},\delta} := & \widetilde{\mathcal H} - \widetilde{\mathcal H}_\mrm{ms}, \\
		v_{\mrm{ms},\delta} := & v - v_\mrm{ms}, \\
		w_{\mrm{ms},\delta} := & w - w_\mrm{ms}.
	\end{aligned}
\end{equation}
Also, we use $ \mathcal K_1 $ to represent the total bound of solutions to \eqref{eq:rf-EEq} and \eqref{eq:mid-s}, i.e., for any fixed $ \sigma \in (0,\mu] $,
\begin{equation}\label{est:uni-total-ee+ms}
	\sup_{0\leq s \leq T}\mathcal E_{\sigma, 3}(\widetilde q, \widetilde{\mathcal H}, v,w)(s) + \sup_{0\leq s \leq T_\mrm{ms}} \norm{\widetilde{\mathcal H}_\mrm{ms}, v_\mrm{ms}, w_\mrm{ms}}{\Hnorm{3}}(s) \leq \mathcal K_1,
\end{equation}
which are obtained in \eqref{est:uni-total-general} and \eqref{est:uni-total-ms}.

In this section, we prove
\begin{equation}\label{est:uni-total-dff-1}
	\begin{aligned}
	& \sup_{0\leq s \leq T_{\mrm{ms},\delta}}\norm{\widetilde q_{\mrm{ms},\delta },\widetilde{\mathcal H}_{\mrm{ms},\delta}, v_{\mrm{ms},\delta }, w_{\mrm{ms},\delta}}{\Lnorm{2}}(s) \\
	 & \quad \leq C_{\mathcal K_1} \bigl(   \norm{\widetilde q_{\mrm{ms},\delta,\mrm{in}},\widetilde{\mathcal H}_{\mrm{ms},\delta,\mrm{in}}, v_{\mrm{ms},\delta,\mrm{in} }, w_{\mrm{ms},\delta,\mrm{in}}}{\Lnorm{2}} + (\varepsilon +  \varepsilon^{\max\lbrace \mu-\nu, \mu -\sigma \rbrace} )  \bigr).
	 \end{aligned}
\end{equation}

\subsubsection*{{Regime 1: $\mu-\nu \geq \mu - \sigma$}}

We first, by multiplying the first equation in \eqref{eq:rf-EEq} with $ \varpi_0 / \varpi $, recalling $ \varpi $ as given in \eqref{def:perturbation} and \eqref{def:perturbation-q} (i.e., reversing the reformulation of the $ \widetilde q $ equation from \eqref{eq:EEq-2-1} to \eqref{eq:EEq-3}), system \eqref{eq:rf-EEq} can be written as
\begin{equation*}\label{eq:rf-EEq-af}\tag{\ref{eq:rf-EEq}'}
	\begin{cases}
	\mathcal A \varpi_0 \varpi^{-1} \dt \widetilde q + \mathcal A \varpi_0\varpi^{-1}v \cdot \nablah \widetilde q + \mathcal A\varpi_0 \varpi^{-1}w \dz \widetilde q  \\
		\qquad \qquad + \dfrac{1}{\varepsilon} (\dvh v + \dz w) = \mathcal A\varpi_0 \mathcal C G \varpi^{-1} w + \varepsilon^{\mu} \mathcal A\varpi_0 \overline{\mathcal H}_0 \varpi^{-1} w, \\
		\mathcal B \dt \widetilde{\mathcal H} + \mathcal B v \cdot \nablah \widetilde{\mathcal H} + \mathcal B w \dz \widetilde{\mathcal H} - \dfrac{1}{\varepsilon^{\nu}} w = \mathcal B  \widetilde G \cdot \widetilde{\mathcal H} w ,\\
		\vartheta  \dt v + \vartheta v\cdot \nablah v + \vartheta w \dz v + \dfrac{1}{\varepsilon} \nablah \widetilde q = 0, \\
		\vartheta \dt w + \vartheta v \cdot \nablah w + \vartheta w \dz w + \dfrac{1}{\varepsilon} \dz \widetilde q + \dfrac{1}{\varepsilon^{\nu}} \widetilde{\mathcal H} = 0,
	\end{cases}
\end{equation*}

After comparing \eqref{eq:rf-EEq-af} and \eqref{eq:mid-s}, one can derive that
$ (\widetilde q_{\mrm{ms},\delta}, \widetilde{\mathcal H}_{\mrm{ms},\delta}, v_{\mrm{ms},\delta}, w_{\mrm{ms},\delta}) $ satisfies
\begin{equation}\label{eq:ms-perturb}
	\begin{cases}
		\mathcal A \varpi_0\varpi^{-1} \dt \widetilde q_{\mrm{ms},\delta} + \mathcal A\varpi_0 \varpi^{-1} v \cdot \nablah \widetilde q_{\mrm{ms},\delta} + \mathcal A \varpi_0\varpi^{-1} w \dz \widetilde q_{\mrm{ms},\delta} \\
		\qquad\qquad + \dfrac{1}{\varepsilon} (\dvh v_{\mrm{ms},\delta} + \dz w_{\mrm{ms},\delta})\\
		\qquad = - \varepsilon \mathcal A \varpi_0\bigl( \varpi^{-1} \dt p_\mrm{ms} + \varpi^{-1} v \cdot \nablah p_\mrm{ms} + \varpi^{-1} w \dz p_\mrm{ms} \bigr) \\
		\qquad\qquad + \mathcal A\varpi_0 \mathcal C G (\varpi^{-1}-\varpi_0^{-1}) w + \mathcal A  \mathcal C G w_{\mrm{ms},\delta} + \varepsilon^{\mu} \mathcal A\varpi_0 \overline{\mathcal H}_0 (\varpi^{-1}-\varpi^{-1}_0) w \\
		\qquad\qquad + \varepsilon^{\mu} \mathcal A \overline{\mathcal H}_0  w_{\mrm{ms},\delta}, \\
	%%%%
		\mathcal B \dt \widetilde{\mathcal H}_{\mrm{ms},\delta} + \mathcal B v \cdot \nablah \widetilde{\mathcal H}_{\mrm{ms},\delta} + \mathcal B w \dz \widetilde{\mathcal H}_{\mrm{ms},\delta}- \dfrac{1}{\varepsilon^\nu} w_{\mrm{ms},\delta} \\
		\qquad = - \mathcal B \bigl( v_{\mrm{ms},\delta} \cdot\nablah \widetilde{\mathcal H}_\mrm{ms} + w_{\mrm{ms},\delta} \dz \widetilde{\mathcal H}_\mrm{ms} \bigr) \\
		\qquad\qquad +  \mathcal B \widetilde G \cdot \bigl( \widetilde{\mathcal H} w_{\mrm{ms},\delta} + \widetilde{\mathcal H}_{\mrm{ms},\delta} w_\mrm{ms} \bigr), \\
		%%%%
		\vartheta \dt v_{\mrm{ms},\delta} + \vartheta v \cdot \nablah v_{\mrm{ms},\delta} + \vartheta w \dz v_{\mrm{ms},\delta} + \dfrac{1}{\varepsilon} \nablah \widetilde q_{\mrm{ms},\delta}\\
		\qquad = \mathcal C^{-1} \bigl( \varepsilon^{\mu} \widetilde G \overline{\mathcal H}_0 + \varepsilon^{\mu+\nu} \widetilde G \widetilde{\mathcal H} \bigr) \nablah p_\mrm{ms} \\
		\qquad\qquad - \vartheta v_{\mrm{ms},\delta} \cdot \nablah v_\mrm{ms} - \vartheta w_{\mrm{ms},\delta} \dz v_\mrm{ms}, \\
		%%%%%
		\vartheta\dt w_{\mrm{ms},\delta} + \vartheta v \cdot \nablah w_{\mrm{ms},\delta} + \vartheta w \dz w_{\mrm{ms},\delta} + \dfrac{1}{\varepsilon} \dz \widetilde q_{\mrm{ms},\delta} + \dfrac{1}{\varepsilon^\nu} \widetilde{\mathcal H}_{\mrm{ms},\delta}\\
		\qquad = \mathcal C^{-1} \bigl( \varepsilon^{\mu} \widetilde G \overline{\mathcal H}_0 + \varepsilon^{\mu+\nu} \widetilde G \widetilde{\mathcal H} \bigr) \bigl( \dz p_\mrm{ms} + \dfrac{1}{\varepsilon^{\nu}} \widetilde{\mathcal H}_\mrm{ms} \bigr) \\
		\qquad\qquad - \vartheta v_{\mrm{ms},\delta} \cdot \nablah w_\mrm{ms} - \vartheta w_{\mrm{ms},\delta} \dz w_\mrm{ms}.
	\end{cases}
\end{equation}
Now, we consider the $ L^2 $-inner product of equations in \eqref{eq:ms-perturb} with $$ 2 \widetilde q_{\mrm{ms},\delta}, 2 \widetilde{\mathcal H}_{\mrm{ms},\delta}, 2 v_{\mrm{ms},\delta}, 2 w_{\mrm{ms},\delta}, $$ respectively. After applying integration by parts and summing up the resultant equations, one can write down that
\begin{equation}\label{est:dff-ms-1}
\begin{aligned}
	& \dfrac{d}{dt} \biggl( \mathcal A \varpi_0 \norm{\varpi^{-1/2} \widetilde q_{\mrm{ms},\delta}}{\Lnorm{2}}^2 + \mathcal B \norm{\widetilde{\mathcal H}_{\mrm{ms},\delta}}{\Lnorm{2}}^2 + \norm{\vartheta^{1/2} v_{\mrm{ms},\delta}}{\Lnorm{2}}^2 \\
	& \qquad + \norm{\vartheta^{1/2} w_{\mrm{ms},\delta}}{\Lnorm{2}}^2 \biggr) = \sum_{j=1}^{4}\mathfrak I_j,
\end{aligned}
\end{equation}
where
\begin{align*}
	& \mathfrak I_1:= \int \biggl( \mathcal A \varpi_0 \dt (\varpi^{-1}) \lvert \widetilde q_{\mrm{ms},\delta}\rvert^2 + \dt  \vartheta  \lvert v_{\mrm{ms},\delta}\rvert^2 + \dt \vartheta \lvert w_{\mrm{ms},\delta} \rvert^2 \biggr) \idx, \\
	& \mathfrak I_2 := \int \biggl( \mathcal A \varpi_0 \bigl(\dvh (\varpi^{-1} v) + \dz (\varpi^{-1} w) \bigr) \lvert \widetilde q_{\mrm{ms},\delta}\rvert^2 + \mathcal B \bigl( \dvh v + \dz w \bigr) \lvert \widetilde{\mathcal H}_{\mrm{ms},\delta} \rvert^2\\
	& \qquad\qquad + \bigl( \dvh (\vartheta v) + \dz (\vartheta w) \bigr) \bigl( \lvert v_{\mrm{ms},\delta}\rvert^2 + \lvert w_{\mrm{ms},\delta} \rvert^2 \bigr) \biggr) \idx,\\
	& \mathfrak I_3 := - 2 \varepsilon \mathcal A \varpi_0 \int \bigl( \varpi^{-1} \dt p_\mrm{ms} + \varpi^{-1} v \cdot \nablah p_\mrm{ms} + \varpi^{-1} w \dz p_\mrm{ms} \bigr) \widetilde q_{\mrm{ms},\delta} \idx \\
	& \qquad + 2 \varepsilon^{\mu} \mathcal C^{-1}  \int \bigl( \widetilde G  \overline{\mathcal H}_0 + \varepsilon^\nu \widetilde G \widetilde{\mathcal H} \bigr) \bigl( \nablah p_\mrm{ms} \cdot v_{\mrm{ms},\delta} + \dz p_\mrm{ms} w_{\mrm{ms},\delta} \\
	&\qquad\qquad + \dfrac{1}{\varepsilon^\nu} \widetilde{\mathcal H}_\mrm{ms} w_{\mrm{ms},\delta} \bigr) \idx, \\
	& \mathfrak I_4: = 2 \mathcal A \varpi_0 \int \biggl( \mathcal C G (\varpi^{-1} - \varpi_0^{-1}) w \widetilde q_{\mrm{ms},\delta} + \varepsilon^\mu \overline{\mathcal H}_0 (\varpi^{-1}-\varpi_0^{-1}) w \widetilde q_{\mrm{ms},\delta} \\
	& \qquad\qquad + \mathcal C G \varpi_{0}^{-1} w_{\mrm{ms},\delta} \widetilde q_{\mrm{ms},\delta} + \varepsilon^\mu \overline{\mathcal H}_0 \varpi^{-1}_0 w_{\mrm{ms},\delta} \widetilde q_{\mrm{ms},\delta} \biggr) \idx \\
	& \qquad + 2 \mathcal B \int \biggl( - ( v_{\mrm{ms},\delta} \cdot\nablah \widetilde{\mathcal H}_\mrm{ms} \widetilde{\mathcal H}_{\mrm{ms},\delta} + w_{\mrm{ms},\delta} \dz \widetilde{\mathcal H}_\mrm{ms} \widetilde{\mathcal H}_{\mrm{ms},\delta} )\\
	& \qquad\qquad + \widetilde G \cdot ( \widetilde{\mathcal H} w_{\mrm{ms},\delta} \widetilde{\mathcal H}_{\mrm{ms},\delta} + \widetilde{\mathcal H}_{\mrm{ms},\delta} w_\mrm{ms} \widetilde{\mathcal H}_{\mrm{ms},\delta} ) \idx \\
	& \qquad - 2 \int \vartheta \biggl( v_{\mrm{ms},\delta} \cdot \nablah v_\mrm{ms} \cdot v_{\mrm{ms},\delta} + w_{\mrm{ms},\delta} \dz v_\mrm{ms} \cdot v_{\mrm{ms},\delta}\\
	& \qquad\qquad + v_{\mrm{ms},\delta} \cdot \nablah w_\mrm{ms} w_{\mrm{ms},\delta} + w_{\mrm{ms},\delta} \dz w_\mrm{ms} w_{\mrm{ms},\delta} \biggr)\idx. 
\end{align*}
Owing to \eqref{def:perturbation}, \eqref{def:perturbation-q}, and \eqref{def:perturbation-theta}, 
$$
\norm{\dt(\varpi^{-1}),\dt \vartheta }{\Lnorm{\infty}} \leq  \varepsilon C_{\mathcal K_1} \norm{\dt \widetilde q}{\Hnorm{2}} + \varepsilon^{\mu+\nu} C_{\mathcal K_1} \norm{\dt \widetilde{\mathcal H}}{\Hnorm{2}}.
$$
Therefore, 
$$
\mathfrak I_1 \leq C_{\mathcal K_1} (\varepsilon^{1-\sigma}  + \varepsilon^{\mu+\nu - \sigma}) \norm{\widetilde q_{\mrm{ms},\delta},v_{\mrm{ms},\delta}, w_{\mrm{ms},\delta}}{\Lnorm{2}}^2.
$$
Similarly,
$$
\mathfrak I_2 \leq C_{\mathcal K_1} \norm{\widetilde q_{\mrm{ms},\delta}, \widetilde{\mathcal H}_{\mrm{ms},\delta}, v_{\mrm{ms},\delta}, w_{\mrm{ms},\delta}}{\Lnorm{2}}^2,
$$
and
$$
\mathfrak I_4 \leq C_{\mathcal K_1} \bigl( \norm{\widetilde q_{\mrm{ms},\delta}, \widetilde{\mathcal H}_{\mrm{ms},\delta}, v_{\mrm{ms},\delta}, w_{\mrm{ms},\delta}}{\Lnorm{2}}^2 + \varepsilon \norm{\widetilde q_{\mrm{ms},\delta}}{\Lnorm{2}} \bigr). 
$$
To estimate $ \mathfrak I_3 $, owing to \eqref{est:pressure-ms-1}, \eqref{est:pressure-ms-2}, \eqref{est:dt-ms-1}, \eqref{est:dt-ms-2}, and \eqref{est:dt-ms-3}, after a tedious but straightforward calculations, one can conclude that, since $ \mu + 2 \nu = 1 $,
\begin{equation}\label{est:I3-ms}
\begin{gathered}
	\mathfrak I_3 \leq C_{\mathcal K_1}\bigl(\varepsilon \norm{\dt p_\mrm{ms}}{\Hnorm{3}} + (\varepsilon + \varepsilon^\mu) \norm{p_\mrm{ms}}{\Hnorm{4}} \bigr) \norm{\widetilde q_{\mrm{ms},\delta}, v_{\mrm{ms},\delta}, w_{\mrm{ms},\delta}}{\Lnorm{2}} \\
	+ \varepsilon^{\mu - \nu}C_{\mathcal K_1} \norm{w_{\mrm{ms},\delta}}{\Lnorm{2}} \leq C_{\mathcal K_1}( \varepsilon + \varepsilon^{\mu -\nu} ) \norm{\widetilde q_{\mrm{ms},\delta}, v_{\mrm{ms},\delta}, w_{\mrm{ms},\delta}}{\Lnorm{2}}.
\end{gathered}
\end{equation}
Therefore, one can conclude from \eqref{est:dff-ms-1} that, provided $ \varepsilon \ll 1 $ small enough, for any $ t \in (0,\min\lbrace T, T_\mrm{ms} \rbrace ]$, since $ \sigma \leq \mu < 1 $, 
\begin{align*}
	& \dfrac{d}{dt} \biggl( \mathcal A \varpi_0 \norm{\varpi^{-1/2} \widetilde q_{\mrm{ms},\delta}}{\Lnorm{2}}^2 + \mathcal B \norm{\widetilde{\mathcal H}_{\mrm{ms},\delta}}{\Lnorm{2}}^2 + \norm{\vartheta^{1/2} v_{\mrm{ms},\delta}}{\Lnorm{2}}^2 \\
	& \qquad + \norm{\vartheta^{1/2} w_{\mrm{ms},\delta}}{\Lnorm{2}}^2 \biggr) \leq C_{\mathcal K_1} \norm{\widetilde q_{\mrm{ms},\delta}, \widetilde{\mathcal H}_{\mrm{ms},\delta}, v_{\mrm{ms},\delta}, w_{\mrm{ms},\delta}}{\Lnorm{2}}^2\\
	& \qquad + C_{\mathcal K_1}( \varepsilon + \varepsilon^{\mu -\nu} ) \norm{\widetilde q_{\mrm{ms},\delta}, v_{\mrm{ms},\delta}, w_{\mrm{ms},\delta}}{\Lnorm{2}}.
\end{align*}
%\begin{equation*}
%	\begin{align*}
%	& (1-  C_{\mathcal K_1} t )\sup_{0\leq s \leq t}\norm{\widetilde q_{\mrm{ms},\delta },\widetilde{\mathcal H}_{\mrm{ms},\delta}, v_{\mrm{ms},\delta }, w_{\mrm{ms},\delta}}{\Lnorm{2}}(s) \\
%	 & \quad \leq C_{\mathcal K_1} \bigl(   \norm{\widetilde q_{\mrm{ms},\delta,\mrm{in}},\widetilde{\mathcal H}_{\mrm{ms},\delta,\mrm{in}}, v_{\mrm{ms},\delta,\mrm{in} }, w_{\mrm{ms},\delta,\mrm{in}}}{\Lnorm{2}} + (\varepsilon + \varepsilon^{\mu-\nu} ) t   \bigr).
%	 \end{align*}
%\end{equation*}
Therefore, after applying Gr\"onwall's inequality, there exists $ T_{\mrm{ms},\delta} \in (0, \min\lbrace T, T_\mrm{ms} \rbrace ] $, depending only on $$ \norm{\widetilde q_{\mrm{ms},\delta,\mrm{in}},\widetilde{\mathcal H}_{\mrm{ms},\delta,\mrm{in}}, v_{\mrm{ms},\delta,\mrm{in} }, w_{\mrm{ms},\delta,\mrm{in}}}{\Lnorm{2}}, $$
%independent of $ \varepsilon $, 
such that
\begin{equation}\label{est:uni-total-dff-1-1}
	\begin{aligned}
	& \sup_{0\leq s \leq T_{\mrm{ms},\delta}}\norm{\widetilde q_{\mrm{ms},\delta },\widetilde{\mathcal H}_{\mrm{ms},\delta}, v_{\mrm{ms},\delta }, w_{\mrm{ms},\delta}}{\Lnorm{2}}(s) \\
	 & \quad \leq C_{\mathcal K_1} \bigl(   \norm{\widetilde q_{\mrm{ms},\delta,\mrm{in}},\widetilde{\mathcal H}_{\mrm{ms},\delta,\mrm{in}}, v_{\mrm{ms},\delta,\mrm{in} }, w_{\mrm{ms},\delta,\mrm{in}}}{\Lnorm{2}} + (\varepsilon + \varepsilon^{\mu-\nu} )  \bigr).
	 \end{aligned}
\end{equation}
Here $ C_{\mathcal K_1} \in (0,\infty) $ is a constant depending only on $ \mathcal K_1 $ given in \eqref{est:uni-total-ee+ms}.

\subsubsection*{{Regime 2: $\mu-\nu < \mu - \sigma$}}

Recalling that $ \vartheta  = \mathcal C + \mathcal O( \varepsilon^\mu )  $ as in \eqref{def:perturbation-theta}, instead of \eqref{eq:ms-perturb}, one can write down the following system by rewriting the $ v $ and $ w $ components:
\begin{equation}\label{eq:ms-perturb-2}
	\begin{cases}
%		\mathcal A \varpi_0\varpi^{-1} \dt \widetilde q_{\mrm{ms},\delta} + \mathcal A\varpi_0 \varpi^{-1} v \cdot \nablah \widetilde q_{\mrm{ms},\delta} + \mathcal A \varpi_0\varpi^{-1} w \dz \widetilde q_{\mrm{ms},\delta} \\
%		\qquad\qquad + \dfrac{1}{\varepsilon} (\dvh v_{\mrm{ms},\delta} + \dz w_{\mrm{ms},\delta})\\
%		\qquad = - \varepsilon \mathcal A \varpi_0\bigl( \varpi^{-1} \dt p_\mrm{ms} + \varpi^{-1} v \cdot \nablah p_\mrm{ms} + \varpi^{-1} w \dz p_\mrm{ms} \bigr) \\
%		\qquad\qquad + \mathcal A\varpi_0 \mathcal C G (\varpi^{-1}-\varpi_0^{-1}) w + \mathcal A  \mathcal C G w_{\mrm{ms},\delta} + \varepsilon^{\mu} \mathcal A\varpi_0 \overline{\mathcal H}_0 (\varpi^{-1}-\varpi^{-1}_0) w \\
%		\qquad\qquad + \varepsilon^{\mu} \mathcal A \overline{\mathcal H}_0  w_{\mrm{ms},\delta}, \\
	%%%%
%		\mathcal B \dt \widetilde{\mathcal H}_{\mrm{ms},\delta} + \mathcal B v \cdot \nablah \widetilde{\mathcal H}_{\mrm{ms},\delta} + \mathcal B w \dz \widetilde{\mathcal H}_{\mrm{ms},\delta}- \dfrac{1}{\varepsilon^\nu} w_{\mrm{ms},\delta} \\
%		\qquad = - \mathcal B \bigl( v_{\mrm{ms},\delta} \cdot\nablah \widetilde{\mathcal H}_\mrm{ms} + w_{\mrm{ms},\delta} \dz \widetilde{\mathcal H}_\mrm{ms} \bigr) \\
%		\qquad\qquad +  \mathcal B \widetilde G \cdot \bigl( \widetilde{\mathcal H} w_{\mrm{ms},\delta} + \widetilde{\mathcal H}_{\mrm{ms},\delta} w_\mrm{ms} \bigr), \\
		\text{The first and second equations as in \eqref{eq:ms-perturb}}\\
		%%%%
		\mathcal C \dt v_{\mrm{ms},\delta} + \mathcal C v \cdot \nablah v_{\mrm{ms},\delta} + \mathcal C w \dz v_{\mrm{ms},\delta} + \dfrac{1}{\varepsilon} \nablah \widetilde q_{\mrm{ms},\delta}\\
		\qquad = \bigl( \varepsilon^{\mu} \widetilde G \overline{\mathcal H}_0 + \varepsilon^{\mu+\nu} \widetilde G \widetilde{\mathcal H} \bigr) (\dt v + v\cdot \nablah v + w \dz v ) \\
		\qquad\qquad - \mathcal C v_{\mrm{ms},\delta} \cdot \nablah v_\mrm{ms} - \mathcal C w_{\mrm{ms},\delta} \dz v_\mrm{ms}, \\
		%%%%%
		\mathcal C \dt w_{\mrm{ms},\delta} + \mathcal C v \cdot \nablah w_{\mrm{ms},\delta} + \mathcal C w \dz w_{\mrm{ms},\delta} + \dfrac{1}{\varepsilon} \dz \widetilde q_{\mrm{ms},\delta} + \dfrac{1}{\varepsilon^\nu} \widetilde{\mathcal H}_{\mrm{ms},\delta}\\
		\qquad = \bigl( \varepsilon^{\mu} \widetilde G \overline{\mathcal H}_0 + \varepsilon^{\mu+\nu} \widetilde G \widetilde{\mathcal H} \bigr) \bigl( \dt w + v\cdot \nablah w + w \dz w \bigr) \\
		\qquad\qquad - \mathcal C v_{\mrm{ms},\delta} \cdot \nablah w_\mrm{ms} - \mathcal C w_{\mrm{ms},\delta} \dz w_\mrm{ms}.
	\end{cases}
\end{equation}
Then similar arguments as before will yield 
\begin{equation}\label{est:uni-total-dff-1-2}
	\begin{aligned}
	& \sup_{0\leq s \leq T_{\mrm{ms},\delta}}\norm{\widetilde q_{\mrm{ms},\delta },\widetilde{\mathcal H}_{\mrm{ms},\delta}, v_{\mrm{ms},\delta }, w_{\mrm{ms},\delta}}{\Lnorm{2}}(s) \\
	 & \quad \leq C_{\mathcal K_1} \bigl(   \norm{\widetilde q_{\mrm{ms},\delta,\mrm{in}},\widetilde{\mathcal H}_{\mrm{ms},\delta,\mrm{in}}, v_{\mrm{ms},\delta,\mrm{in} }, w_{\mrm{ms},\delta,\mrm{in}}}{\Lnorm{2}} + (\varepsilon + \varepsilon^{\mu-\sigma} )  \bigr).
	 \end{aligned}
\end{equation}
Indeed, only the corresponding $ \mathfrak I_3 $ estimate is different, where the control of 
$$
\nablah p_\mrm{ms} \cdot v_{\mrm{ms},\delta} + \dz p_\mrm{ms} w_{\mrm{ms},\delta} + \dfrac{1}{\varepsilon^\nu} \widetilde{\mathcal H}_\mrm{ms} w_{\mrm{ms},\delta} = \mathcal O(\varepsilon^{-\nu})
$$
is replaced by
\begin{equation}\label{rm:error-1}
(\dt v + v\cdot \nablah v + w \dz v ) \cdot v_{\mrm{ms},\delta} + ( \dt w + v\cdot \nablah  w + w \dz w)  w_{\mrm{ms},\delta} = \mathcal O(\varepsilon^{-\sigma}).
\end{equation}
Estimate \eqref{est:uni-total-dff-1} follows from \eqref{est:uni-total-dff-1-1} and \eqref{est:uni-total-dff-1-2}.

\subsection{The soundproof model}\label{subsec:snpf-md}.

{For convenience of the reader, we recall} that
the soundproof model reads%, which is formally the zero Mach number limit of \eqref{eq:rf-EEq}:
\begin{equation*}\tag{\ref{eq:sndprf}}
	\begin{cases}
		\dvh v_\mrm{sp} + \dz w_\mrm{sp} = 0, \\
		\mathcal B \dt \widetilde{\mathcal H}_\mrm{sp} + \mathcal B v_\mrm{sp}\cdot\nablah \widetilde{\mathcal H}_\mrm{sp} + \mathcal B w_\mrm{sp}\dz \widetilde{\mathcal H}_\mrm{sp} - \dfrac{1}{\varepsilon^{\nu}} w_\mrm{sp} = \mathcal B \widetilde G \cdot \widetilde{\mathcal H}_\mrm{sp} w_\mrm{sp}, \\
		\mathcal C \dt v_\mrm{sp} + \mathcal C v_\mrm{sp} \cdot \nablah v_\mrm{sp} + \mathcal C w_\mrm{sp}\dz v_\mrm{sp} + \nablah p_\mrm{sp} = 0,\\
		\mathcal C \dt w_\mrm{sp} + \mathcal C v_\mrm{sp} \cdot \nablah w_\mrm{sp} + \mathcal C w_\mrm{sp}\dz w_\mrm{sp} + \dz p_\mrm{sp} + \dfrac{1}{\varepsilon^\nu} \widetilde{\mathcal H}_\mrm{sp} = 0,
	\end{cases}
\end{equation*}
where $ p_\mrm{sp} $ is determined by
\begin{equation}\label{eq:pressure-sp}
	\begin{gathered}
		- \Delta p_\mrm{sp} %= \mathcal C \dvh (v_\mrm{sp} \cdot\nablah v_\mrm{sp} + w_\mrm{sp}\dz v_\mrm{sp} ) \\
%	 	+ \mathcal C \dz (v_\mrm{sp} \cdot\nablah w_\mrm{sp} + w_\mrm{sp}\dz w_\mrm{sp} ) + \dfrac{1}{\varepsilon^\nu} \widetilde{\mathcal H}_\mrm{sp}\\
	 	 = \mathcal C ( (\nablah v_\mrm{sp})^\top : \nablah v_\mrm{sp} + 2 \dz v_\mrm{sp} \cdot \nablah w_\mrm{sp} + (\dz w_\mrm{sp})^2) \\
	 	  + \dfrac{1}{\varepsilon^\nu}\dz \widetilde{\mathcal H}_\mrm{sp}, \qquad
	 	 \int p_\mrm{sp} \idx = 0.
	 \end{gathered}
\end{equation}
Then, following similar, if not simpler, arguments to those in section \ref{subsec:itmd-md} leads to the conclusion that: there exists $ T_\mrm{sp} \in (0,\infty) $, depending only on $ \norm{\widetilde{\mathcal H}_\mrm{sp,in},v_\mrm{sp,in},w_\mrm{sp,in}}{\Hnorm{3}} $, such that
\begin{equation}\label{est:uni-total-sp}
	\sup_{0\leq s \leq T_\mrm{sp}}\norm{\widetilde{\mathcal H}_\mrm{sp},v_\mrm{sp},w_\mrm{sp}}{\Hnorm{3}}(s) \leq C \norm{\widetilde{\mathcal H}_\mrm{sp,in},v_\mrm{sp,in},w_\mrm{sp,in}}{\Hnorm{3}},
\end{equation}
with some constant $ C \in (0,\infty) $, independent of $ \varepsilon $.

Now we list the estimate of $ p_\mrm{sp} $, induced by the elliptic estimates on \eqref{eq:pressure-sp}:
\begin{gather}
\label{est:pressure-1}
	\norm{p_\mrm{sp}}{\Hnorm{4}} \leq \norm{v_\mrm{sp},w_\mrm{sp}}{\Hnorm{3}}^2 + \dfrac{1}{\varepsilon^\nu}\norm{\widetilde{\mathcal H}_\mrm{sp}}{\Hnorm{3}}.
%	, \\
%	\label{est:pressure-2}
%	\norm{\dt p_\mrm{sp}}{\Hnorm{3}} \leq \norm{v_\mrm{sp},w_\mrm{sp}}{\Hnorm{3}}\norm{\dt v_\mrm{sp},\dt w_\mrm{sp}}{\Hnorm{2}} + \dfrac{1}{\varepsilon^\nu} \norm{\dt \widetilde{\mathcal H}_\mrm{sp}}{\Hnorm{2}}, \\
%	\label{est:pressure-3}
%	\begin{gathered}
%		\norm{\dt^2 p_\mrm{sp}}{\Hnorm{2}} \leq \norm{v_\mrm{sp},w_\mrm{sp}}{\Hnorm{3}}\norm{\dt^2 v_\mrm{sp},\dt^2 w_\mrm{sp}}{\Hnorm{1}} \\
%		+ \norm{\dt v_\mrm{sp},\dt w_\mrm{sp}}{\Hnorm{2}}^2 + \dfrac{1}{\varepsilon^\nu} \norm{\dt^2 \widetilde{\mathcal H}_\mrm{sp}}{\Hnorm{1}}.
%	\end{gathered}
\end{gather}

\subsection{Soundproof approximation}\label{subsec:snpf-ap}

Now we are ready to estimate the difference of solutions to \eqref{eq:mid-s} and \eqref{eq:sndprf}. Denote by
\begin{equation}\label{def:dff-vrb}
	\begin{aligned}
	p_{\mrm{sp},\delta} := & p_\mrm{ms} - p_\mrm{sp},\\
		\widetilde{\mathcal H}_{\mrm{sp},\delta} := & \widetilde{\mathcal H}_\mrm{ms} - \widetilde{\mathcal H}_\mrm{sp}, \\
		v_{\mrm{sp},\delta} := & v_\mrm{ms} - v_\mrm{sp}, \\
		w_{\mrm{sp},\delta} := & w_\mrm{ms} - w_\mrm{sp}.
	\end{aligned}
\end{equation}
Also, we use $ \mathcal K_2 $ to represent the total bound of solutions to \eqref{eq:mid-s} and \eqref{eq:sndprf}, i.e.,
\begin{equation}\label{est:uni-total-ms+sp}
	\sup_{0\leq s \leq T_\mrm{ms}} \norm{\widetilde{\mathcal H}_\mrm{ms}, v_\mrm{ms},w_\mrm{ms}}{\Hnorm{3}} + \sup_{0\leq s \leq T_\mrm{sp}} \norm{\widetilde{\mathcal H}_\mrm{sp}, v_\mrm{sp},w_\mrm{sp}}{\Hnorm{3}} \leq \mathcal K_2,
\end{equation}
which are obtained in \eqref{est:uni-total-ms} and \eqref{est:uni-total-sp}. 

After comparing \eqref{eq:mid-s} and \eqref{eq:sndprf}, one can derive that $ (p_{\mrm{sp},\delta}, \widetilde{\mathcal H}_{\mrm{sp},\delta}, v_{\mrm{sp},\delta}, w_{\mrm{sp},\delta}) $ satisfies
\begin{equation}\label{eq:sp-perturb}
	\begin{cases}
		\dvh v_{\mrm{sp},\delta} + \dz w_{\mrm{sp},\delta} = \varepsilon \mathcal A (\mathcal C G + \varepsilon^\mu \overline{\mathcal H}_0) w_\mrm{ms},\\
		\mathcal B \dt \widetilde{\mathcal H}_{\mrm{sp},\delta} + \mathcal B v_\mrm{ms} \cdot \nablah \widetilde{\mathcal H}_{\mrm{sp},\delta} + \mathcal B w_\mrm{ms} \dz \widetilde{\mathcal H}_{\mrm{sp},\delta} - \dfrac{1}{\varepsilon^\nu} w_{\mrm{sp},\delta}\\
		\qquad = - \mathcal B ( v_{\mrm{sp},\delta}\cdot \nablah \widetilde{\mathcal H}_\mrm{sp} + w_{\mrm{sp},\delta} \dz \widetilde{\mathcal H}_\mrm{sp} ) \\
		\qquad\qquad + \mathcal B \widetilde G \cdot ( \widetilde{\mathcal H}_\mrm{ms} w_{\mrm{sp},\delta} + \widetilde{\mathcal H}_{\mrm{sp},\delta} w_\mrm{sp} ),\\
		\mathcal C \dt v_{\mrm{sp},\delta} + \mathcal C v_\mrm{ms}\cdot\nablah v_{\mrm{sp},\delta} + \mathcal C w_\mrm{ms} \dz v_{\mrm{sp},\delta} + \nablah p_{\mrm{sp},\delta} \\
		\qquad = - \mathcal C ( v_{\mrm{sp},\delta} \cdot \nablah v_\mrm{sp} + w_{\mrm{sp},\delta} \dz v_\mrm{sp} ), \\
		\mathcal C \dt w_{\mrm{sp},\delta} + \mathcal C v_\mrm{ms} \cdot \nablah w_{\mrm{sp},\delta} + \mathcal C w_\mrm{ms}\dz w_{\mrm{sp},\delta} + \dz p_{\mrm{sp},\delta} + \dfrac{1}{\varepsilon^\nu}\widetilde{\mathcal H}_{\mrm{sp},\delta}\\
		\qquad = - \mathcal C ( v_{\mrm{sp},\delta} \cdot \nablah w_\mrm{sp} + w_{\mrm{sp},\delta} \dz w_\mrm{sp} ).
	\end{cases}
\end{equation}
To write down the equation of $ p_{\mrm{sp},\delta} $, instead of using \eqref{eq:pressure-ms} and \eqref{eq:pressure-sp}, we first rewrite
\begin{align*}
	& v_\mrm{ms} \cdot \nablah \biggl( \begin{array}{c} v_{\mrm{sp},\delta}\\w_{\mrm{sp},\delta} \end{array}\biggr) + w_\mrm{ms} \dz \biggl( \begin{array}{c} v_{\mrm{sp},\delta}\\w_{\mrm{sp},\delta} \end{array}\biggr) \\
	& \quad = v_{\mrm{sp},\delta} \cdot \nablah \biggl( \begin{array}{c} v_{\mrm{sp},\delta}\\w_{\mrm{sp},\delta} \end{array}\biggr) + w_{\mrm{sp},\delta} \dz \biggl( \begin{array}{c} v_{\mrm{sp},\delta}\\w_{\mrm{sp},\delta} \end{array}\biggr) \\
	& \qquad + \biggl( \begin{array}{c} \dvh (v_{\mrm{sp},\delta}\otimes v_{\mrm{sp}}) + \dz (w_\mrm{sp} v_{\mrm{sp},\delta}) \\ \dvh ( w_{\mrm{sp},\delta} v_\mrm{sp}) + \dz (w_\mrm{sp} w_{\mrm{sp},\delta} ) \end{array} \biggr),
\end{align*}
and then after applying $$ \biggl( \begin{array}{c}
	\dvh \\ \dz 
\end{array} \biggr) \quad \text{to} \quad \biggl( \begin{array}{c} \eqref{eq:sp-perturb}_{3} \\ \eqref{eq:sp-perturb}_{4} \end{array}\biggr), $$ one can derive that
\begin{equation}\label{eq:pressure-perturb-sp}
	\begin{gathered}
		- \Delta p_{\mrm{sp},\delta} = \varepsilon \mathcal A \mathcal C  ( \mathcal C G + \varepsilon^\mu \overline{\mathcal H}_0 ) \dt w_\mrm{ms} + \mathcal C \biggl( \dvh  \dvh (v_{\mrm{sp},\delta} \otimes v_\mrm{sp}) \\
		+ \dvh \dz (w_\mrm{sp} v_{\mrm{sp},\delta}) + \dz \dvh (w_{\mrm{sp},\delta} v_\mrm{sp}) + \dz^2(w_\mrm{sp} w_{\mrm{sp},\delta}) \\
		+ \dvh ( v_{\mrm{sp},\delta} \cdot \nablah v_{\mrm{sp},\delta}) + \dvh (w_{\mrm{sp},\delta} \dz v_{\mrm{sp},\delta}) + \dz ( v_{\mrm{sp},\delta} \cdot \nablah w_{\mrm{sp},\delta}) \\
		+ \dz (w_{\mrm{sp},\delta} \dz w_{\mrm{sp},\delta} ) \biggr)
		+ \mathcal C \biggl( \dvh (v_{\mrm{sp},\delta} \cdot \nablah v_\mrm{sp}) + \dvh (w_{\mrm{sp},\delta} \dz v_\mrm{sp}) \\
			+ \dz( v_{\mrm{sp},\delta} \cdot \nablah w_\mrm{sp} ) + \dz ( w_{\mrm{sp},\delta} \dz w_\mrm{sp}) \biggr) + \dfrac{1}{\varepsilon^\nu} \dz \widetilde{\mathcal H}_{\mrm{sp},\delta},\\
			\int p_{\mrm{sp},\delta} \idx = 0.
	\end{gathered}
\end{equation}
Consequently, applying the standard elliptic estimate on \eqref{eq:pressure-perturb-sp} yields that
\begin{equation}\label{est:ms-sp-1}
	\begin{aligned}
		& \norm{p_{\mrm{sp},\delta}}{\Lnorm{2}} \lesssim \varepsilon \norm{\dt w_\mrm{ms}}{\Lnorm{2}} + \norm{v_{\mrm{sp},\delta} \otimes v_\mrm{sp},w_\mrm{sp}v_{\mrm{sp},\delta}}{\Lnorm{2}} \\
		& \qquad + \norm{w_{\mrm{sp},\delta} v_\mrm{sp}, w_\mrm{sp} w_{\mrm{sp},\delta}, v_{\mrm{sp},\delta} \cdot \nablah v_{\mrm{sp},\delta}, w_{\mrm{sp},\delta} \dz v_{\mrm{sp},\delta} }{\Lnorm{2}}\\
		& \qquad + \norm{v_{\mrm{sp},\delta} \cdot \nablah w_{\mrm{sp},\delta}, w_{\mrm{sp},\delta} \dz w_{\mrm{sp},\delta}, v_{\mrm{sp},\delta} \cdot\nablah v_\mrm{sp}, w_{\mrm{sp},\delta} \dz v_\mrm{sp}}{\Lnorm{2}}\\
		& \qquad + \norm{v_{\mrm{sp},\delta} \cdot\nablah w_\mrm{sp}, w_{\mrm{sp},\delta} \dz w_\mrm{sp} }{\Lnorm{2}} + \dfrac{1}{\varepsilon^\nu} \norm{\widetilde{\mathcal H}_{\mrm{sp},\delta}}{\Lnorm{2}}\\
		& \lesssim C_{\mathcal K_2} \bigl( 1 + \varepsilon^{-\nu} \bigr) \norm{\widetilde{\mathcal H}_{\mrm{sp},\delta}, v_{\mrm{sp},\delta}, w_{\mrm{sp},\delta}}{\Lnorm{2}} + C_{\mathcal K_2} \bigl(\varepsilon + \varepsilon^{1-\nu}\bigr),
	\end{aligned}
\end{equation}
where we have applied \eqref{est:pressure-ms-1}, \eqref{est:dt-ms-3}, and the Sobolev embedding inequality in the last inequality.

Now we are ready to estimate the $ L^2 $ norm of $ (p_{\mrm{sp},\delta}, \widetilde{\mathcal H}_{\mrm{sp},\delta}, v_{\mrm{sp},\delta}, w_{\mrm{sp},\delta}) $. Indeed, after applying the $ L^2 $-inner product of equations in \eqref{eq:sp-perturb} with $ 2 p_{\mrm{sp},\delta}, 2 \widetilde{\mathcal H}_{\mrm{sp},\delta}, 2 v_{\mrm{sp},\delta}, 2 w_{\mrm{sp},\delta} $, respectively, applying integration by parts, and summing up the resultant equations, one has
\begin{align*}
	& \dfrac{d}{dt} \biggl(\mathcal B\norm{\widetilde{\mathcal H}_{\mrm{sp},\delta}}{\Lnorm{2}}^2 + \mathcal C \norm{v_{\mrm{sp},\delta}, w_{\mrm{sp},\delta} }{\Lnorm{2}}^2 \biggr) = \sum_{j=5}^{7} \mathfrak I_{j}
\end{align*}
where
\begin{align*}
	& \mathfrak I_5 :=  \int (\dvh v_\mrm{ms} + \dz w_\mrm{ms} ) \bigl( \mathfrak B \lvert \widetilde{\mathcal H}_{\mrm{sp},\delta}\rvert^2 + \mathcal C \lvert v_{\mrm{sp},\delta} \rvert^2 + \mathcal C \lvert w_{\mrm{sp},\delta} \rvert^2 \bigr) \idx,\\
	& \mathfrak I_6 := - 2 \int \biggl( \mathfrak B ( v_{\mrm{sp},\delta} \cdot \nablah \widetilde{\mathcal H}_\mrm{sp} \widetilde{\mathcal H}_{\mrm{sp},\delta} + w_{\mrm{sp},\delta} \dz \widetilde{\mathcal H}_\mrm{sp}\widetilde{\mathcal H}_{\mrm{sp},\delta} ) \\
	& \qquad\qquad + \mathcal C ( v_{\mrm{sp},\delta} \cdot \nablah v_\mrm{sp} \cdot v_{\mrm{sp},\delta} + w_{\mrm{sp},\delta} \dz v_\mrm{sp} \cdot v_{\mrm{sp},\delta}) \\
	& \qquad\qquad + \mathcal C ( v_{\mrm{sp},\delta} \cdot \nablah w_\mrm{sp} w_{\mrm{sp},\delta} + w_{\mrm{sp},\delta} \dz w_\mrm{sp} w_{\mrm{sp},\delta} ) \biggr) \idx \\
	& \qquad + 2 \int \mathcal B \widetilde G \cdot ( \widetilde{\mathcal H}_\mrm{ms} w_{\mrm{sp},\delta} + \widetilde{\mathcal H}_{\mrm{sp},\delta} w_\mrm{sp} ) \widetilde{\mathcal H}_{\mrm{sp},\delta} \idx,\\
	& \mathfrak I_{7} := 2 \varepsilon \mathcal A \int (\mathcal C G + \varepsilon^\mu \overline{\mathcal H}_0) w_\mrm{ms} p_{\mrm{sp},\delta} \idx.
\end{align*}
Thanks to \eqref{est:ms-sp-1}, one has
\begin{equation*}
	\mathfrak I_7 \leq C_{\mathcal K_2} (\varepsilon + \varepsilon^{1-\nu}) \norm{\widetilde{\mathcal H}_{\mrm{sp},\delta}, v_{\mrm{sp},\delta},w_{\mrm{sp},\delta}}{\Lnorm{2}} + C_{\mathcal K_2} (\varepsilon^2 + \varepsilon^{2-\nu}),
\end{equation*}
while the estimates of $ \mathfrak I_5 $ and $ \mathfrak I_6 $ are straightforward. Hence, we have shown that 
\begin{gather*}
	\dfrac{d}{dt} \biggl(\mathcal B\norm{\widetilde{\mathcal H}_{\mrm{sp},\delta}}{\Lnorm{2}}^2 + \mathcal C \norm{v_{\mrm{sp},\delta}, w_{\mrm{sp},\delta} }{\Lnorm{2}}^2 \biggr) \\
	\leq C_{\mathcal K_2} \norm{\widetilde{\mathcal H}_{\mrm{sp},\delta}, v_{\mrm{sp},\delta},w_{\mrm{sp},\delta}}{\Lnorm{2}}^2 \\
	 + C_{\mathcal K_2} (\varepsilon + \varepsilon^{1-\nu}) \norm{\widetilde{\mathcal H}_{\mrm{sp},\delta}, v_{\mrm{sp},\delta},w_{\mrm{sp},\delta}}{\Lnorm{2}} + C_{\mathcal K_2} (\varepsilon^2 + \varepsilon^{2-\nu}).
\end{gather*}
Consequently, after applying Gr\"onwall's inequality, one can conclude that, there is $ T_{\mrm{sp},\delta} \in (0,\min\lbrace T_\mrm{ms},T_\mrm{sp} \rbrace ] $, depending only on $ \norm{\widetilde{\mathcal H}_{\mrm{sp},\delta,\mrm{in}}, v_{\mrm{sp},\delta,\mrm{in} }, w_{\mrm{sp},\delta,\mrm{in}}}{\Lnorm{2}} $,
%independent of $ \varepsilon $, 
such that
\begin{equation}\label{est:uni-total-dff-2}
	\begin{aligned}
		& \sup_{0\leq s \leq T_{\mrm{sp},\delta}} \norm{\widetilde{\mathcal H}_{\mrm{sp},\delta},v_{\mrm{sp},\delta},w_{\mrm{sp},\delta}}{\Lnorm{2}}(s) \\
		& \quad \leq C_{\mathcal K_2}\bigl(\norm{\widetilde{\mathcal H}_{\mrm{sp},\delta,\mrm{in}}, v_{\mrm{sp},\delta,\mrm{in} }, w_{\mrm{sp},\delta,\mrm{in}}}{\Lnorm{2}} + ( \varepsilon + \varepsilon^{1-\nu}) \bigr).
	\end{aligned}
\end{equation}
Here $ C_{\mathcal K_2} \in (0,\infty) $ is a constant depending only on $ \mathcal K_2 $ given in \eqref{est:uni-total-ms+sp}.
In particular, \eqref{est:uni-total-dff-1}, \eqref{est:ms-sp-1}, and \eqref{est:uni-total-dff-2} imply that, since $ \mu + 2\nu = 1 $,
\begin{equation}\label{est:uni-total-dff-3}
\begin{gathered}
	\sup_{0\leq s \leq \min\lbrace T_{\mrm{ms},\delta} T_{\mrm{sp},\delta} \rbrace} \norm{\widetilde q - \varepsilon p_\mrm{sp} ,\widetilde{\mathcal H} - \widetilde{\mathcal H}_\mrm{sp}, v - v_\mrm{sp}, w-w_\mrm{sp}}{\Lnorm{2}}\\
	 \leq C_{\mathcal K_1,\mathcal K_2} \bigl( \varepsilon^{\max\lbrace \mu - \nu,  \mu - \sigma \rbrace} 
	  + \norm{\widetilde q_{\mrm{ms},\delta,\mrm{in}},\widetilde{\mathcal H}_{\mrm{ms},\delta,\mrm{in}}, v_{\mrm{ms},\delta,\mrm{in} }, w_{\mrm{ms},\delta,\mrm{in}}}{\Lnorm{2}}\\
	   + \norm{\widetilde{\mathcal H}_{\mrm{sp},\delta,\mrm{in}}, v_{\mrm{sp},\delta,\mrm{in} }, w_{\mrm{sp},\delta,\mrm{in}}}{\Lnorm{2}} \bigr).
\end{gathered}
\end{equation}
Theorem \ref{thm:sp-approximation} follows from \eqref{est:pressure-ms-1}, \eqref{est:uni-total-ms}, \eqref{est:uni-total-sp}, \eqref{est:pressure-1}, and \eqref{est:uni-total-dff-3}

\section{Fast-slow decompositions: the linear theory}\label{sec:fast-slow-linear}

Our goal is to decompose the solution to \eqref{eq:rf-EEq} into  waves with different frequencies. Ideally, due the appearance of two different scales of oscillation, we are expecting at least three waves. 

\begin{enumerate}[label = {\bf H\arabic{enumi})}, label = {\bf H\arabic{enumi})}]
\setcounter{enumi}{\value{hypothesis}}
\item\label{def:AssumptionH4} To simplify our presentation, we will, from now on, assume that 
\begin{equation}\label{amtp:atsym-cfts}
\mathcal A = \mathcal B = \mathcal C = 1.
\end{equation}
\setcounter{hypothesis}{\value{enumi}}
\end{enumerate}

A linear system associated with \eqref{eq:rf-EEq} is introduced in this section, using two oscillation operators, corresponding to the acoustic waves and the internal waves, respectively. %As one will see, the induced acoustic waves are faster than the internal gravity waves. %As a start, we investigate the two oscillations with different frequencies separately. In particular, we introduce the decomposition of vector fields into the acoustic waves, the incompressible internal waves, and the incompressible mean flows. While such a decomposition provides the foundation to our soundproof approximation, but it is not enough to understand the whole picture in the full compressible system. 

In addition, we will investigate an $ \varepsilon $-dependent linear oscillation operator, associated with the linear system, which can be treated as a perturbation of the acoustic wave operator. The eigenvalue-eigenvector pairs associated with such oscillation operator will be investigated. 

%\subsection{Linear oscillations}

To be more precise, we introduce the following linear system:
\begin{equation}\label{eq:linear-osc}
	\dt U + \dfrac{1}{\varepsilon} \mathcal L_a U + \dfrac{1}{\varepsilon^\nu} \mathcal L_g U = 0,
\end{equation}
where $ U, \mathcal L_a,\, \mathcal L_g $ are defined as in \eqref{def:vec-opt}. 
Roughly speaking, $ \frac{1}{\varepsilon} \mathcal L_a U $ and $ \frac{1}{\varepsilon^\nu} \mathcal L_g U $ are the driving forces of acoustic waves and internal waves, respectively. 
One can immediately see from \eqref{eq:linear-osc}, that, as $ \varepsilon \rightarrow 0^+ $, the oscillation induced by operator $ \frac{1}{\varepsilon}\mathcal L_a $ is faster than the one induced by $ \frac{1}{\varepsilon^\nu}\mathcal L_g $, meaning that the acoustic waves will oscillate faster and thus will be averaged out before the internal waves dissipate. This is exactly why we can use the soundproof system \eqref{eq:sndprf} as an approximation to \eqref{eq:rf-EEq}. 

In the following subsections, we will investigate the acoustic waves, internal waves, and mean flows, in the linear system \eqref{eq:linear-osc}.

\subsection{Perturbed acoustic waves}

In this subsection, we consider the following perturbed acoustic wave operator
\begin{equation}\label{def:ptb-aw-opt}
	\mathcal L_\varepsilon := \mathcal L_a + \varepsilon^{1-\nu} \mathcal L_g.
\end{equation}
Then \eqref{eq:linear-osc} is equivalent to
\begin{equation}\label{eq:linear-osc-2}
	\dt U + \dfrac{1}{\varepsilon} \mathcal L_\varepsilon U = 0.
\end{equation}
Notice that $ \mathcal L_\varepsilon $ can be viewed as a perturbation of $ \mathcal L_a $. An ad hoc analysis will be that, the eigenvalues of $ \mathcal L_\varepsilon $ lie within neighborhoods with width $ \mathcal O (\varepsilon^{1-\nu}) $ of the eigenvalues of $ \mathcal L_a $. In particular, the eigenvalues corresponding to the acoustic free vector fields lies in an neighborhood with width $ \mathcal O (\varepsilon^{1-\nu}) $ of the origin. In view of \eqref{eq:linear-osc-2}, one can decompose the eigenvalues of $ \frac{1}{\varepsilon}\mathcal L_\varepsilon $, corresponding to the wave decomposition of solutions to \eqref{eq:linear-osc-2}, into three kinds: the zero eigenvalue; the eigenvalues of $ \mathcal O(\varepsilon^{-\nu}) $ near the origin; the eigenvalues of $ \mathcal O(\varepsilon^{-1})  $ ($ \mathcal O(\varepsilon^{-1}) \pm \mathcal O(\varepsilon^{-\nu}) $ to be more precise). We will refer the waves corresponding to these three kinds of eigenvalues as the {\bf mean flows}, the {\bf perturbed internal waves}, and the {\bf perturbed acoustic waves}, respectively. In the following, we shall make the above ad hoc discussion rigid. 

Let 
\begin{equation}\label{def:vec-field}
	\begin{aligned}
		\mathfrak V := \bigl\lbrace & U = (\widetilde q, \widetilde{\mathcal H}, v = (v_1,v_2)^\top, w)^\top \in C^\infty(\mathbb T^3 ; \mathbb R^5) \vert \\ & \text{Symmetry \eqref{SYM} is satisfied.} \bigr\rbrace.
	\end{aligned}
\end{equation}
We first investigate $ \ker \mathcal L_\varepsilon $, i.e., the space associated with the zero eigenvalue. Let
\begin{equation}\label{def:project-3}
	\mathcal P_{\varepsilon,\mrm{mf}}: \mathfrak V \mapsto \ker \mathcal L_\varepsilon. % \quad \mathcal P_{\varepsilon,\mrm{ac}} := Id - \mathcal P_{\varepsilon, \mrm{mf}}: \mathfrak V \mapsto (\ker \mathcal L_\varepsilon)^\perp. 
\end{equation}
Then
\begin{equation}\label{def:ptb-mf-vec-field}
	\begin{aligned}
		\ker \mathcal L_\varepsilon = & \bigl\lbrace U_{\varepsilon,\mrm{mf}} =  (\widetilde q_{\varepsilon,\mrm{mf}}, \mathcal{\mathcal H}_{\varepsilon,\mrm{mf}}, v_{\varepsilon,\mrm{mf}}, w_{\varepsilon,\mrm{mf}})^\top \in \mathfrak V \vert \\
		& \quad \widetilde q_{\varepsilon,\mrm{mf}} = \widetilde q_{\varepsilon,\mrm{mf}}(z) \in C^\infty(\mathbb T;\mathbb R), \widetilde{\mathcal H}_{\varepsilon,\mrm{mf}} = - \varepsilon^{\nu -1} \dz \widetilde q_{\varepsilon,\mrm{mf}},\\
		& \quad \dvh v_{\varepsilon,\mrm{mf}} = 0, w_{\varepsilon,\mrm{mf}} = 0 \bigr\rbrace.
%		, \\
%		(\ker \mathcal L_\varepsilon)^\perp = & \bigl\lbrace U_{\varepsilon,\mrm{aw}} =  (\widetilde q_{\varepsilon,\mrm{aw}}, \mathcal{\mathcal H}_{\varepsilon,\mrm{aw}}, v_{\varepsilon,\mrm{aw}}, w_{\varepsilon,\mrm{aw}})^\top \in \mathfrak V\lvert\\
%		& \quad \int ( \widetilde q_{\varepsilon, \mrm{aw}} + \varepsilon^{\nu-1} \dz \widetilde{\mathcal H}_{\varepsilon, \mrm{aw}}) \,dxdy(z) = 0,\\
%		& \quad v_{\varepsilon,\mrm{aw}} = \nablah \phi, ~ \text{for some $ \phi \in C^\infty(\mathbb T^3;\mathbb R) $} \bigr\rbrace.
	\end{aligned}
\end{equation}
Denote by
$$
	U_{\varepsilon,\mrm{mf}}= (\widetilde q_{\varepsilon,\mrm{mf}}, \widetilde{\mathcal H}_{\varepsilon,\mrm{mf}}, v_{\varepsilon,\mrm{mf}}, w_{\varepsilon,\mrm{mf}})^\top = \mathcal P_{\varepsilon,\mrm{mf}}(U = (\widetilde q, \widetilde{\mathcal H}, v,w)^\top ).
$$
Then, to look for the representation of $ \mathcal P_{\varepsilon,\mrm{mf}} $, we calculate the following functional: for any $ V = ( a, b ,\xi ,\eta )^\top \in \ker \mathcal L_\varepsilon $
\begin{align*}
	& \norm{ V - U }{\Lnorm{2}}^2 = \int \bigl(\lvert\widetilde q - a \rvert^2 + \lvert\varepsilon^{\nu-1} \dz a + \widetilde{\mathcal H} \rvert^2 \bigr) \idx \\
	& \qquad \qquad + \int \bigl( \lvert\xi - v\rvert^2 + \lvert w \rvert^2 \bigr) \idx.
\end{align*}
Then $ U_{\varepsilon,\mrm{mf}}$ should be the minimizer of the above functional subject to the condition $ U_{ \varepsilon,\mrm{mf}} \in \ker \mathcal L_\varepsilon $. Then calculating the Euler-Lagrangian equations yields that $ U_{\varepsilon, \mrm{mf}} = \mathcal P_{\varepsilon,\mrm{mf}}(U) $ is given by
%
%\pagebreak
%Then applying similar arguments as before yields that 
\begin{align*}
	\widetilde{\mathcal H}_{\varepsilon,\mrm{mf}} \equiv & - \varepsilon^{\nu-1} \dz \widetilde{q}_{\varepsilon,\mrm{mf}},\\
	v_{\varepsilon,\mrm{mf}} \equiv & v - \nablah \psi_v, \\
	w_{\varepsilon,\mrm{mf}} \equiv & 0,
\end{align*}
where $ \widetilde q_{\varepsilon,\mrm{mf}},\, \psi_v $ are solutions to
\begin{equation}\label{eq:ptb-potential-1}
	\begin{gathered}
		- \varepsilon^{2(\nu-1)} \partial_{zz} \widetilde q_{\varepsilon,\mrm{mf}} +  \widetilde q_{\varepsilon, \mrm{mf}} - \varepsilon^{\nu-1}\dz \int \widetilde{\mathcal H} \,dxdy(z) \\
		 -  \int \widetilde q \,dxdy(z) = 0, \qquad \int \widetilde q_{\varepsilon,\mrm{mf}} \,dz = \int \widetilde q \idx, \\
		\text{and} \quad \deltah \psi_v = \dvh v, \quad \int \psi_v \,dxdy = 0. 
	\end{gathered}
\end{equation}
%One can check, formally, as $ \varepsilon \rightarrow 0 $, $ \mathcal P_{\varepsilon,\mrm{mf}} \rightarrow \mathcal P_\mrm{mf}\circ \mathcal P_\mrm{sp} $.

We remind readers that $ \ker \mathcal L_\varepsilon $ is nothing but the space of eigenfunctions corresponding to the zero eigenvalue of $  \mathcal L_\varepsilon $. 
Next we focus on the non-zero eigenvalue problem of $ \mathcal L_\varepsilon $, i.e., the structure of $ (\ker \mathcal L_\varepsilon)^\perp $. Since $ \mathcal L_\varepsilon $ is anti-symmetric, it suffices to investigate the pure imaginary eigenvalues with non-zero imaginary part, i.e., 
\begin{equation}\label{def:eg-pgm-og}
i \omega U_{\omega} = \mathcal L_\varepsilon U_{\omega}, \quad \omega \neq 0, \quad U_\omega = (\widetilde q_\omega, \widetilde{\mathcal H}_\omega, v_\omega, w_\omega )^\top \in \mathfrak V. 
\end{equation}

We will not discuss the representations of the solutions to the eigenvalue problem in this section. Instead, we would like to estimate the value of the eigenvalues, assuming we have found the eigenvalue-eigenvector pairs. The exact quantity calculation will be postponed in the next section using Fourier representation. 

If $ \widetilde q_\varepsilon \equiv 0 $, then the eigenvalue problem \eqref{def:eg-pgm-og} is reduced to
\begin{equation*}
	v_\varepsilon = 0, \quad \dz w_\varepsilon =  0, \quad
	- \eta w_\varepsilon =  i \omega \widetilde{\mathcal H}, \quad
	\eta \widetilde{\mathcal H} =  i \omega w,
\end{equation*}
where, hereafter, $ \eta := \varepsilon^{1-\nu} $,
which yields $ U_\varepsilon = 0 $ due to symmetry \eqref{SYM}. 

In the following, we assume, without loss of generality, $ \widetilde q_\varepsilon \not\equiv 0 $. 
Direct calculation of the eigenvalue problem \eqref{def:eg-pgm-og} shows that
\begin{equation}\label{eq:eg-pbm}
	- (\omega^2 - \eta^2 ) \deltah \widetilde q_\varepsilon - \omega^2 \partial_{zz} \widetilde q_\varepsilon = \omega^2 ( \omega^2 - \eta^2 ) \widetilde q_\varepsilon.
\end{equation}
Notice that when $ \omega \simeq \eta $, \eqref{eq:eg-pbm} admits strong degeneracy. 

In addition, we introduce the following eigenvalue problem:
\begin{equation}\label{eq:eg-pbm-aw}
	- \deltah \widetilde q_\mrm{ac} - \partial_{zz} \widetilde q_\mrm{ac} = \omega^2_\mrm{ac} \widetilde q_\mrm{ac}.
\end{equation}
In fact, \eqref{eq:eg-pbm-aw} can be seen as the counter-part of \eqref{eq:eg-pbm} from \eqref{def:eg-pgm-og} for $ \mathcal L_a $, i.e., the eigenvalue problem of the acoustic operator. Unsurprisingly, \eqref{eq:eg-pbm-aw} is just \eqref{eq:eg-pbm} when $ \eta = 0 $, at least formally. We denote the eigenvalue-eigenfunction pairs of \eqref{eq:eg-pbm-aw} as $ (\omega_{\mrm{ac}, n}^\pm \widetilde q_{\mrm{ac},n} )\vert_{n = 0,1,2,\cdots} $, where $ \omega_{\mrm{ac},0}^\pm = 0 $, $ \lvert\omega_{\mrm{ac},1}^\pm\rvert < \lvert\omega_{\mrm{ac},2}^\pm\rvert < \cdots $, $ \int \lvert\widetilde q_{\mrm{ac},n}\rvert^2 \idx = 1 $. 
Then it is easy to check
\begin{equation}\label{id:eg-phm-aw}
	\begin{gathered}
		\int \widetilde q_{\mrm{ac},m} \widetilde q_{\mrm{ac},n} \idx = \delta_{m,n}, \quad \int \nablah \widetilde q_{\mrm{ac},m} \cdot \nablah \widetilde q_{\mrm{ac},n} \idx = \norm{\nablah \widetilde q_{\mrm{ac},m}}{\Lnorm{2}}^2 \delta_{m,n}, \\
		\int \nablah \widetilde q_{\mrm{ac},m} \cdot \nablah \widetilde q_{\mrm{ac},n} \idx + \int \dz  \widetilde q_{\mrm{ac},m} \dz \widetilde q_{\mrm{ac},n} \idx = \lvert\omega_{\mrm{ac},m}^\pm\rvert^2 \delta_{m,n}, \\
%\norm{\nablah \widetilde q_{\mrm{ac},m}}{\Lnorm{2}}^2 + \norm{\dz \widetilde q_{\mrm{ac},m}}{\Lnorm{2}}^2 = \omega_{\mrm{ac},m}^2, 
		\quad m,n \in  \cup \lbrace 0,1,2,\cdots \rbrace.
	\end{gathered}
\end{equation}
Then, one can represent solution $ \widetilde q_\varepsilon $ to \eqref{eq:eg-pbm} as
\begin{equation}\label{id:eg-ep-fnct}
	\widetilde q_\varepsilon = \sum_{n = 0,1,2, \cdots} Q_n \widetilde q_{\mrm{ac},n}, \quad Q_n \in \mathbb R.
\end{equation}
After taking the $ L^2 $-inner product of \eqref{eq:eg-pbm} with $ \widetilde q_{\mrm{ac},m} $, it follows that, thanks to \eqref{id:eg-phm-aw}, 
\begin{equation*}
	\begin{gathered}
	Q_m \times \lbrack \omega^2 (\omega^2 - \eta^2) - \omega^2 \lvert\omega_{\mrm{ac},m}^\pm\rvert^2 + \eta^2 \norm{ \nablah \widetilde q_{\mrm{ac},m}}{\Lnorm{2}}^2   \rbrack = 0,\\
	\quad \text{for any $m \in \lbrace 0,1,2,\cdots \rbrace$}.
	\end{gathered}
\end{equation*}
Suppose that for some $ m $, $ Q_m \neq 0 $. Then
\begin{equation}\label{eq:algebra}
\omega^2 (\omega^2 - \eta^2) - \omega^2 \lvert \omega_{\mrm{ac},m}^\pm\rvert^2 + \eta^2 \norm{ \nablah \widetilde q_{\mrm{ac},m}}{\Lnorm{2}}^2 = 0. 
\end{equation}
We {\bf claim} that 
\begin{equation}\label{est:eg-vlu-total}
	\lvert\omega_{\mrm{ac},m}^\pm\rvert^2  \leq \lvert\omega\rvert^2 \leq \lvert\omega_{\mrm{ac},m}^\pm\rvert^2 + \eta^2. % \quad \text{or} \quad \lvert\omega\lvert \leq \eta .
\end{equation}
%In particular, there exists at most one $ m\in \lbrace 0 , 1,2,\cdots \rbrace $ such that $ Q_m \neq 0 $, for $ \eta $ small enough. 

The rest of this section is devoted to the proof of \eqref{est:eg-vlu-total}.
Notice that, if $ m = 0 $, we have $ \omega_{\mrm{ac},0}^\pm = 0 $, $\nablah \widetilde q_{\mrm{ac},0} = 0 $, which implies $ \omega = 0 $ or $ \lvert\omega\rvert = \eta $. In particular $ 0 \leq \lvert\omega\rvert \leq \eta $, i.e., \eqref{est:eg-vlu-total} holds. 

Without loss of generality, we assume that $ m \geq 1 $ and $ \lvert\omega\rvert > \eta $, below. 
%
%We consider two cases: $ \lvert\omega\lvert > \lvert\eta\lvert $ and $ \lvert\omega\lvert \leq \lvert\eta\lvert $. 
%{\par\noindent\bf Case $ \lvert\omega\lvert > \lvert\eta\lvert $: } 
Since, from \eqref{id:eg-phm-aw}, $ \norm{ \nablah \widetilde q_{\mrm{ac},m}}{\Lnorm{2}}^2 \leq \lvert\omega_{\mrm{ac},m}^\pm\rvert^2 $, one has, from \eqref{eq:algebra}, that
$$
\omega^2 ( \omega^2 - \eta^2) \geq \lvert\omega_{\mrm{ac},m}^\pm\rvert^2 ( \omega^2 - \eta^2 ),
$$
which implies 
\begin{equation}\label{est:eg-vlu}
	\lvert\omega\rvert \geq \lvert\omega_{\mrm{ac},m}^\pm\rvert.% \geq \lvert \omega_{\mrm{ac},1}\lvert > 0.
\end{equation}
On the other hand, \eqref{eq:algebra} can be written as
\begin{gather*}
	\omega^2 - \lvert\omega_{\mrm{ac},m}^\pm\rvert^2 = \eta^2 ( 1 - \dfrac{\norm{\nablah \widetilde q_{\mrm{ac},m}}{\Lnorm{2}}^2}{\omega^2})  \leq \eta^2.
\end{gather*}
%which implies $ \lvert\omega\lvert \leq \lvert\omega_{\mrm{ac},m}^\pm\lvert + \eta $. 
Together with \eqref{est:eg-vlu},
this proves \eqref{est:eg-vlu-total}. 
%{\par\noindent\bf Case $ \lvert\omega\lvert < \lvert\eta\lvert $: } Applying similar arguments as above shows that
%$$
%	\lvert\omega\lvert < \lvert\omega_{}
%$$

Therefore, 
we have proved the following lemma:
\begin{lemma}\label{lm:est-of-eg}
	Let $ i \omega $ be an eigenvalue of operator $ \mathcal L_\varepsilon $. Then, there exists $ m \in \lbrace 0, 1,2,\cdots \rbrace $, such that
	$$
	\lvert\omega_{\mrm{ac},m}^\pm\rvert^2 \leq \omega^2 \leq \lvert\omega_{\mrm{ac},m}^\pm\rvert^2 + \varepsilon^{2-2\nu},
	$$
	where $ \lbrace i \omega_{\mrm{ac},m}^\pm\rbrace_{m\in \lbrace 0,1,2,\cdots \rbrace } $ are the eigenvalues of $ \mathcal L_a $. 
\end{lemma}
In particular, Lemma \ref{lm:est-of-eg} confirms the ad hoc analysis at the beginning of this section. 

We remark that, \eqref{eq:algebra} can be solved explicitly for $ \omega^2 $. Indeed, there exist exactly two solutions $ (\omega^2)_1 $ and $ (\omega^2)_2 $ satisfying \eqref{est:eg-vlu-total}. We will make it more clear using Fourier representations in the next subsection.

%We remark that, in the previous study [paper by didier rupert etc.], the eigenfunctions corresponding to the eigenvalues near $ 0 $ are referred to the gravity internal waves, while the rest are referred to the (perturbed) acoustic waves.
%In this paper, we have called the ones corresponding to the zero eigenvalue the mean flows, the ones corresponding to the almost-zero eigenvalues the internal waves, and the rest the perturbed acoustic waves, respectively. 

\subsection{Fourier representations}
\label{subsec:FourierRepresentations}

Owing to the symmetry \eqref{SYM}, we consider the follow Fourier expansion of $ U $:
\begin{equation}\label{def:frr-epsn}
	U = \sum_{ k_h \in 2\pi \mathbb Z^2, k_z \in 2\pi \mathbb N^+\cup\lbrace 0 \rbrace} \left(
	\begin{array}{c}
		Q_{({k}_h,k_z)}e^{ i {k}_h \cdot x} \cos( k_z z) \\  H_{({k}_h,k_z)} e^{ i {k}_h \cdot x} \sin( k_z z) \\ 
		V_{({k}_h,k_z)} e^{ i {k}_h \cdot x} \cos( k_z z) \\   W_{({k}_h,k_z)} e^{ i {k}_h \cdot x} \sin( k_z z)
	\end{array}
	\right),
\end{equation}
%\begin{equation}\label{def:frr-epsn}
%	U_\varepsilon = \sum_{ k_h \in \mathbb Z^2, k_z \in \mathbb N} \left(
%	\begin{array}{c}
%		Q_{({k}_h,k_z)}e^{2 \pi i {k}_h \cdot x} \dfrac{e^{2 \pi i  k_z z } +  e^{-2 \pi i  k_z z}}{2} \\ -i H_{({k}_h,k_z)} e^{2 \pi i {k}_h \cdot x} \dfrac{e^{2 \pi i  k_z z } -  e^{-2 \pi i  k_z z}}{2} \\ 
%		V_{({k}_h,k_z)} e^{2 \pi i {k}_h \cdot x} \dfrac{e^{2 \pi i  k_z z } +  e^{-2 \pi i  k_z z}}{2} \\  -i W_{({k}_h,k_z)} e^{2 \pi i {k}_h \cdot x} \dfrac{e^{2 \pi i  k_z z } -  e^{-2 \pi i  k_z z}}{2}
%	\end{array}
%	\right),
%\end{equation}
with
$$
F_{(-k_h,k_z)} = F_{(k_h,k_z)}, \qquad F \in \lbrace Q, H, V, W \rbrace. 
$$

Then, with $ \eta = \varepsilon^{1-\nu} \ll 1 $, the eigenvalue problem \eqref{def:eg-pgm-og} can be written as:
\begin{equation}\label{eq:frr-eg-pgm}
	\begin{aligned}
		\omega_{(k_h,k_z)} Q_{(k_h,k_z)} = & k_h \cdot V_{(k_h,k_z)} - i  k_z W_{(k_h,k_z)}, \\
		\omega_{(k_h,k_z)} H_{(k_h,k_z)} = & i \eta W_{(k_h,k_z)}, \\
		\omega_{(k_h,k_z)} V_{(k_h,k_z)} = &  Q_{(k_h,k_z)} k_h, \\
		\omega_{(k_h,k_z)} W_{(k_h,k_z)} = & i  k_z Q_{(k_h,k_z)} - i \eta H_{(k_h,k_z)}. 
	\end{aligned}
\end{equation}
We investigate the solutions to \eqref{eq:frr-eg-pgm} in the following three cases:

{\par\noindent \bf Case 1: $ \omega_{(k_h,k_z)} = 0 $.} Then if $ k_h \neq (0,0) $, one can easily check that $$ \left(\begin{array}{c}
		Q_{(k_h,k_z)} \\ H_{(k_h,k_z)} \\ V_{(k_h,k_z)} \\ W_{(k_h,k_z)}
	\end{array} \right) = \left(\begin{array}{c}
	0 \\ 0 \\ V_{(k_h,k_z)} \\ 0
\end{array}\right), \quad k_h \cdot V_{(k_h,k_z)} = 0, $$
or, equivalently,
\begin{equation*}
	\left(\begin{array}{c}
		Q_{(k_h,k_z)} \\ H_{(k_h,k_z)} \\ V_{(k_h,k_z)} \\ W_{(k_h,k_z)}
	\end{array} \right) = \pm \lvert V_{(k_h,k_z)}\rvert
	\left(\begin{array}{c}
		0 \\ 0 \\ \dfrac{k_h^\perp}{\lvert k_h\rvert} \\ 0
	\end{array}\right).
\end{equation*}

On the other hand, $ k_h = (0,0)	 $ imply that 
\begin{equation*}
	\left(\begin{array}{c}
		Q_{((0,0),k_z)} \\ H_{((0,0),k_z)} \\ V_{((0,0),k_z)} \\ W_{((0,0),k_z)}
	\end{array} \right) = \left(
	\begin{array}{c}
		Q_{((0,0),k_z)} \\   \dfrac{k_z}{\eta} Q_{((0,0),k_z)} \\ V_{((0,0),k_z)} \\ 0
	\end{array}
	\right) = Q_{((0,0),k_z)} \left( \begin{array}{c}
		1 \\  \dfrac{k_z}{\eta} \\ 0 \\ 0
	\end{array}\right)
	+  \left( \begin{array}{c}
		0 \\ 0 \\ V_{((0,0),k_z)} \\ 0
	\end{array}\right). 
\end{equation*}

{\par\noindent\bf Case 2: $ \lvert \omega_{(k_h,k_z)}\rvert = \lvert\eta\rvert $.} If $ k_z \neq 0 $, it is easy to check that there is no non-trivial solution to \eqref{eq:frr-eg-pgm}. Thus $ k_z = 0 $, and one can find the following solution: $ \lvert\omega_{(k_h,0)}\rvert = \lvert\eta\rvert \ll 1 $, and
$$
\begin{gathered}
\left(\begin{array}{c}
		Q_{(k_h,0)} \\ H_{(k_h,0)} \\ V_{(k_h,0)} \\ W_{(k_h,0)}
	\end{array} \right)
	=
	\left( \begin{array}{c}
		0 \\ H_{(k_h,0)} \\ 0 \\ -i H_{(k_h,0)}\dfrac{\omega_{(k_h,0)}}{\eta}
	\end{array} \right) = % Q_{k_h,0}  \left(\begin{array}{c}
%1 \\ 0 \\ \dfrac{\omega_{k_h,0}}{\lvert\omega_{k_h,0}\lvert} \dfrac{k_h}{\lvert k_h\lvert} \\ 0	
%\end{array}\right) \\
% + 
 H_{(k_h,0)} \left(\begin{array}{c} 0 \\ 1 \\ 0 \\ -i \dfrac{\omega_{(k_h,0)}}{\eta} \end{array} \right),\\
 k_h \neq (0,0).
\end{gathered}
$$
{\par\noindent\bf Case 3: $ \omega_{(k_h,k_z)} \neq 0 $ nor $  \lvert \omega_{(k_h,k_z)}\rvert \neq \lvert \eta\rvert $.}
Then from \eqref{eq:frr-eg-pgm}, one can derive 
\begin{equation}\label{eq:frr-eg-phm-algb}
\begin{gathered}
\bigl(  (\omega^2_{(k_h,k_z)} - \eta^2)  \lvert k_h\rvert^2 + \omega^2_{(k_h,k_z)} \lvert k_z\rvert^2 - (\omega^2_{(k_h,k_z)} - \eta^2)\omega^2_{(k_h,k_z)}  \bigr)\\
\times  Q_{(k_h,k_z)} 
 = 0.
 \end{gathered}
\end{equation}
Notice that \eqref{eq:frr-eg-phm-algb} is just the Fourier representation of \eqref{eq:eg-pbm}. If $ Q_{(k_h, k_z)} = 0 $, one can easily check from \eqref{eq:frr-eg-pgm}, only when $ \lvert\omega_{(k_h,k_z)}\rvert = \lvert \eta \rvert $ or $ 0 $, there will be non-trivial solutions,  which is already covered in the previous case. Therefore, we focus on \eqref{eq:frr-eg-phm-algb} when $ Q_{(k_h, k_z)} \neq 0 $, which leads to the algebraic equation
\begin{equation}\label{eq:frr-algebra}
\omega_{(k_h,k_z)}^4 - ( \lvert k_h\rvert^2 +  \lvert k_z\rvert^2 + \eta^2) \omega_{(k_h,k_z)}^2 + \eta^2 \lvert k_h\rvert^2 = 0.
	%(\omega^2_{k_h,k_z} - \eta^2)  (2\pi)^2 \lvert k_h\rvert^2 + \omega^2_{k_h,k_z} (2\pi)^2 \lvert k_z\rvert^2 - (\omega^2_{k_h,k_z} - \eta^2)\omega^2_{k_h,k_z}
\end{equation}
Notice that \eqref{eq:frr-algebra} is nothing but \eqref{eq:algebra}. Thus, the solutions to \eqref{eq:frr-algebra} are given by
\begin{equation}\label{id:frr-sl-eg-pbm}
	\begin{aligned}
		\omega_{(k_h,k_z)}^2 = & \dfrac{\lvert k_h\rvert^2 +  \lvert k_z\rvert^2 + \eta^2 + \sqrt{A}}{2}, \quad \text{or} \\
		\omega_{(k_h,k_z)}^2 = & \dfrac{\lvert k_h \rvert^2 +  \lvert k_z\rvert^2 + \eta^2 - \sqrt{A}}{2}\\
		= & \dfrac{2\eta^2 \lvert k_h\rvert^2 }{\lvert k_h\rvert^2 + \lvert k_z\rvert^2 + \eta^2 + \sqrt{A}} \in [0,\eta^2],
	\end{aligned}
\end{equation} 
where $A := (\lvert k_h\rvert^2 +  \lvert k_z\rvert^2 + \eta^2 )^2-4\eta^2  \lvert k_h\rvert^2 = ( \lvert k_h\rvert^2 - \eta^2 )^2 + \lvert k_z\rvert^4 + 2\lvert k_h\rvert^2\lvert k_z\rvert^2 + 2 \eta^2 \lvert k_z\rvert^2 \geq ( \lvert k_h\rvert^2 - \eta^2 )^2 \geq 0 $.
%, which correspond to \eqref{est:eg-vlu-total}. 
Then the solution to \eqref{eq:frr-eg-pgm} is given by
\begin{equation*}
	\begin{aligned}
		\left(\begin{array}{c}
		Q_{(k_h,k_z)} \\ H_{(k_h,k_z)} \\ V_{(k_h,k_z)} \\ W_{(k_h,k_z)}
	\end{array} \right) = & \left(\begin{array}{c} Q_{(k_h,k_z)} \\ H_{(k_h,k_z)} \\ \dfrac{1}{\omega_{(k_h,k_z)}} Q_{(k_h,k_z)} k_h \\ - i \dfrac{\omega_{(k_h,k_z)}}{\eta} H_{(k_h,k_z)}
	\end{array}\right),
	\end{aligned}
\end{equation*}
with $ H_{(k_h,k_z)} $ satisfies
\begin{equation}\label{eq:h-q}
\begin{gathered}
 k_z  H_{(k_h,k_z)} = \eta (\dfrac{\lvert k_h\rvert^2 }{\omega_{(k_h,k_z)}^2} - 1  ) Q_{(k_h,k_z)} \qquad \text{and}\\
 \eta(1-\dfrac{\omega^2_{(k_h,k_z)}}{\eta^2})H_{(k_h,k_z)} = k_z Q_{(k_h,k_z)}.
 \end{gathered}
\end{equation}
If $ k_z = 0 $, from \eqref{id:frr-sl-eg-pbm} and \eqref{eq:h-q}, 
$$
\omega_{(k_h,0)}^2 =  \lvert k_h\rvert^2 \quad \bigl( \text{or} \quad \omega_{(k_h,0)}^2 = \eta^2 \quad (\text{discarded}) \bigr)
$$
%However, \eqref{eq:h-q} yields that 
with $ k_h \neq (0,0) $ (otherwise it is covered in previous case), 
and thus
$$
\left(\begin{array}{c}
		Q_{(k_h,0)} \\ H_{(k_h,0)} \\ V_{(k_h,0)} \\ W_{(k_h,0)}
	\end{array} \right)
	=
	\left( \begin{array}{c}
		Q_{(k_h,0)} \\ 0 \\ Q_{(k_h,0)}\dfrac{ k_h}{\omega_{(k_h,0)}} \\ 0
	\end{array} \right).
$$
We remark that for $ \eta $ small enough, in oder to reach the endpoint values of \subeqref{id:frr-sl-eg-pbm}{2}, i.e., $ \lvert \omega_{(k_h,k_z)}\rvert = 0 \, \text{or} \, \eta  $, the necessary condition will be $ k_h = (0,0) $ or $ k_z = 0 $, respectively, while $ k_h = (0,0) $ is also a sufficient condition for $ \omega_{k_h,k_z} = 0 $. 

In summary, we have established the following eigenvalue-eigenvector pairs to \eqref{def:eg-pgm-og}:
\begin{proposition}\label{prop:wave-bases}
There exist three classes of eigenvalue-eigenvector pairs to \eqref{def:eg-pgm-og}: the mean flows, the perturbed internal waves, and the perturbed acoustic waves. They are given as below: with $ k_h \in 2\pi \mathbb Z^2 $ and $ k_z \in 2\pi \mathbb N $,
{\par\noindent\bf Mean flows:} 
$ \omega = 0 $ and the space of mean flows $ \mathfrak E_{0,\varepsilon} $ is given by
\begin{equation}\label{def:mn-flw}
	\begin{aligned}
		& \mathfrak E_{0,\varepsilon}:=  \spn \biggl\lbrace U^\mrm{mf}_{1,(k_h,k_z)} := \left( \begin{array}{c} 0\\ 0 \\ \dfrac{k_h^\perp}{\lvert k_h\rvert } e^{ i k_h\cdot x} \cos (k_z z) \\ 0
		\end{array} \right), k_h \neq (0,0) \biggr\rbrace \\
		& \oplus \spn \biggl\lbrace  U^\mrm{mf}_{2,((0,0),k_z)} := \left( \begin{array}{c}
		\cos( k_z z) \\  \dfrac{k_z}{\eta} \sin( k_z z) \\ 0 \\ 0
		\end{array} \right), \\
		&   U^\mrm{mf}_{\mrm j,((0,0),k_z)} :=\left( \begin{array}{c}
			0 \\ 0 \\ \cos ( k_z z) \vec e_{\mrm j - 2} \\ 0
\end{array} \right), ~~\vec e_1 = \left(\begin{array}{c}
			1 \\ 0
		\end{array}\right), ~ \vec e_2 =  \left(\begin{array}{c}
			0 \\1 
		\end{array}\right), ~ \mrm j = 3,4   \biggr\rbrace.
	\end{aligned}
\end{equation}
%where $ \vec{e}_x = \biggl(\begin{array}{c}
%	1 \\ 0 
%\end{array} \biggr), \, \vec{e}_y = \biggl(\begin{array}{c}
%	0 \\ 1 
%\end{array} \biggr) $.

{\par\noindent\bf Perturbed internal waves:} $ \omega = \pm \omega^{\mrm{gw}}_{(k_h,k_z)} $ where \begin{equation}\label{def:gw-fqcy} \omega^{\mrm{gw}}_{(k_h,k_z)} = \bigl( \dfrac{2\eta^2  \lvert k_h\rvert^2}{  \lvert k_h\rvert^2 + \lvert k_z\rvert^2 + \eta^2 + \sqrt{ A } } \bigr)^{1/2} \end{equation} with $ A = ( \lvert k_h\rvert^2 + \lvert k_z\rvert^2 + \eta^2)^2- 4\eta^2 \lvert k_h\rvert^2 $, $ k_h \neq 0, k_z \neq 0 $, and the space of internal waves $ \mathfrak E_{\pm\omega^{\mrm{gw}}_{(k_h,k_z)},\varepsilon} $ is given by
\begin{equation}\label{def:g-w}
	\begin{aligned}
		& \mathfrak E_{\pm\omega^{\mrm{gw}}_{(k_h,k_z)},\varepsilon} : =  \\
		 &  \spn \biggl\lbrace U^\mrm{gw}_{\pm,(k_h,k_z)} := \left( \begin{array}{c}
			e^{ i k_h \cdot x} \cos( k_z z) \\ \dfrac{\eta}{ k_z} \bigl( \dfrac{ \lvert k_h\rvert^2}{(\omega_{(k_h,k_z)}^\mrm{gw})^2} - 1 \bigr) e^{ i k_h\cdot x} \sin( k_z z) \\ \pm \dfrac{1}{ \omega_{(k_h,k_z)}^\mrm{gw}} k_h e^{ i k_h \cdot x} \cos ( k_z z) \\ \mp i \dfrac{\omega_{(k_h,k_z)}^\mrm{gw}}{ k_z} \bigl( \dfrac{ \lvert k_h\rvert^2}{\lvert\omega_{(k_h,k_z)}^\mrm{gw}\rvert^2} - 1 \bigr) e^{ i k_h \cdot x} \sin( k_z z)
		\end{array} \right),
		\\
		& \qquad\qquad\qquad k_h \neq (0,0), k_z \neq 0
		 \biggr\rbrace.
	\end{aligned}
\end{equation}

{\par\noindent\bf Perturbed acoustic waves:} $ \omega = \pm \omega^\mrm{aw}_{(k_h,k_z)} $ where
\begin{equation}\label{def:aw-fqcy}
\omega^\mrm{aw}_{(k_h,k_z)} = \bigl( \dfrac{ \lvert k_h\rvert^2 + \lvert k_z\rvert^2 + \eta^2 + \sqrt{ A }}{2} \bigr)^{1/2}
\end{equation}
with $ A $ as above, $ (k_h, k_z) \neq ((0,0),0) $, and the space of perturbed acoustic waves $ \mathfrak E_{\pm \omega^\mrm{aw}_{(k_h,k_z)},\varepsilon} $ is given by
\begin{equation}\label{def:prb-aw}
	\begin{aligned}
		& \mathfrak E_{\pm \omega^\mrm{aw}_{(k_h,k_z)},\varepsilon} := \\
		&  \spn\biggl\lbrace U^\mrm{aw}_{\pm,(k_h,k_z)} := \left( \begin{array}{c}
			e^{ i k_h \cdot x} \cos( k_z z) \\ \dfrac{\eta}{ k_z} \bigl( \dfrac{ \lvert k_h\rvert^2}{(\omega_{(k_h,k_z)}^\mrm{aw})^2} - 1 \bigr) e^{ i k_h\cdot x} \sin( k_z z) \\ \pm \dfrac{1}{ \omega_{(k_h,k_z)}^\mrm{aw}} k_h e^{ i k_h \cdot x} \cos ( k_z z) \\ \mp i \dfrac{\omega_{(k_h,k_z)}^\mrm{aw}}{ k_z} \bigl( \dfrac{ \lvert k_h\rvert^2}{\lvert\omega_{(k_h,k_z)}^\mrm{aw}\rvert^2} - 1 \bigr) e^{ i k_h \cdot x} \sin( k_z z)
		\end{array} \right), \\
		& \qquad  k_z \neq 0
		 \biggr\rbrace \oplus
%		  , \qquad \text{or in the case when $ k_z = 0 $,}\\
%		 & \mathfrak E_{\pm \omega^\mrm{aw}_{(k_h,0)},\varepsilon} :=  
		 \spn \biggl\lbrace U^\mrm{aw}_{\pm,(k_h,0)} := \left( \begin{array}{c}
			e^{ i k_h \cdot x}  \\ 0 \\ \pm \dfrac{1}{ \omega_{(k_h,0)}^\mrm{aw}} k_h e^{ i k_h \cdot x} \\ 0
		\end{array} \right), k_h \neq (0,0) \biggr\rbrace.
	\end{aligned}
\end{equation}
Here $ \eta = \varepsilon^{1-\nu} $. 

Moreover, since $ \mathcal L_\varepsilon $ is anti-symmetric, it is easy to check
	\begin{gather*}
		U_{n,(k_h,k_z)}^\mrm{mf}, \quad U_{\pm,(k_h,k_z)}^\mrm{gw}, \quad U_{\pm,(k_h,k_z)}^\mrm{aw}, \quad n \in \lbrace 1, 2, 3, 4 \rbrace, \\
		\quad (k_h,k_z) \in 2\pi \mathbb Z^2 \times 2\pi \mathbb N,
	\end{gather*}
	form orthogonal basis with respect to the complex $ L^2 $-inner product.
%its duality is given by $ \mathcal L_\varepsilon^* : = - \mathcal L_\varepsilon $. Therefore, the corresponding conjugate eigenvectors are given as
%\begin{equation}\label{def:eg-vct-cjg}
%	\begin{gathered}
%		U_n^{\mrm{mf},*} = U_n^\mrm{mf}, \quad n = 1,2,3,4, \\
%		U_\pm^{\mrm{gw},*} = U_\pm^\mrm{gw}, \quad U_{m,\pm}^{\mrm{aw},*} = U_{m,\pm}^\mrm{aw}, \quad m = 1,2.
%	\end{gathered}
%\end{equation}
\end{proposition}

\subsection{Internal waves in the soundproof model \eqref{eq:sndprf}}\label{sec:gw-sp}

We have already known that in the full compressible system \eqref{eq:rf-EEq}, the internal waves bear frequencies of order $ \mathcal O(\varepsilon^{-\nu}) $ from previous sections (see, e.g., Lemma \ref{lm:est-of-eg} and Proposition \ref{prop:wave-bases}). In this subsection, we would like to investigate the internal gravity waves in the soundproof model \eqref{eq:sndprf}, and provide a comparison study with those in system \eqref{eq:rf-EEq}. 

Denote by
\begin{equation}\label{def:vec-opt-sp}
	U_\mrm{sp} := \left( \begin{array}{c}
		\widetilde{\mathcal H}_\mrm{sp} \\ v_\mrm{sp} \\ w_\mrm{sp}
	\end{array} \right), \qquad \text{and} \qquad \mathcal L_\mrm{sp} U_\mrm{sp} := \left( \begin{array}{c}
		- w_\mrm{sp} \\ 0 \\ \widetilde{\mathcal H}_\mrm{sp}
	\end{array} \right).
\end{equation}
Then we introduce the linear system associated with the soundproof model \eqref{eq:sndprf} as follows:
\begin{equation}\label{eq:linear-osc-sp}
	\dt U_\mrm{sp} + \dfrac{1}{\varepsilon^\nu} \mathcal L_\mrm{sp} U_\mrm{sp} + \left( \begin{array}{c}
		0 \\ \nablah p_\mrm{sp} \\ \dz p_\mrm{sp}
	\end{array} \right) = 0, \qquad \dvh v_\mrm{sp} + \dz v_\mrm{sp} = 0, 
\end{equation}
with $ p_\mrm{sp}, \widetilde{\mathcal H}_\mrm{sp}, v_\mrm{sp}, w_\mrm{sp} $ satisfying the same symmetries as $ \widetilde q, \widetilde{\mathcal H}, v, w $, respectively, as in \eqref{SYM}. We consider the following eigenvalue problem:
\begin{equation}\label{def:eg-pgm-og-sp}
	i\omega_\mrm{sp} U_\mrm{sp} = \eta \mathcal L_\mrm{sp} U_\mrm{sp} + \left( \begin{array}{c}
		0 \\ \nablah (\varepsilon p_\mrm{sp}) \\ \dz (\varepsilon p_\mrm{sp})
	\end{array} \right), \qquad \dvh v_\mrm{sp} + \dz v_\mrm{sp} = 0. 
\end{equation}
Recalling $ \eta = \varepsilon^{1-\nu} $, our scale of $ \omega_\mrm{sp} $ in \eqref{def:eg-pgm-og-sp} is the same as $ \omega $ in \eqref{def:eg-pgm-og}, for the sake of convenience for comparison. Direct calculation of \eqref{def:eg-pgm-og-sp} leads to the following differential equation:
\begin{equation}\label{eq:eg-pbm-sp}
	( 1 - \dfrac{\omega_\mrm{sp}^2}{\eta^2} ) \deltah (\varepsilon p_\mrm{sp} ) - \dfrac{\omega_\mrm{sp}^2}{\eta^2} \partial_{zz} (\varepsilon p_\mrm{sp} ) = 0.
\end{equation}
It is obvious that \eqref{eq:eg-pbm-sp} changes types according to $ {\omega_\mrm{sp}^2}/{\eta^2} \in \lbrace 0 \rbrace $, or $ (0,1) $, or $ \lbrace 1 \rbrace $, or $ (1,\infty) $, respectively. In particular, when $ {\omega_\mrm{sp}^2}/{\eta^2} \in (1,\infty) $, \eqref{eq:eg-pbm-sp} is a non-degenerate elliptic equation and has only $ 0 $ as the trivial solution. However, when $ {\omega_\mrm{sp}^2}/{\eta^2} \in [0,1] $, unlike \eqref{eq:eg-pbm-aw}, \eqref{eq:eg-pbm-sp} is a (degenerate) hyperbolic-type equation. 

In the rest of this subsection, we shall use the Fourier representations to persuade further investigation. As in \eqref{def:frr-epsn}, let
\begin{equation}\label{def:frr-epsn-sp}
	\begin{gathered}
	\varepsilon p_\mrm{sp} = \sum_{ k_h \in 2\pi \mathbb Z^2, k_z \in 2\pi \mathbb N} P_{\mrm{sp},({k}_h,k_z)}e^{ i {k}_h \cdot x} \cos( k_z z),
	\\
	U_\mrm{sp} = \sum_{ k_h \in 2\pi \mathbb Z^2, k_z \in 2\pi \mathbb N} \left(
	\begin{array}{c}
		 H_{\mrm{sp},({k}_h,k_z)} e^{ i {k}_h \cdot x} \sin( k_z z) \\ 
		V_{\mrm{sp},({k}_h,k_z)} e^{ i {k}_h \cdot x} \cos( k_z z) \\   W_{\mrm{sp},({k}_h,k_z)} e^{ i {k}_h \cdot x} \sin( k_z z)
	\end{array}
	\right),
	\end{gathered}
\end{equation}
with
$$
F_{\mrm{sp},(-k_h,k_z)} = F_{\mrm{sp},(k_h,k_z)}, \qquad F \in \lbrace P, H, V, W \rbrace. 
$$
Without loss of generality, we also assume that $ P_{\mrm{sp},(0,0)} = 0 $.

Then \eqref{def:eg-pgm-og-sp} is equivalent to
\begin{equation}\label{eq:frr-eq-pgm-sp}
	\begin{aligned}
		i \omega_{\mrm{sp},(k_h,k_z)} H_{\mrm{sp},(k_h,k_z)} & = - \eta W_{\mrm{sp},(k_h,k_z)}, \\
		i \omega_{\mrm{sp},(k_h,k_z)} V_{\mrm{sp},(k_h,k_z)} & = i  P_{\mrm{sp},(k_h,k_z)}k_h, \\
		i \omega_{\mrm{sp},(k_h,k_z)} W_{\mrm{sp},(k_h,k_z)} & = \eta H_{\mrm{sp},(k_h,k_z)} - k_z P_{\mrm{sp},(k_h,k_z)}, \\
		i k_h \cdot V_{\mrm{sp},(k_h,k_z)} + k_z W_{\mrm{sp},(k_h,k_z)} & = 0,
	\end{aligned}
\end{equation} 
and \eqref{eq:eg-pbm-sp} is equivalent to
\begin{equation}\label{eq:frr-eq-pgm-sp-2}
	\biggl( ( 1 - \dfrac{\omega_{\mrm{sp},(k_h,k_z)}^2}{\eta^2} ) \lvert k_h\rvert^2 - \dfrac{\omega_{\mrm{sp},(k_h,k_z)}^2}{\eta^2} \lvert k_z\rvert^2 \biggr) P_{\mrm{sp},(k_h,k_z)} = 0.
\end{equation}

{\noindent \bf Case 1:} $ P_{\mrm{sp},(k_h,k_z)} = 0 $. Then it is easy to verify that, the nontrivial solutions to \eqref{eq:frr-eq-pgm-sp} are given by
\begin{equation*}
	\begin{gathered}
	\lvert\omega_{\mrm{sp},(k_h,0)}\rvert = \eta, \quad
	\left( \begin{array}{c}
	P_{\mrm{sp},(k_h,0)}\\
		H_{\mrm{sp},(k_h,0)} \\ V_{\mrm{sp},(k_h,0)} \\ W_{\mrm{sp},(k_h,0)} 
	\end{array} \right) = \left( \begin{array}{c}
	0 \\
		H_{\mrm{sp},(k_h,0)} \\ 0 \\ - i \dfrac{\omega_{\mrm{sp},(k_h,0)}}{\eta} H_{\mrm{sp},(k_h,0)} 
	\end{array} \right), \\
	\text{or} \qquad
	\omega_{\mrm{sp},(k_h,k_z)} = 0, \quad \left( \begin{array}{c}
	P_{\mrm{sp},(k_h,k_z)} \\
		H_{\mrm{sp},(k_h,k_z)} \\ V_{\mrm{sp},(k_h,k_z)} \\ W_{\mrm{sp},(k_h,k_z)} 
	\end{array} \right) =  \left( \begin{array}{c}
	0 \\
		0 \\ V_{\mrm{sp},(k_h,k_z)} \\ 0 
	\end{array} \right) \\
	\qquad \text{with}\quad   k_h \cdot V_{\mrm{sp},(k_h,k_z)} = 0. 
	\end{gathered}
\end{equation*}
Next, we focus on the cases when $ P_{\mrm{sp},(k_h,k_z)} \neq 0 $. Then it must hold, from \eqref{eq:frr-eq-pgm-sp-2},
\begin{equation}\label{eq:frr-eq-pgm-sp-3}
	( 1 - \dfrac{\omega_{\mrm{sp},(k_h,k_z)}^2}{\eta^2} ) \lvert k_h\rvert^2 - \dfrac{\omega_{\mrm{sp},(k_h,k_z)}^2}{\eta^2} \lvert k_z\rvert^2 = 0.
\end{equation}
{\noindent\bf Case 2:} $ k_h = (0,0) $. Solving \eqref{eq:frr-eq-pgm-sp-3} leads to either $ \omega_{\mrm{sp}, (0,k_z)} = 0 $ or $ k_z = 0 $. Then the nontrivial solutions to \eqref{eq:frr-eq-pgm-sp} are given by
\begin{equation*}
	\begin{gathered}
		\omega_{\mrm{sp}, (0,k_z)} = 0, \quad \left( \begin{array}{c}
	P_{\mrm{sp},(0,k_z)} \\
		H_{\mrm{sp},(0,k_z)} \\ V_{\mrm{sp},(0,k_z)} \\ W_{\mrm{sp},(0,k_z)} 
	\end{array} \right) =  \left( \begin{array}{c}
	P_{\mrm{sp},(0,k_z)} \\
		\dfrac{k_z}{\eta} P_{\mrm{sp},(0,k_z)} \\ V_{\mrm{sp},(0,k_z)} \\ 0 
	\end{array} \right),\\
	\text{or} \qquad \lvert\omega_{\mrm{sp},(0,0)}\rvert = \eta, \quad 
	\left( \begin{array}{c}
	P_{\mrm{sp},(0,0)} \\
		H_{\mrm{sp},(0,0)} \\ V_{\mrm{sp},(0,0)} \\ W_{\mrm{sp},(0,0)} 
	\end{array} \right) = \left( \begin{array}{c}
		P_{\mrm{sp},(0,0)} \\ H_{\mrm{sp},(0,0)}  \\ 0 \\ -i \dfrac{\omega_{\mrm{sp},(0,0)}}{\eta} H_{\mrm{sp},(0,0)}
	\end{array} \right).
	\end{gathered}
\end{equation*}

{\noindent\bf Case 3:} $ k_h \neq (0,0) $. Solving \eqref{eq:frr-eq-pgm-sp-3} leads to
\begin{equation*}
	\dfrac{\omega_{\mrm{sp},(k_h,k_z)}^2}{\eta^2} = \dfrac{\lvert k_h\rvert^2}{\lvert k_h\rvert^2 + \lvert k_z\rvert^2}.
\end{equation*}
Then solving \eqref{eq:frr-eq-pgm-sp} yields
\begin{equation*}
	\begin{gathered}
		\lvert\omega_{\mrm{sp},(k_h,0)}\rvert = \eta, \quad  \left( \begin{array}{c}
	P_{\mrm{sp},(k_h,0)}\\
		H_{\mrm{sp},(k_h,0)} \\ V_{\mrm{sp},(k_h,0)} \\ W_{\mrm{sp},(k_h,0)} 
	\end{array} \right)  =  \left( \begin{array}{c}
	0 \\
		H_{\mrm{sp},(k_h,0)} \\ 0 \\ - i \dfrac{\omega_{\mrm{sp},(k_h,0)}}{\eta} H_{\mrm{sp},(k_h,0)} 
	\end{array} \right) \\
	\quad (\text{discarded}), \\
		\text{or} \quad k_z \neq 0, \quad \lvert \omega_{\mrm{sp},(k_h,k_z)}\rvert = \dfrac{\eta \lvert k_h\rvert}{\sqrt{\lvert k_h\rvert^2 + \lvert k_z\rvert^2}}, \\
		\left( \begin{array}{c}
	P_{\mrm{sp},(k_h,k_z)}\\
		H_{\mrm{sp},(k_h,k_z)} \\ V_{\mrm{sp},(k_h,k_z)} \\ W_{\mrm{sp},(k_h,k_z)} 
	\end{array} \right) = \left( \begin{array}{c}
	P_{\mrm{sp},(k_h,k_z)}  \\
		\dfrac{\eta\lvert k_h\rvert^2}{k_z \omega_{\mrm{sp},(k_h,k_z)}^2} P_{\mrm{sp},(k_h,k_z)} \\ \dfrac{P_{\mrm{sp},(k_h,k_z)}}{\omega_{\mrm{sp},(k_h,k_z)}}k_h \\ - i \dfrac{\lvert k_h\rvert^2}{k_z\omega_{\mrm{sp},(k_h,k_z)}} P_{\mrm{sp},(k_h,k_z)} 
	\end{array} \right).
	\end{gathered}
\end{equation*}

In summary, we have established the following eigenvalue-eigenvector pairs to \eqref{def:eg-pgm-og-sp}:
\begin{proposition}\label{prop:wave-bases-sp}
The mean flows and the internal waves in the eigenvalue problem  \eqref{def:eg-pgm-og-sp} for the soundproof model are given as below: with $ k_h \in 2\pi \mathbb Z^2 $ and $ k_z \in 2\pi \mathbb N $, 
{\par\noindent\bf Mean flows:} $ \omega_\mrm{sp} = 0 $ and the space of mean flows $ \mathfrak E_{\mrm{sp},0,\varepsilon} $ is given by
\begin{equation}\label{def:mn-flw-sp}
	\begin{aligned}
		& \mathfrak E_{\mrm{sp},0,\varepsilon} := \spn\biggl\lbrace \varepsilon p_{\mrm{sp},1,(k_h,k_z)}^\mrm{mf} := 0, \\
		& \qquad\qquad  U_{\mrm{sp},1,(k_h,k_z)}^\mrm{mf}: = \left(\begin{array}{c}
			0 \\ \dfrac{k_h^\perp}{\lvert k_h\rvert} e^{ik_h\cdot x} \cos(k_z z) \\ 0
		\end{array} \right), k_h \neq (0,0) \biggr\rbrace \\
		&\oplus \spn\biggl\lbrace \varepsilon p_{\mrm{sp},2,((0,0),k_z) }^\mrm{mf}: = \cos(k_z z), U_{\mrm{sp},2,((0,0),k_z)}^\mrm{mf} := \left(\begin{array}{c}
		\dfrac{k_z}{\eta} \sin(k_z z) \\ 0 \\ 0
		\end{array}\right)
		\biggr\rbrace \\
		& \oplus \spn \biggl\lbrace \varepsilon p_{\mrm{sp},\mrm j,((0,0),k_z)}^\mrm{mf} := 0, U_{\mrm{sp},\mrm j,((0,0),k_z)}^\mrm{mf} := \left( \begin{array}{c}
			0 \\  \cos(k_z z) \vec{e}_{\mrm j - 2} \\ 0
		\end{array}\right), \\
		& \qquad\qquad \vec e_1 = \left(\begin{array}{c}
			1 \\ 0
		\end{array}\right), \quad \vec e_2 = \left(\begin{array}{c}
			0 \\1 
		\end{array}\right), ~ \mrm j = 3,4 \biggr\rbrace 
		.
	\end{aligned}
\end{equation}

{\par\noindent\bf Internal waves}: $ \omega_\mrm{sp} = \pm \omega_{\mrm{sp},(k_h,k_z)}^\mrm{gw} $ where
\begin{equation}\label{def:gw-fqcy-sp}
	\omega_{\mrm{sp},(k_h,k_z)}^\mrm{gw} := \dfrac{\eta\lvert k_h\rvert}{(\lvert k_h\rvert^2 + \lvert k_z\rvert^2)^{1/2}},
\end{equation}
with $ k_h \neq (0,0) $ and $ k_z \neq 0 $,
and the space of internal waves $ \mathfrak E_{\mrm{sp},\pm\omega^\mrm{gw}_{\mrm{sp},(k_h,k_z)},\varepsilon} $ is given by
\begin{equation}\label{def:g-w-sp}
	\begin{aligned}
		& \mathfrak E_{\mrm{sp},\pm\omega^\mrm{gw}_{\mrm{sp},(k_h,k_z)},\varepsilon} : = \spn \biggl\lbrace \varepsilon p_{\mrm{sp},(k_h,k_z)}^\mrm{gw}: = e^{ik_h\cdot x} \cos (k_z z), \\
		& \qquad\qquad U_{\mrm{sp},\pm,(k_h,k_z)}^\mrm{gw}:=\left(\begin{array}{c}
			\dfrac{\eta\lvert k_h\rvert^2}{k_z \lvert \omega_{\mrm{sp},(k_h,k_z)}^\mrm{gw}\rvert^2} e^{ik_h\cdot x} \sin(k_z z) \\
			\pm \dfrac{1}{\omega_{\mrm{sp},(k_h,k_z)}^\mrm{gw}} k_h e^{ik_h\cdot x} \cos(k_z z) \\
			\mp i \dfrac{\lvert k_h\rvert^2}{k_z\omega_{\mrm{sp},(k_h,k_z)}^\mrm{gw}} e^{ik_h\cdot x} \sin(k_z z) 
		\end{array} \right), \\
		& \qquad \qquad k_h \neq (0,0), k_z \neq 0
		\biggr\rbrace.
	\end{aligned}
\end{equation}
Here $ \eta = \varepsilon^{1-\nu} $. 
\end{proposition}

\subsection{Comparison with limit cases}

%In this subsection, we will quantitatively analyze the differences between the eigenvalue-eigenvector pairs obtained in Proposition \ref{prop:wave-bases}. \chgd{uncommment changes!}

{In this subsection, we will quantitatively compare the linear dynamics of several reduced systems studied in this work. (i) We compare the internal waves between the full compressible and the pseudo-incompressible models as derived in propositions \ref{prop:wave-bases} and \ref{prop:wave-bases-sp} above. (ii) We compare the acoustic waves generated by the pure acoustic operator in \eqref{def:eg-pgm-ac} alone and by the full fast mode operator of the compressible system.}

{First, we summarize the eigenvalue-eigenvector pair of the pure acoustic system,}
\begin{equation}\label{def:eg-pgm-ac}
	i \omega_\mrm{a} U_\mrm{a} = \mathcal L_a U_\mrm{a}, \qquad U_\mrm{a} = (\widetilde q_\mrm{a}, \widetilde{\mathcal H}_\mrm{a}, v_\mrm{a}, w_\mrm{a} )^\top \in \mathfrak V.
\end{equation}
That is,
\begin{lemma}\label{lm:acoustic-wave-bases}
	There exist two classes of eigenvalue-eigenvector pairs to \eqref{def:eg-pgm-ac}; the incompressible flows and the acoustic waves. They are given as below: with $ k_h \in 2\pi \mathbb Z^2 $ and $ k_z \in 2\pi \mathbb N $,
	{\par\noindent\bf Incompressible flows:} $ \omega_\mrm{a} = 0 $ and the space of incompressible flows $ \mathfrak E_{\mrm{a}, 0} $ is given by
	\begin{equation}\label{def:icf-flw}
		\begin{aligned}
			& \mathfrak E_{\mrm{a},0} := \spn \biggl\lbrace U_{\mrm a, 1,((0,0),0) }^\mrm{icf} := \left( \begin{array}{c}
				 1 \\ 0 \\ 0 \\ 0 
			\end{array} \right),	  ~ U_{\mrm a, 2,(k_h,k_z) }^\mrm{icf} := \left( \begin{array}{c}
				 0 \\ e^{ik_h \cdot x} \sin (k_z z) \\ 0 \\ 0 
			\end{array} \right) \biggr\rbrace \\
			& 
			\oplus \biggl\lbrace U_{\mrm a, 3,(k_h,k_z) }^\mrm{icf} := \left( \begin{array}{c}
				0 \\ 0 \\ \dfrac{k_h^\perp}{\lvert k_h\rvert} e^{ik_h\cdot x} \cos(k_z z) \\ 0 
			\end{array}\right), \\
			& \qquad  U_{\mrm a, 4,(k_h,k_z) }^\mrm{icf} := \left( \begin{array}{c}
				0 \\ 0 \\ k_z \dfrac{k_h}{\lvert k_h\rvert} e^{ik_h\cdot x} \cos(k_z z) \\ - i \lvert k_h\rvert e^{ik_h\cdot x} \sin(k_z z)
			\end{array}\right), ~ k_h \neq (0,0) \biggr\rbrace\\
			& \oplus \biggl\lbrace U_{\mrm a, 5,((0,0),k_z)}^\mrm{icf} := \left(\begin{array}{c}
				0 \\ 0 \\ \cos(k_z z) \vec e_h \\ 0
			\end{array}\right), ~ \vec e_h = \left(\begin{array}{c}
			1 \\ 0
		\end{array}\right) ~ \text{or} ~ \left(\begin{array}{c}
			0 \\1 
		\end{array}\right) 
			\biggr\rbrace .
		\end{aligned}
	\end{equation}
	{\par\noindent\bf Acoustic waves:} $ \omega_\mrm{a} = \pm \omega_{\mrm{a},(k_h,k_z)}^\mrm{aw} $ where
	\begin{equation}\label{def:ac-fqcy}
		\omega_{\mrm{a},(k_h,k_z)}^\mrm{aw} := (\lvert k_h\rvert^2 + \lvert k_z\rvert^2 )^{1/2}, \qquad (k_h,k_z) \neq ((0,0),0),
	\end{equation}
	and the space of acoustic waves $ \mathfrak E_{\mrm a,\pm\omega_{\mrm{a},(k_h,k_z)}^\mrm{aw}} $ is given by
	\begin{equation}\label{def:ac-flw}
		\begin{aligned}
			& \mathfrak E_{\mrm a,\pm\omega_{\mrm{a},(k_h,k_z)}^\mrm{aw}} := \spn \biggl\lbrace U^\mrm{aw}_{\mrm a, \pm, (k_h,k_z)} := \left( \begin{array}{c}
				e^{ik_h\cdot x} \cos(k_z z) \\ 0 \\ \pm \dfrac{k_h}{\omega_{\mrm{a},(k_h,k_z)}^\mrm{aw}}e^{ik_h\cdot x} \cos(k_z z)  \\ \pm \dfrac{ik_z}{\omega_{\mrm{a},(k_h,k_z)}^\mrm{aw}}e^{ik_h\cdot x} \sin(k_z z)
			\end{array} \right), \\
			& \qquad\qquad (k_h,k_z) \neq ((0,0),0) 
			\biggr\rbrace.
		\end{aligned}
	\end{equation}
\end{lemma}

In the following, we will compare the eigenvalue-eigenvector pairs obtained in Proposition \ref{prop:wave-bases-sp} and Lemma \ref{lm:acoustic-wave-bases} with those in Proposition \ref{prop:wave-bases}.

{\par\noindent\bf Perturbed acoustic waves v.s. acoustic waves,} i.e., $ (\omega^\mrm{aw}_{(k_h,k_z), \varepsilon}, U^\mrm{aw}_{\pm,(k_h,k_z)} ) $ v.s. $ (\omega^\mrm{aw}_{\mrm a, (k_h,k_z), \varepsilon}, U^\mrm{aw}_{\mrm a,\pm,(k_h,k_z)} ) $: Direct calculation, from \eqref{def:aw-fqcy} and \eqref{def:ac-fqcy}, shows that, for $ (k_h, k_z) \neq ((0,0),0) $, 
\begin{equation}\label{cmp-aw-fqcy}
	\omega^\mrm{aw}_{(k_h,k_z)} = \omega^\mrm{aw}_{\mrm a, (k_h,k_z)} + \eta^2 \cdot \dfrac{\lvert k_z\rvert^2}{2 (\lvert k_h\rvert^2 + \lvert k_z\rvert^2 )^{3/2}} + \mathcal O(\eta^4).
\end{equation}
Meanwhile, owing to \eqref{def:prb-aw} and \eqref{def:ac-flw}, one has, for $ k_z \neq 0 $, 
\begin{equation}\label{cmp-aw-flw-1}
	\begin{aligned}
	& U^\mrm{aw}_{\pm, (k_h,k_z)} - U^\mrm{aw}_{\mrm a, \pm, (k_h,k_z)} \\
	&  = \left(\begin{array}{c}
		0 \\ \dfrac{\eta}{k_z} \dfrac{\lvert k_h \rvert^2 - (\omega^\mrm{aw}_{(k_h,k_z)})^2 }{(\omega^\mrm{aw}_{(k_h,k_z)})^2} e^{ik_h\cdot x} \sin (k_z z) \\ \pm k_h \dfrac{\omega^\mrm{aw}_{\mrm a, (k_h,k_z)}-\omega^\mrm{aw}_{(k_h,k_z)}}{\omega^\mrm{aw}_{(k_h,k_z)}\omega^\mrm{aw}_{\mrm a, (k_h,k_z)}} e^{ik_h\cdot x} \cos (k_z z) \\ \pm i \dfrac{(\omega^\mrm{aw}_{(k_h,k_z)})^2 \omega^\mrm{aw}_{\mrm a, (k_h,k_z)} - \lvert k_h\rvert^2 \omega^\mrm{aw}_{\mrm a, (k_h,k_z)} - \lvert k_z\rvert^2 \omega^\mrm{aw}_{(k_h,k_z)} }{k_z \omega^\mrm{aw}_{(k_h,k_z)}\omega^\mrm{aw}_{\mrm a, (k_h,k_z)} } e^{ik_h\cdot x} \sin(k_z z)
	\end{array} \right) \\
	& =  \left( \begin{array}{c}
		0 \\ - \biggl( \eta \dfrac{k_z}{\lvert k_h\rvert^2 + \lvert k_z\rvert^2} + \eta^3 \dfrac{k_z \lvert k_h\rvert^2 }{(\lvert k_h\rvert^2 + \lvert k_z\rvert^2)^3} \biggr) e^{ik_h\cdot x} \sin (k_z z) + \mathcal O(\eta^5 ) \\ \mp \eta^2 \cdot \dfrac{k_h \lvert k_z\rvert^2 }{2 ( \lvert k_h\rvert^2 + \lvert k_z\rvert^2 )^{5/2}} e^{ik_h\cdot x} \cos (k_z z) + \mathcal O (\eta^4) \\ \pm \eta^2 \cdot i \dfrac{k_z ( 2 \lvert k_h\rvert^2 + \lvert k_z\rvert^2 )}{2 (\lvert k_h\rvert^2 + \lvert k_z\rvert^2 )^{5/2}} e^{ik_h\cdot x} \sin (k_z z) + \mathcal O(\eta^4)
	\end{array} \right); 
	\end{aligned}
\end{equation}
for $ k_z = 0 $, $ k_h \neq (0,0) $,
\begin{equation}\label{cmp-aw-flw-2}
	\begin{aligned}
		& U^\mrm{aw}_{\pm, (k_h,0)} - U^\mrm{aw}_{\mrm a, \pm, (k_h,0)} = \left(\begin{array}{c}
			0 \\0 \\ \mathcal O(\eta^4)) \\ 0
		\end{array} \right).
	\end{aligned}
\end{equation}

{\par\noindent\bf Perturbed internal waves v.s. internal waves,} i.e., $ (\omega^\mrm{gw}_{(k_h,k_z)}, U_{\pm, (k_h,k_z)}^\mrm{gw} ) $ v.s. $ (\omega^\mrm{gw}_{\mrm{sp}, (k_h,k_z)}, \left( \begin{array}{c} \varepsilon p_{\mrm{sp},(k_h,k_z)}^\mrm{gw} \\ U_{\mrm{sp},\pm, (k_h,k_z)}^\mrm{gw} \end{array} \right) ) $: Direct calculation, from \eqref{def:gw-fqcy} and \eqref{def:gw-fqcy-sp}, shows that, for $ k_h \neq (0,0) $, $ k_z \neq 0 $,
\begin{equation}\label{cmp-gw-fqcy}
	\dfrac{\omega_{(k_h,k_z)}^\mrm{gw}}{\eta} = \dfrac{\omega_{\mrm{sp}, (k_h,k_z)}^\mrm{gw}}{\eta} - \eta^2 \cdot \dfrac{\lvert k_h\lvert \lvert k_z\rvert^2 }{2 (\lvert k_h\rvert^2 + \lvert k_z\rvert^2)^{5/2}} + \mathcal O(\eta^4). 
\end{equation}
Meanwhile, owing to \eqref{def:g-w} and \eqref{def:g-w-sp}, one has, 
\begin{equation}\label{cmp-gw-flw}
	\begin{aligned}
		& U_{\pm,(k_h,k_z)}^\mrm{gw} - \left( \begin{array}{c} \varepsilon p_{\mrm{sp},(k_h,k_z)}^\mrm{gw} \\ U_{\mrm{sp},\pm, (k_h,k_z)}^\mrm{gw} \end{array} \right) \\
		& = \left(\begin{array}{c}
			0 \\ \dfrac{\eta}{k_z} \biggl( \dfrac{\lvert k_h\rvert^2}{\lvert\omega^\mrm{gw}_{(k_h,k_z)}\rvert^2} - \dfrac{\lvert k_h\rvert^2}{\lvert\omega^\mrm{gw}_{\mrm{sp}, (k_h,k_z)}\rvert^2}  - 1 \biggr) e^{ik_h\cdot x}\sin(k_z z) \\
			\pm \biggl( \dfrac{k_h}{\omega^\mrm{gw}_{(k_h,k_z)}} - \dfrac{k_h}{\omega^\mrm{gw}_{\mrm{sp}, (k_h,k_z)}} \biggr) e^{ik_h\cdot x} \cos (k_z z) \\
			\mp i \biggl\lbrack \dfrac{\lvert k_h\rvert^2}{k_z} \biggl( \dfrac{1}{\omega^\mrm{gw}_{(k_h,k_z)}} - \dfrac{1}{\omega^\mrm{gw}_{\mrm{sp}, (k_h,k_z)}} \biggr) -  \dfrac{\omega^\mrm{gw}_{(k_h,k_z)}}{k_z}  \biggr\rbrack e^{ik_h\cdot x} \sin(k_z z)
		\end{array} \right)\\
		&  = \left( \begin{array}{c}
			0 \\ %\biggl( 
			- \eta \cdot \dfrac{\lvert k_h\rvert^2}{k_z(\lvert k_h\rvert^2 + \lvert k_z\rvert^2)} % + \eta^3 \dfrac{k_z \lvert k_h\rvert^2}{(\lvert k_h\rvert^2 + \lvert k_z\rvert^2)^3}\biggr)
			 e^{ik_h\cdot x} \sin(k_z z) + \mathcal O(\eta^3)\\
			 \pm  \eta \cdot \dfrac{\lvert k_z\rvert^2 k_h}{2\lvert k_h\lvert(\lvert k_h\rvert^2 + \lvert k_z\rvert^2 )^{3/2} } e^{ik_h\cdot x} \cos (k_z z) + \mathcal O(\eta^3)  \\
			\pm i \eta \cdot \dfrac{\lvert k_h\lvert(2\lvert k_h\rvert^2 + \lvert k_z\rvert^2 )}{2 k_z (\lvert k_h\rvert^2 + \lvert k_z\rvert^2 )^{3/2}} e^{ik_h\cdot x} \sin (k_z z) + \mathcal O(\eta^3) 
		\end{array} \right).
	\end{aligned}
\end{equation}

{\par\noindent\bf Mean flows:} It is obvious that $ U_{j,(k_h,k_z)}^\mrm{mf} = \left( \begin{array}{c}
	\varepsilon p_{\mrm{sp},j,(k_h,k_z)}^\mrm{mf} \\ U_{\mrm{sp},j,(k_h,k_z)}^\mrm{mf}
\end{array} \right) $, $ j = 1,2,3 $. 

In summary, we have proved the following:
\begin{corollary}\label{cor:cmprs-gw}
%	With the same scale as in Propositions \ref{prop:wave-bases} and \ref{prop:wave-bases-sp}, we have:
For $ (k_h,k_z) $ satisfying the corresponding restrictions, one has
	\begin{gather}
		\begin{gathered}
		\mathfrak E_{0,\varepsilon} \equiv \mathfrak E_{\mrm{sp},0,\varepsilon} \quad \text{or equivalently}\\ \quad U_{\mrm{j},(k_h,k_z)}^\mrm{mf} \equiv \biggl(\begin{array}{c}
			\varepsilon p_{\mrm{sp},\mrm j, (k_h,k_z)}^\mrm{mf} \\ U^\mrm{mf}_{\mrm{sp},\mrm j,(k_h,k_z)}
		\end{array} \biggr), \quad \mrm j = 1,2,3,4, 
		\end{gathered}
		\label{sm-mf} \\
		0 <   \omega_{\mrm{sp},(k_h,k_z)}^\mrm{gw} - \omega_{(k_h,k_z)}^\mrm{gw} = \mathcal O(\eta^3), \label{prt-fq-gw} \\
		\biggl\lvert U_{\pm,(k_h,k_z)}^\mrm{gw} - \biggl( \begin{array}{c}
 		\varepsilon p_{\mrm{sp},(k_h,k_z)}^{\mrm{gw}} \\
 		U_{\mrm{sp},\pm,(k_h,k_z)}^\mrm{gw}
		 \end{array} \biggr) \biggr\lvert = \mathcal O(\eta), \label{prt-bs-gw} \\
		 0 < \omega^\mrm{aw}_{(k_h,k_z)} - \omega^\mrm{aw}_{\mrm a, (k_h,k_z)} = \mathcal O(\eta^2), \label{prt-fq-aw} \\
		 \lvert U_{\pm,(k_h,k_z)}^\mrm{aw} - U_{\mrm a, \pm,(k_h,k_z)}^\mrm{aw} \lvert = \mathcal O(\eta), \label{prt-bs-aw}
	\end{gather}
	uniformly in $ (k_h,k_z) $. 
	Here $ \eta = \varepsilon^{1-\nu} $. 
\end{corollary}

\section{Fast-slow waves interactions: Soundproof approximation with ill-prepared initial data}\label{sec:fast-slow-nonlinear}

\subsection{Nonlinear equations}\label{subsec:f-s-eq}

\subsubsection*{The full compressible system}

With the understanding of the linear theory, we will discuss the nonlinear theory of fast-slow waves decompositions in system \eqref{eq:rf-EEq}. Notice that, under assumption \eqref{amtp:atsym-cfts}, \eqref{eq:rf-EEq} can be written as, with $ U = (\widetilde q, \widetilde{\mathcal H}, v,  w)^\top $, $ \mathcal L_\varepsilon $ as in \eqref{def:ptb-aw-opt},
\begin{equation}\label{eq:rf-EEq-dcp-01}
		\dt U + \dfrac{1}{\varepsilon}\mathcal L_\varepsilon U + \mathcal N(U) = \mathcal M(U) + \mathcal K_\varepsilon(U),
\end{equation}
where
\begin{align}
	\label{nonlinearity:full}
	\mathcal N(U) := & v \cdot \nablah U + w \dz U 
	 + \left( \begin{array}{c}
		 \varpi_0^{-1} \widetilde q(\dvh v + \dz w) \\
		- \widetilde G \cdot \widetilde{\mathcal H} w \\
		0 \\ 0
	\end{array} \right), \\
	\mathcal M(U):= & \left( \begin{array}{c}
		 G w + \varpi_0^{-1} \int_0^z G(z') \,dz' (\dvh v + \dz w) \\
		0 \\
		0 \\ 0
	\end{array} \right), \\
	\mathcal K_\varepsilon(U) := & \left( \begin{array}{c}
		\varepsilon^\mu \overline{\mathcal H}_0 w + \varepsilon^\mu \varpi_0^{-1} \int_0^z \overline{\mathcal H}_0(z') \,dz'  (\dvh v + \dz w) \\
		0 \\
		- (\varepsilon^\mu \widetilde G \overline{\mathcal H}_0 + \varepsilon^{\mu+\nu} \widetilde G \widetilde{\mathcal H} ) ( \dt v + v \cdot \nablah v + w \dz v) \\
		- (\varepsilon^\mu \widetilde G \overline{\mathcal H}_0 + \varepsilon^{\mu+\nu} \widetilde G \widetilde{\mathcal H} ) ( \dt w + v \cdot \nablah w + w \dz w)
	\end{array} \right).
\end{align}
%\begin{equation}
%	\begin{cases}
%		\dt \widetilde q + v \cdot \nablah \widetilde q + w \dz \widetilde q + \dfrac{1}{\varepsilon} (\dvh v + \dz w) = - \pi_0^{-1} \widetilde q ( \dvh v + \dz w) \\
%		\qquad + (G+ \varepsilon^\mu \overline{\mathcal H}_0 ) w + \pi_0^{-1}  \int_0^z ( G(z') + \varepsilon^{\mu} \overline{\mathcal H}_0(z') )\,dz' ( \dvh v + \dz w), \\
%		\dt \widetilde{\mathcal H} + v \cdot \nablah \widetilde{\mathcal H} + w \dz \widetilde{\mathcal H} - \dfrac{1}{\varepsilon^\nu} w = \dz \log G \cdot \widetilde{\mathcal H} w, \\
%		\dt v + v \cdot \nablah v + w \dz v + \dfrac{1}{\varepsilon} \nablah \widetilde q \\
%		\qquad = - (\varepsilon^\mu G^{-1} \overline{\mathcal H}_0 + \varepsilon^{\mu+\nu} G^{-1} \widetilde{\mathcal H} ) ( \dt v + v \cdot \nablah v + w \dz v),  \\
%		\dt w + v \cdot \nablah w + w \dz w + \dfrac{1}{\varepsilon} \dz \widetilde q + \dfrac{1}{\varepsilon^\nu} \widetilde{\mathcal H} \\
%		\qquad = - (\varepsilon^\mu G^{-1} \overline{\mathcal H}_0 + \varepsilon^{\mu+\nu} G^{-1} \widetilde{\mathcal H} ) ( \dt w + v \cdot \nablah w + w \dz w).
%	\end{cases}
%\end{equation}
With estimate \eqref{est:uni-total-general} and proper initial data, one can assume that $ \mathcal K_\varepsilon (U) = \mathcal O(\varepsilon^{\mu - \sigma}) $ in suitable Sobolev space ($ H^2 $ for instance). In particular, we choose $ \sigma = \mu/2 $, and thus $ \mathcal K_\varepsilon(U) $ will be considered as an error term. For this reason, we write %will focus on the wave decomposition of
\begin{equation}\label{eq:rf-EEq-dcp-02}
	\dt U + \dfrac{1}{\varepsilon}\mathcal L_\varepsilon U + \mathcal N(U) = \mathcal M(U)  + \mathcal O(\varepsilon^{\mu - \sigma}).
\end{equation}
%instead of \eqref{eq:rf-EEq-dcp-01}. 
%The decomposition of \eqref{eq:rf-EEq-dcp-01} is left to interesting readers. 

\begin{enumerate}[label = {\bf H\arabic{enumi})}, ref = {\bf H\arabic{enumi})}]
\setcounter{enumi}{\value{hypothesis}}
\item \label{def:AssumptionH5} Furthermore, {to simplify the presentation}, we assume
\begin{equation}\label{asmpt:gravity}
	G = \widetilde G = \sin (2\pi z).
\end{equation}
\setcounter{hypothesis}{\value{enumi}}
\end{enumerate}
{We emphasise that with some modification, the following arguments work without assumption \ref{def:AssumptionH5}.}
We will adopt the notation \eqref{def:frr-epsn} for our solutions $ U $.

%, and denote by
%\begin{equation}\label{def:frr-cft}
%	\hat U_{(k_h,k_z)} := ( Q_{(k_h,k_z)}, H_{(k_h,k_z)}, V_{(k_h,k_z)}, W_{(k_h,k_z)} )^\top,
%\end{equation}
%for any $ (k_h,k_z) \in 2\pi \mathbb Z^2 \otimes ( 2\pi \mathbb N ) $. 
%
%We proceed, in the following, to decompose each term of \eqref{eq:rf-EEq-dcp-02} into the spaces of mean flows, internal waves, and perturbed acoustic waves, respectively, using the eigenvalue-eigenvector pairs obtained in Proposition \ref{prop:wave-bases}. 
%
%Denote by $$ \varsigma = \int \cos^2(k_z z) \,dz = \begin{cases}
%	1 & \text{if} \quad k_z = 0, \\
%	\dfrac{1}{2} & \text{if} \quad k_z \neq 0. 
%\end{cases} $$
%

Let
\begin{equation}\label{def:project_op_perturbed}
	\mathcal P^\mrm{mf}_\varepsilon, \qquad \mathcal P^\mrm{gw}_\varepsilon,  \quad \text{and} \quad \mathcal P^\mrm{aw}_\varepsilon,
\end{equation}
be the $ L^2 $-orthogonal projections to the spaces 
\begin{gather*}
	 \mathfrak E_\varepsilon^\mrm{mf} :=\mathfrak E_{0,\varepsilon}, \quad  \mathfrak E_\varepsilon^{\mrm{gw}}:= \oplus_{k_h \neq (0,0), k_z \neq 0} \mathfrak E_{\pm \omega_{k_h,k_z}^\mrm{gw},\varepsilon}, \\ \quad \text{and} \quad  \mathfrak E_\varepsilon^\mrm{aw}:= \oplus_{k_z\neq 0} \mathfrak E_{\pm \omega^\mrm{aw}_{(k_h,k_z)},\varepsilon} \oplus_{k_h \neq (0,0)} \mathfrak E_{\pm \omega^\mrm{aw}_{(k_h,0)},\varepsilon}, 
\end{gather*} 
respectively,  given in Proposition \ref{prop:wave-bases}. 
%Also, denote by
%\begin{equation}\label{def:project_op_perturbed}
%	\mathcal P_\varepsilon^\mrm{mfgw}:= \mathcal P^\mrm{mf}_\varepsilon + \mathcal P^\mrm{gw}_\varepsilon.
%\end{equation}

\subsubsection*{The soundproof system}

Similarly, denote by $ U_\mrm{sp} $ and $ \mathcal L_\mrm{sp} $ as in \eqref{def:vec-opt-sp}. Under assumption \ref{def:AssumptionH4} and {the simplifying but not critical assumption \ref{def:AssumptionH5}}, \eqref{eq:sndprf} can be written as,
\begin{gather}
\dvh v_\mrm{sp} +  \dz w_\mrm{sp} = 0, \\
\dt U_\mrm{sp} + \dfrac{1}{\varepsilon^\nu} \mathcal L_\mrm{sp}U_\mrm{sp} + \left( \begin{array}{c}
	0 \\ \nablah p_\mrm{sp} \\ \dz p_\mrm{sp}
\end{array}\right) + \mathcal N_\mrm{sp}(U_\mrm{sp}) = 0. \label{eq:speq-dcp-01} 	
\end{gather}
Here, thanks to assumption \ref{def:AssumptionH5},
\begin{equation}\label{nonlinearity:sp}
\mathcal N_\mrm{sp}(U_\mrm{sp}) :=  v_\mrm{sp} \cdot \nablah U_\mrm{sp} + w_\mrm{sp} \dz U_\mrm{sp} 
	 + \left( \begin{array}{c}
		- \sin(2\pi z) \cdot \widetilde{\mathcal H}_\mrm{sp} w_\mrm{sp} \\
		0 \\ 0
	\end{array} \right).
\end{equation}

Notice that equations \eqref{eq:rf-EEq-dcp-02} and \eqref{eq:speq-dcp-01} have different dimensions. In particular, \eqref{eq:speq-dcp-01} does not have an evolutionary equation of $ p_\mrm{sp} $, corresponding to the $ \widetilde q $-component of \eqref{eq:rf-EEq-dcp-02}. For this reason, in order to investigate the rigidity of the soundproof approximation, we denote the dimension reduction projection $ \mathcal P_\mrm{rd} $, defined as
\begin{equation}\label{def:dr-prjtn}
	\mathcal P_\mrm{rd}:  \left(\begin{array}{c}
\widetilde q \\ \widetilde{\mathcal H} \\ v \\ w	
\end{array}
\right) \mapsto \left(\begin{array}{c}
\widetilde{\mathcal H} \\ v \\ w	
\end{array}
\right).
\end{equation}
Notice that $ \mathcal P_\mrm{rd} $ is a bounded operator in any Sobolev space. Moreover, from \eqref{nonlinearity:full} and \eqref{nonlinearity:sp}, one can check that
\begin{equation}\label{nonlinearity:prjtn}
	\mathcal P_\mrm{rd}\mathcal N(U) = \mathcal N_\mrm{sp}(\mathcal P_\mrm{rd} U).
\end{equation}

\subsection{Soundproof approximation with ill-prepared initial data}\label{subsec:s-a-i-d}

\subsubsection*{Compactness theory of solutions to \eqref{eq:rf-EEq-dcp-02} and finite dimension truncation}

Denote by 
%\begin{equation}
$	\mathcal S_\varepsilon (t) $
%\end{equation}
the solving operator of $ \dt + \mathcal L_\varepsilon $, $ \mathcal L_\varepsilon $ as in \eqref{def:ptb-aw-opt}; that is
\begin{equation}\label{def:solv_opt}
	\dt \mathcal S_\varepsilon(t) U_0 + \mathcal L_\varepsilon \mathcal S_\varepsilon(t) U_0 = 0. 
\end{equation}
Then Proposition \ref{prop:wave-bases} implies that
$$
\begin{gathered}
\mathcal S_\varepsilon(t) U_\iota = e^{-i\omega_\iota t} U_\iota, \\
(\omega_\iota,U_\iota) \in \big\lbrace (0, U^{\mrm{mf}}_{\mrm j, (k_h,k_z)}), ~ \mrm j = 1,2,3,4,  ~ (\pm \omega^\mrm{gw}_{(k_h,k_z)}, U_{\pm, (k_h,k_z)}^\mrm{gw}),\\
 ~ (\pm \omega^\mrm{aw}_{(k_h,k_z)}, U_{\pm, (k_h,k_z)}^\mrm{aw}) \big\rbrace.
\end{gathered}
$$

Then $ \mathcal S_\varepsilon(t) $ is an isometry from $ H^s(\mathbb T^3) $ to $ H^s(\mathbb T^3) $, $ \forall s $. 
Let 
\begin{equation}\label{def:non-oscillation-variable} 
	V_\varepsilon(t):= \mathcal S_\varepsilon (-\dfrac{t}{\varepsilon})U(t), 
\end{equation} 
where $ U(t) $ is the solution to \eqref{eq:rf-EEq-dcp-02}. Then it follows from \eqref{eq:rf-EEq-dcp-02} and \eqref{def:solv_opt} that
\begin{equation}\label{eq:non-oscillation}
	\dt V_\varepsilon + \mathcal S_\varepsilon( - \dfrac{t}{\varepsilon}) \mathcal N(U) = \mathcal S_\varepsilon( - \dfrac{t}{\varepsilon}) \mathcal M(U) + \mathcal O(\varepsilon^{\mu - \sigma}).
\end{equation}
Owing to Proposition \ref{thm:uniform-est}, it is straightforward to verify that, with the same initial data for \eqref{eq:rf-EEq-dcp-02} as stated in the proposition, 
\begin{equation}\label{unest:non-oscillation}
	\sup_{0\leq t\leq T_\sigma } \bigl(\norm{\dt V_\varepsilon(t)}{\Hnorm{2}}^2 + \norm{V_\varepsilon(t)}{\Hnorm{3}}^2 \bigr) \lesssim \mathcal C \mathcal C_\mrm{in}, \qquad \sigma \in (0,\mu].
\end{equation}
%Notice that, $ \mathcal S_\varepsilon(-\dfrac{t}{\varepsilon}), \mathcal L_\varepsilon,   \mathcal P_\varepsilon^\mrm{mf},\mathcal P_\varepsilon^\mrm{gw}$, and $ \mathcal P_\varepsilon^\mrm{aw} $ commute with each other.
%With such notations, after applying $ \mathcal P_\varepsilon^\mrm{mfgw} $% and $ \mathcal P_\varepsilon^\mrm{aw} $ 
%to \eqref{eq:rf-EEq-dcp-02}, % respectively, 
%it follows that
%\begin{equation}\label{eq:soundproof-perturbed}
%	\dt \mathcal P_\varepsilon^\mrm{mfgw} U + \dfrac{1}{\varepsilon} \mathcal L_\varepsilon \mathcal P_\varepsilon^\mrm{mfgw} U + \mathcal P_\varepsilon^\mrm{mfgw} \mathcal N(U) = \mathcal P_\varepsilon^\mrm{mfgw} \mathcal M(U).
%\end{equation}
%With such notations, after applying $ \mathcal P_\varepsilon^\mrm{mfgw} $
%to \eqref{eq:non-oscillation}, 
%it follows that
%\begin{equation}\label{eq:soundproof-perturbed}
%	\dt \mathcal P_\varepsilon^\mrm{mfgw} V_\varepsilon + \mathcal S_\varepsilon(-\dfrac{t}{\varepsilon}) \mathcal P_\varepsilon^\mrm{mfgw} \mathcal N(U) = \mathcal S_\varepsilon(-\dfrac{t}{\varepsilon}) \mathcal P_\varepsilon^\mrm{mfgw} \mathcal M(U).
%\end{equation}
%
%Our goal in the rest of this subsection is to investigate the asymptotic structure of \eqref{eq:rf-EEq-dcp-02} as $ \varepsilon \rightarrow 0^+ $. 

With Proposition \ref{prop:wave-bases}, we can write $ V_\varepsilon $ as,
\begin{equation}\label{eps:V}
	V_\varepsilon = V^\mrm{mf}_\varepsilon + V^\mrm{gw}_\varepsilon + V^\mrm{aw}_\varepsilon,	
\end{equation}
with
\begin{align}
	&  \label{eps:V-mf} \begin{aligned}
		V^\mrm{mf}_\varepsilon := &  \sum_{k_h \neq (0,0)} \alpha_{1,(k_h,k_z),\varepsilon}^\mrm{mf}(t) U^\mrm{mf}_{1,(k_h,k_z)}  + \sum_{\mrm j = 3,4,~ k_z \in \mathbb Z} \alpha_{\mrm j,((0,0),k_z),\varepsilon}^\mrm{mf}(t)  U_{\mrm j,((0,0),k_z)}^\mrm{mf} \\
		& + \alpha_{2,((0,0),0),\varepsilon}^\mrm{mf}(t) U_{2,((0,0),0)}^\mrm{mf} \\
		& +  \sum_{k_z\neq 0} \alpha_{2,((0,0),k_z),\varepsilon}^\mrm{mf}(t) \varepsilon^{1-\nu} U_{2,((0,0),k_z)}^\mrm{mf},
	 \end{aligned}\\
	 &  \label{eps:V-gw}\begin{aligned}
	 	V^\mrm{gw}_\varepsilon := & \sum_{k_h\neq (0,0), k_z \neq 0} \alpha_{\pm, (k_h,k_z),\varepsilon}^\mrm{gw}(t) \varepsilon^{1-\nu} U_{\pm,(k_h,k_z)}^\mrm{gw},
	 \end{aligned}\\
	 & \label{eps:V-aw} \begin{aligned}
	 	V^\mrm{aw}_\varepsilon := & \sum_{(k_h,k_z) \neq ((0,0),0)} \alpha_{\pm,(k_h,k_z),\varepsilon}^\mrm{aw}(t) U_{\pm,(k_h,k_z)}^\mrm{aw},
	 \end{aligned}
\end{align}
where the factor $ \varepsilon^{1-\nu} $ plays the role of renormalization, such that for fixed $ k = (k_h,k_z) $, $ \varepsilon^{1-\nu} U_{2,((0,0),k_z)}^\mrm{mf}\big\lvert_{k_z\neq 0} $ and $ \varepsilon^{1-\nu} U_{\pm,(k_h,k_z)}^\mrm{gw}\big\lvert_{k_h\neq (0,0), k_z \neq 0} $ are $ \mathcal O(1) $. 
%In particular, $ \mathcal P_\varepsilon^\mrm{mfgw} V_\varepsilon = V^\mrm{mf}_\varepsilon + V^\mrm{gw}_\varepsilon $. 

Notice that the coefficients $ \alpha^\cdot_{\cdot,\cdot,\cdot}(t)$'s in \eqref{eps:V-mf}--\eqref{eps:V-aw} are equicontinuous thanks to \eqref{unest:non-oscillation}. Then, recalling \eqref{def:non-oscillation-variable}, one has
\begin{equation}\label{eps:U}
	U(t) = \mathcal S_\varepsilon (\dfrac{t}{\varepsilon}) V_\varepsilon (t) = U_\varepsilon^\mrm{mf}(t) + U_\varepsilon^\mrm{gw}(t) + U_\varepsilon^\mrm{aw}(t),
\end{equation}
with
\begin{align}
		&  \label{eps:U-mf} 
		U^\mrm{mf}_\varepsilon := V^\mrm{mf}_\varepsilon ,\\
	 &  \label{eps:U-gw}
	 	U^\mrm{gw}_\varepsilon := \sum_{k_h\neq (0,0), k_z \neq 0} e^{\mp i \frac{\omega^\mrm{gw}_{(k_h,k_z)}}{\varepsilon^{1-\nu}} \frac{t}{\varepsilon^{\nu}}}  \alpha_{\pm, (k_h,k_z),\varepsilon}^\mrm{gw}(t) \varepsilon^{1-\nu} U_{\pm,(k_h,k_z)}^\mrm{gw},
	 \\
	 & \label{eps:U-aw} 
	 	U^\mrm{aw}_\varepsilon := \sum_{(k_h,k_z) \neq ((0,0),0)} e^{\mp i \omega^\mrm{aw}_{(k_h,k_z)} \frac{t}{\varepsilon}} \alpha_{\pm,(k_h,k_z),\varepsilon}^\mrm{aw}(t) U_{\pm,(k_h,k_z)}^\mrm{aw}.
\end{align}
Meanwhile, let
\begin{align}
%	& \label{eps:mf} \begin{aligned}
%			U^\mrm{mf} := &  \sum_{k_h \neq (0,0)} \alpha_{1,(k_h,k_z),\varepsilon}^\mrm{mf}(t) \left( \begin{array}{c} 
%		 	\varepsilon p^\mrm{mf}_{\mrm{sp},1,(k_h,k_z)} \\ U^\mrm{mf}_{\mrm{sp},1,(k_h,k_z)} 
%		 \end{array}\right) \\
% 		&  + \sum \alpha_{2,((0,0),k_z),\varepsilon}^\mrm{mf}(t)  \left( \begin{array}{c} 
% 			\varepsilon p^\mrm{mf}_{\mrm{sp},2,((0,0),k_z)} \\ U^\mrm{mf}_{\mrm{sp},2,((0,0),k_z)} 
% 		\end{array}\right) \\
%		& + \alpha_{3,((0,0),0),\varepsilon}^\mrm{mf}(t) \left( \begin{array}{c} 
% 	\varepsilon p^\mrm{mf}_{\mrm{sp},3,((0,0),0)} \\ U^\mrm{mf}_{\mrm{sp},3,((0,0),0)} 
% \end{array}\right) \\
%		& +  \sum_{k_z\neq 0} \alpha_{3,((0,0),k_z),\varepsilon}^\mrm{mf}(t) \varepsilon^{1-\nu} \left( \begin{array}{c} 
% 	\varepsilon p^\mrm{mf}_{\mrm{sp},3,((0,0),k_z)} \\ U^\mrm{mf}_{\mrm{sp},3,((0,0),k_z)} 
% \end{array}\right),
%	\end{aligned}	\\
	&  \label{eps:gw} 
	 	\mathfrak U^\mrm{gw} := \sum_{k_h\neq (0,0), k_z \neq 0} e^{\mp i \frac{\omega^\mrm{gw}_{\mrm{sp},(k_h,k_z)}}{\varepsilon^{1-\nu}} \frac{t}{\varepsilon^{\nu}}}  \alpha_{\pm, (k_h,k_z),\varepsilon}^\mrm{gw}(t) \varepsilon^{1-\nu} \left( \begin{array}{c} 
 	\varepsilon p^\mrm{gw}_{\mrm{sp},(k_h,k_z)} \\ U^\mrm{gw}_{\mrm{sp},\pm,(k_h,k_z)} 
 \end{array}\right),
	 \\
	 & \label{eps:aw} 
	 	\mathfrak U^\mrm{aw} := \sum_{(k_h,k_z) \neq ((0,0),0)} e^{\mp i \omega^\mrm{aw}_{(k_h,k_z)} \frac{t}{\varepsilon}} \alpha_{\pm,(k_h,k_z),\varepsilon}^\mrm{aw}(t) U_{\mrm a, \pm,(k_h,k_z)}^\mrm{aw}.
\end{align}
Notice that $ \mathfrak U^\mrm{gw} $ and $ \mathfrak U^\mrm{aw} $ are obtained by changing the basis corresponding to the perturbed internal and acoustic waves in $ U_\varepsilon^\mrm{gw} $ and $ U^\mrm{aw}_\varepsilon $ to those corresponding to the non-perturbed ones, respectively. We don't need similar representation for $ \mathfrak U^\mrm{mf} $ thanks to \eqref{sm-mf}. However, to simplify the representation later on, we denote 
\begin{equation}\label{eps:mf} 
	\mathfrak U^\mrm{mf} := U_\varepsilon^\mrm{mf}. 
\end{equation}

%Now we are ready to write down the decomposition of $ \mathcal N(U) $. In fact, 
On the other hand, notice that
$ \mathcal N (U) = \mathcal B(U,U) $, with bilinear form $ \mathcal B(\cdot,\cdot) $ defined by
\begin{equation}\label{def:bilinear-form}
%	\begin{aligned}
		\mathcal B(U_1,U_2): = v_1 \cdot \nablah U_2 + w_1 \dz U_2 
	 + \left( \begin{array}{c}
		 \varpi_0^{-1} \widetilde q_1 (\dvh v_2 + \dz w_2) \\
		- \widetilde G \cdot \widetilde{\mathcal H}_1 w_2 \\
		0 \\ 0
	\end{array} \right),
%	\end{aligned}
\end{equation}
where
%\begin{equation*}
$
	U_j = (\widetilde q_j, \widetilde{\mathcal H}_j, v_j,  w_j)^\top, \, j = 1,2. 
$
%\end{equation*}
Then one can write
\begin{align*}
	& \mathcal N(U) = \mathcal N (U^\mrm{mf}_\varepsilon + U_\varepsilon^\mrm{gw}) \\
	& ~~~~  + \mathcal B ( U^\mrm{mf}_\varepsilon + U_\varepsilon^\mrm{gw}, U_\varepsilon^\mrm{aw}) + \mathcal B( U_\varepsilon^\mrm{aw}, U^\mrm{mf}_\varepsilon + U_\varepsilon^\mrm{gw}) \\
	& ~~~~ + \mathcal N ( U^\mrm{aw}_\varepsilon).
\end{align*}

In addition, let $ T_k $, $ k \in \mathbb N^+ $, be a finite dimensional truncation defined as
\begin{equation}\label{def:k-truncation}
	T_k U := \sum_{ \lvert k_h\lvert\leq k , \lvert k_z\lvert \leq k} \left(
	\begin{array}{c}
		Q_{({k}_h,k_z)}e^{ i {k}_h \cdot x} \cos( k_z z) \\  H_{({k}_h,k_z)} e^{ i {k}_h \cdot x} \sin( k_z z) \\ 
		V_{({k}_h,k_z)} e^{ i {k}_h \cdot x} \cos( k_z z) \\   W_{({k}_h,k_z)} e^{ i {k}_h \cdot x} \sin( k_z z)
	\end{array}
	\right)
\end{equation}
for $ U $ in \eqref{def:frr-epsn}. For the sake of clear representation, we assume that $ T_k $ applies to $ U_\mrm{sp} $ in a similar method. 

Then thanks to the uniform estimates obtained in Proposition \ref{thm:uniform-est}, $ \norm{U(t)-T_kU(t)}{\Hnorm{1}} \rightarrow 0 $, as $ k \rightarrow \infty $, and the convergence is uniform-in-$ \varepsilon $. Therefore, to analyze $ \mathcal N(U) $, it suffices to analyze $ \mathcal N(T_kU ) $. 

Let us begin with $ \mathcal N(T_k U_\varepsilon^\mrm{aw}) $. In particular, thanks to \eqref{def:ac-flw} and \eqref{prt-bs-aw}, by denoting $ T_k \mathfrak U^\mrm{aw} = ( Q_k, 0 , \nablah P_k, \partial_z P_k)^\top $, one has
\begin{equation}\label{eps:acoustic-resonance}
\begin{aligned}
	 & \mathcal N ( T_k U_\varepsilon^\mrm{aw} )  =  \mathcal N ( T_k \mathfrak U^\mrm{aw} ) + \mathcal O( \varepsilon^{1-\nu} )\\
	 &  = \mathcal N ( \left(\begin{array}{c}
	 	Q_k \\ 0 \\ \nablah P_k \\ \partial_z P_k
	 \end{array} \right) ) + \mathcal O( \varepsilon^{1-\nu} ) = \left(\begin{array}{c}
	 	 (\nabla P_k\cdot \nabla) Q_k + \varpi_0^{-1} Q_k \Delta P_k \\ 0 \\ \dfrac{1}{2} \nablah \lvert\nabla P_k\rvert^2 \\ \dfrac{1}{2} \partial_z \lvert\nabla P_k\rvert^2
	 \end{array} \right) + \mathcal O(\varepsilon^{1-\nu}).
\end{aligned}
\end{equation}
Moreover, 
\begin{equation}\label{eps:a-r}
\begin{aligned}
	 & \mathcal N ( T_k U_\varepsilon^\mrm{aw} ) =  \sum_{\mathclap{\lvert k_h\lvert, \lvert k_z\lvert, \lvert k_h'\lvert, \lvert k_z'\lvert \leq k}} e^{\mp i (\omega^\mrm{aw}_{\mrm a, (k_h,k_z)} + \omega^\mrm{aw}_{\mrm a, (k_h',k_z')}) \frac{t}{\varepsilon}} \\
	 & \qquad\qquad \times e^{\mp i (\omega^\mrm{aw}_{(k_h,k_z)} - \omega^\mrm{aw}_{\mrm a, (k_h,k_z)} + \omega^\mrm{aw}_{(k_h',k_z')} - \omega^\mrm{aw}_{\mrm a, (k_h',k_z')} ) \frac{t}{\varepsilon}} 
	 \alpha_{\pm,(k_h,k_z),\varepsilon}^\mrm{aw}\alpha_{\pm,(k_h',k_z'),\varepsilon}^\mrm{aw} \\
	 & \qquad\qquad \times \mathcal B ( U_{\pm,(k_h,k_z)}^\mrm{aw},  U_{\pm,(k_h',k_z')}^\mrm{aw}) \\
	 & \qquad + \sum_{\mathclap{\lvert k_h\lvert, \lvert k_z\lvert, \lvert k_h'\lvert, \lvert k_z'\lvert \leq k}} e^{\mp i (\omega^\mrm{aw}_{\mrm a, (k_h,k_z)} - \omega^\mrm{aw}_{\mrm a, (k_h',k_z')}) \frac{t}{\varepsilon}} \\
	 & \qquad\qquad \times e^{\mp i ((\omega^\mrm{aw}_{(k_h,k_z)} - \omega^\mrm{aw}_{\mrm a, (k_h,k_z)}) - ( \omega^\mrm{aw}_{(k_h',k_z')} - \omega^\mrm{aw}_{\mrm a, (k_h',k_z')}) ) \frac{t}{\varepsilon}} 
	 \alpha_{\pm,(k_h,k_z),\varepsilon}^\mrm{aw}\alpha_{\mp,(k_h',k_z'),\varepsilon}^\mrm{aw} \\
	 & \qquad\qquad \times \mathcal B ( U_{\pm,(k_h,k_z)}^\mrm{aw},  U_{\mp,(k_h',k_z')}^\mrm{aw}).
\end{aligned}
\end{equation}
Therefore, the possible resonances are determined by $ (k_h,k_z), (k_h',k_z') $ such that $ \omega^\mrm{aw}_{\mrm a, (k_h,k_z)} - \omega^\mrm{aw}_{\mrm a, (k_h',k_z')} = 0 $, i.e., $ \lvert k_h\rvert^2 + \lvert k_z\rvert^2 = \lvert k_h'\rvert^2 + \lvert k_z'\rvert^2 $, and 

\begin{gather}
		(\omega^\mrm{aw}_{(k_h,k_z)} - \omega^\mrm{aw}_{\mrm a, (k_h,k_z)}) - ( \omega^\mrm{aw}_{(k_h',k_z')} - \omega^\mrm{aw}_{\mrm a, (k_h',k_z')}) = \begin{cases}
			\mathcal O(\varepsilon^{4-4\nu}) & \text{if} ~ k_z = k_z', \\
			\mathcal O(\varepsilon^{2-2\nu}) & \text{if} ~ k_z \neq k_z', 
		\end{cases}\label{eps:a-r-fq}
%		\\
%		\mathcal B( U_{\pm,(k_h,k_z)}^\mrm{aw},  U_{\mp,(k_h',k_z')}^\mrm{aw}) = \mathcal B ( U_{\mrm a, \pm,(k_h,k_z)}^\mrm{aw},  U_{\mrm a, \mp,(k_h',k_z')}^\mrm{aw}) + \mathcal O(\varepsilon^{1-\nu}),
\end{gather}
thanks to \eqref{cmp-aw-fqcy}. We remark that, since $ \nu < 1/2 $, \eqref{eps:a-r-fq} implies that there will be resonances in the second term of \eqref{eps:a-r}. However, according to \eqref{eps:acoustic-resonance}, these resonances will form a gradient in the momentum equations, and therefore will converge to the Lagrangian multiplier $ \nabla p_\mrm{sp} $ in the soundproof model. In fact, as we will see later, these resonances will not affect the dynamic of the soundproof waves. However, the same cannot be said about the $ \widetilde q $ component, which does not exist in the soundproof model. We further remark this in the end of this paper.
%\todo{the possible resonances appeared in the acoustic modes}

On the other hand, thanks to \eqref{sm-mf}, \eqref{prt-bs-gw}, and \eqref{prt-bs-aw}, one has
\begin{equation}\label{eps:nonlinearity}
	\begin{gathered}
		\mathcal N(T_k U_\varepsilon^\mrm{mf} + T_k U_\varepsilon^\mrm{gw}) = \mathcal N(T_k \mathfrak U^\mrm{mf} + T_k \mathfrak U^\mrm{gw}) + \mathcal O(\varepsilon^{2-3\nu}) + \mathcal O(\varepsilon^{2-2\nu}), \\
		\mathcal B ( T_k U^\mrm{mf}_\varepsilon + T_k U_\varepsilon^\mrm{gw}, T_k U_\varepsilon^\mrm{aw}) + \mathcal B( T_k U_\varepsilon^\mrm{aw}, T_k U^\mrm{mf}_\varepsilon + T_k U_\varepsilon^\mrm{gw}) \\
		= \mathcal B ( T_k \mathfrak U^\mrm{mf} + T_k \mathfrak U^\mrm{gw}, T_k \mathfrak U^\mrm{aw}) + \mathcal B( T_k \mathfrak U^\mrm{aw}, T_k \mathfrak U^\mrm{mf} + T_k  \mathfrak U^\mrm{gw}) + \mathcal O(\varepsilon^{1-\nu}). 
	\end{gathered}
\end{equation}
Moreover, 
\begin{equation}\label{eps:a-g-r}
	\begin{aligned}
		& \mathcal B (T_k U_\varepsilon^\mrm{gw}, T_k U_\varepsilon^\mrm{aw}) = \sum_{\mathclap{\substack{k_h \neq (0,0), k_z \neq 0, (k_h',k_z')\neq ((0,0),0),  \\ \lvert k_h\lvert,\lvert k_z\lvert,\lvert k_h'\lvert,\lvert k_z'\lvert \leq k}}} e^{\mp i (\frac{\omega^\mrm{gw}_{(k_h,k_z)}}{\varepsilon^{1-\nu}} \frac{1}{\varepsilon^{\nu}} + \omega^\mrm{aw}_{(k_h',k_z')}\frac{1}{\varepsilon} ) t} \\
		& \qquad\qquad \times \mathcal B( U^\mrm{gw}_{\pm,(k_h,k_z)}, U^\mrm{aw}_{\pm, (k_h',k_z')}) \\
		& \qquad + \sum_{\mathclap{\substack{k_h \neq (0,0), k_z \neq 0, (k_h',k_z')\neq ((0,0),0),  \\ \lvert k_h\lvert,\lvert k_z\lvert,\lvert k_h'\lvert,\lvert k_z'\lvert \leq k}}} e^{\mp i (\frac{\omega^\mrm{gw}_{(k_h,k_z)}}{\varepsilon^{1-\nu}} \frac{1}{\varepsilon^{\nu}} - \omega^\mrm{aw}_{(k_h',k_z')}\frac{1}{\varepsilon} ) t} \\
		& \qquad \qquad  \times \mathcal B ( U^\mrm{gw}_{\pm,(k_h,k_z)}, U^\mrm{aw}_{\mp, (k_h',k_z')}).
	\end{aligned}
\end{equation}
Notice that $ \frac{\omega^\mrm{gw}_{(k_h,k_z)}}{\varepsilon^{1-\nu}} \frac{1}{\varepsilon^{\nu}} - \omega^\mrm{aw}_{(k_h',k_z')}\frac{1}{\varepsilon} =  \mathcal O(\frac{1}{\varepsilon}) $, which implies that $ \mathcal B (T_k U_\varepsilon^\mrm{gw}, T_k U_\varepsilon^\mrm{aw}) $ oscillates in time with a rate of $ \mathcal O(\frac{1}{\varepsilon}) $, and thus weakly converges to zero as $ \varepsilon \rightarrow 0^+ $. 
Similar properties apply to $ \mathcal B ( T_k U^\mrm{mf}_\varepsilon, T_k U_\varepsilon^\mrm{aw}) + \mathcal B( T_k U_\varepsilon^\mrm{aw}, T_k U^\mrm{mf}_\varepsilon + T_k U_\varepsilon^\mrm{gw}) $.

\subsubsection*{Compactness theory of solutions to \eqref{eq:speq-dcp-01} and finite dimension truncation}

We refer to the property of $ U_\mrm{sp} $ such that $ \dvh v_\mrm{sp} + \dz w_\mrm{sp} = 0 $ as the soundproof property. Also, let $ \mathcal P_\sigma $ be the orthogonal projection of vector fields into the space with the soundproof property. 

Denote by 
%\begin{equation}
$	\mathcal S_\mrm{sp}(t) $
%\end{equation}
the solving operator of $$ \dt +  \mathcal L_\mrm{sp} + \left( \begin{array}{c} 0 \\ \nabla_h p \\ \dz p \end{array} \right) $$ 
in the space with the soundproof property,
$ \mathcal L_\mrm{sp} $ as in \eqref{def:vec-opt-sp}; that is
\begin{equation}\label{def:solv_sp}
	\dt \mathcal S_\mrm{sp}(t) U_{\mrm{sp},0} +  \mathcal L_\mrm{sp} \mathcal S_\mrm{sp}(t) U_{\mrm{sp},0} + \left( \begin{array}{c} 0 \\ \nabla_h p \\ \dz p \end{array} \right) = 0
\end{equation}
for some $ p $ (as the Lagrangian multiplier, which might be different from lines to lines, hereafter) 
and $ \dvh (\mathcal S_\mrm{sp}(t) U_{\mrm{sp},0})_{v_\mrm{sp}} + \dz (\mathcal S_\mrm{sp}(t) U_{\mrm{sp},0})_{w_\mrm{sp}}  $. Here $ (\cdot)_{v_\mrm{sp}} $ and $ (\cdot)_{w_\mrm{sp}} $ represent the $ v_\mrm{sp} $ and $ w_\mrm{sp} $ component, respectively. 
Then Proposition \ref{prop:wave-bases-sp} implies that
$$
\begin{gathered}
\mathcal S_\mrm{sp}(t) U_{\mrm{sp},\iota} = e^{-i\omega_{\mrm{sp},\iota} t/\eta } U_{\mrm{sp},\iota}, \\
(\omega_\iota,U_\iota) \in \big\lbrace (0, U^{\mrm{mf}}_{\mrm{sp}, \mrm j, (k_h,k_z)}), ~ \mrm j = 1,2,3,4,  ~ (\pm \omega^\mrm{gw}_{\mrm{sp},(k_h,k_z)}, U_{\mrm{sp},\pm, (k_h,k_z)}^\mrm{gw}) \big\rbrace.
\end{gathered}
$$
We remind readers that our choice of scale in Proposition \ref{prop:wave-bases-sp} implies that $ \omega^\mrm{gw}_{\mrm{sp},(k_h,k_z)} / \eta = \mathcal O(1) $.

Then, it is easy to verify that $ \mathcal S_\mrm{sp}(t) $ is an isometry from $ H^s_\sigma $ to $ H^s_\sigma $, $  \forall s $.
Here $ H^s_\sigma $ represents the $ H^s $ space with the soundproof property. Let
\begin{equation}
	V_\mrm{sp}(t) := \mathcal S_\mrm{sp}(-\dfrac{t}{\varepsilon^\nu})U_\mrm{sp}(t),
\end{equation}
where $ U_\mrm{sp}(t) $ is the solution to \eqref{eq:speq-dcp-01}. Then it follows from \eqref{eq:speq-dcp-01} and \eqref{def:solv_sp} that
\begin{equation}\label{eq:non-oscillation-sp}
	\dt V_\mrm{sp}(t) + \mathcal S_\mrm{sp}(-\dfrac{t}{\varepsilon^\nu})\mathcal P_\sigma \mathcal N_\mrm{sp}(U_\mrm{sp}) = 0.
\end{equation}
Thanks to the estimate \eqref{est:uni-total-sp}, it is straightforward to verify that, with the same initial data as in Theorem \ref{thm:sp-approximation} for \eqref{eq:speq-dcp-01}, one has
\begin{equation}\label{unest:non-oscillation-sp}
	\sup_{0\leq t \leq T_\mrm{app}} \bigl( \norm{\dt V_\mrm{sp}(t)}{H^2(\mathbb T^3)}^2 + \norm{V_\mrm{sp}(t)}{H^3(\mathbb T^3)}^2 \bigr) \leq \mathcal C_\mrm{sp,in},
\end{equation}
for some $ \mathcal C_\mrm{sp,in}\in (0,\infty) $ depending on the initial data.

Thanks to Proposition \ref{prop:wave-bases-sp}, we can write $ V_\mrm{sp} $ as,
\begin{equation}\label{eps:v-sp}
	V_\mrm{sp} = V_\mrm{sp}^\mrm{mf} + V_\mrm{sp}^\mrm{gw},
\end{equation}
with
\begin{align}
& \label{eps:V-mf-sp} \begin{aligned}
		V_\mrm{sp}^\mrm{mf}	:= & \sum_{k_h \neq (0,0)} \alpha_{1,(k_h,k_z),\mrm{sp}}^\mrm{mf}(t) U^\mrm{mf}_{\mrm{sp},1,(k_h,k_z)}  + \sum_{\mrm j = 3,4,~ k_z \in \mathbb Z} \alpha_{\mrm j,((0,0),k_z),\mrm{sp}}^\mrm{mf}(t)  U_{\mrm{sp},\mrm j,((0,0),k_z)}^\mrm{mf} \\
		& + \alpha_{2,((0,0),0),\mrm{sp}}^\mrm{mf}(t) \underbrace{U_{\mrm{sp},2,((0,0),0)}^\mrm{mf}}_{=0} \\
		& +  \sum_{k_z\neq 0} \alpha_{2,((0,0),k_z),\mrm{sp}}^\mrm{mf}(t) \varepsilon^{1-\nu} U_{\mrm{sp},2,((0,0),k_z)}^\mrm{mf},
	\end{aligned}\\
& \label{eps:V-gw-sp} \begin{aligned}
	V_\mrm{sp}^\mrm{gw} := & \sum_{k_h\neq (0,0), k_z \neq 0} \alpha_{\pm, (k_h,k_z),\mrm{sp}}^\mrm{gw}(t) \varepsilon^{1-\nu} U_{\mrm{sp},\pm,(k_h,k_z)}^\mrm{gw},
\end{aligned}
\end{align}
Thanks to \eqref{unest:non-oscillation-sp}, the coefficients $ \alpha^\cdot_{\cdot,\cdot,\cdot}(t)$'s in \eqref{eps:V-mf-sp}--\eqref{eps:V-gw-sp} are equicontinuous. 

Then, one has
\begin{equation}\label{eps:U-sp}
	U_\mrm{sp}(t) = \mathcal S_\mrm{sp}(\dfrac{t}{\varepsilon^\nu}) V_\mrm{sp}(t) = U_\mrm{sp}^\mrm{mf}(t) + U_\mrm{sp}^\mrm{gw}(t),
\end{equation}
with
\begin{align}
	&  \label{eps:U-mf-sp} 
		U^\mrm{mf}_\mrm{sp} := V^\mrm{mf}_\mrm{sp} ,\\
	 &  \label{eps:U-gw-sp}
	 	U^\mrm{gw}_\mrm{sp} := \sum_{k_h\neq (0,0), k_z \neq 0} e^{\mp i \frac{\omega^\mrm{gw}_{\mrm{sp},(k_h,k_z)}}{\varepsilon^{1-\nu}} \frac{t}{\varepsilon^{\nu}}}  \alpha_{\pm, (k_h,k_z),\mrm{sp}}^\mrm{gw}(t) \varepsilon^{1-\nu} U_{\mrm{sp},\pm,(k_h,k_z)}^\mrm{gw}.
\end{align}

On the other hand, similarly as before, $ \mathcal N_\mrm{sp}(U_\mrm{sp}) = \mathcal B_\mrm{sp}(U_\mrm{sp},U_\mrm{sp}) $, with the bilinear form $ \mathcal B_\mrm{sp}(\cdot,\cdot) $ defined by
%\todo{may or maynot need this}
\begin{equation}\label{def:bilinear-form-sp}
%	\begin{aligned}
		\mathcal B_\mrm{sp}(U_{\mrm{sp},1},U_{\mrm{sp},2}): = v_{\mrm{sp},1} \cdot \nablah U_{\mrm{sp},2} + w_{\mrm{sp},1} \dz U_{\mrm{sp},2} 
	 + \left( \begin{array}{c}
		- \sin (2\pi z) \cdot \widetilde{\mathcal H}_{\mrm{sp},1} w_{\mrm{sp},2} \\
		0 \\ 0
	\end{array} \right).
%	\end{aligned}
\end{equation}

Similar to \eqref{nonlinearity:prjtn}, one has, for $ U_j $ as in \eqref{def:bilinear-form},
\begin{equation}\label{nonlinearity:prjtn-2}
\mathcal P_\mrm{rd}\mathcal B(U_1,U_2) = \mathcal B_\mrm{sp} (\mathcal P_\mrm{rd}U_1, \mathcal P_\mrm{rd} U_2).
\end{equation}

\subsubsection*{Estimate of $ \mathcal P_\mrm{rd}(U_\varepsilon^\mrm{mf} + U_\varepsilon^\mrm{gw} )  - U_\mrm{sp} $}

Let $ K \in \mathbb N^+ $ be a fixed positive integer. Then thanks to the uniform estimates obtained in Proposition \ref{thm:uniform-est} and \eqref{est:uni-total-sp}, as mentioned before, \eqref{eq:rf-EEq-dcp-02} and \eqref{eq:speq-dcp-01} can be written as
\begin{gather}
\label{eq:full-101}
	\dt U + \dfrac{1}{\varepsilon} \mathcal L_\varepsilon U + \mathcal N(T_K U) = \mathcal M(T_KU) + \mathcal O(\varepsilon^{\mu - \sigma}) + Err \qquad \text{and} \\
\label{eq:sp-101}
	\dt U_\mrm{sp} + \dfrac{1}{\varepsilon^\nu} \mathcal L_\mrm{sp} U_\mrm{sp} + \left( \begin{array}{c} 0 \\ \nablah p_\mrm{sp} \\ \dz p_\mrm{sp}	
 \end{array}
\right) + \mathcal N_\mrm{sp}(T_KU_\mrm{sp}) = Err,
\end{gather}
respectively, where $ Err $ represents the truncation error, satisfying 
\begin{equation}\label{err:truncation}
	 \norm{Err}{H^1} \rightarrow 0 \qquad \text{uniformly-in-$\varepsilon $ \quad as} \quad  K \rightarrow \infty. 	
 \end{equation}

Recalling \eqref{eps:V-mf}, \eqref{eps:U}, \eqref{eps:U-mf}, and \eqref{eps:U-gw}, one has
\begin{align}
& \label{cal:ev-mf} \begin{aligned}
	& \dt U^\mrm{mf}_\varepsilon + \dfrac{1}{\varepsilon}\mathcal L_\varepsilon U_\varepsilon^\mrm{mf} = \sum_{k_h \neq (0,0)} \dt \alpha_{1,(k_h,k_z),\varepsilon}^\mrm{mf}(t) \left(\begin{array}{c} \varepsilon p^\mrm{mf}_{\mrm{sp},1,(k_h,k_z)} \\ U^\mrm{mf}_{\mrm{sp},1,(k_h,k_z)}\end{array} \right)  \\
		& \qquad + \sum_{\mrm j = 3,4,~ k_z \in \mathbb Z} \dt \alpha_{\mrm j,((0,0),k_z),\varepsilon}^\mrm{mf}(t)  \left(\begin{array}{c} \varepsilon p^\mrm{mf}_{\mrm{sp},\mrm j, ((0,0),k_z)} \\ U_{\mrm{sp}, \mrm j,((0,0),k_z)}^\mrm{mf} \end{array}\right) \\
		& \qquad + \dt \alpha_{2,((0,0),0),\varepsilon}^\mrm{mf}(t) \left( \begin{array}{c} \varepsilon p_{\mrm{sp},2((0,0),0)}^\mrm{mf} \\ U_{\mrm{sp}, 2,((0,0),0)}^\mrm{mf} \end{array}\right) \\
		& \qquad +  \sum_{k_z\neq 0} \dt \alpha_{2,((0,0),k_z),\varepsilon}^\mrm{mf}(t) \varepsilon^{1-\nu} \left( \begin{array}{c} \varepsilon p^\mrm{mf}_{\mrm{sp},2,((0,0),k_z)}\\ U_{\mrm{sp},2,((0,0),k_z)}^\mrm{mf}\end{array}\right),
\end{aligned}\\
& \label{cal:ev-gw}\begin{aligned}
	& \dt U_\varepsilon^\mrm{gw} + \dfrac{1}{\varepsilon}\mathcal L_\varepsilon U_\varepsilon^\mrm{gw} = \sum_{\mathclap{k_h\neq (0,0), k_z \neq 0}} e^{\mp i \frac{\omega^\mrm{gw}_{(k_h,k_z)}}{\varepsilon^{1-\nu}} \frac{t}{\varepsilon^{\nu}}}  \dt \alpha_{\pm, (k_h,k_z),\varepsilon}^\mrm{gw}(t) \varepsilon^{1-\nu} U_{\pm,(k_h,k_z)}^\mrm{gw}\\
	& = \sum_{\mathclap{k_h\neq (0,0), k_z \neq 0}} \biggl\lbrack e^{\mp i \frac{\omega^\mrm{gw}_{\mrm{sp},(k_h,k_z)}}{\varepsilon^{1-\nu}} \frac{t}{\varepsilon^{\nu}}} + \mathcal O(\varepsilon^{2-3\nu}) \biggr\rbrack  \dt \alpha_{\pm, (k_h,k_z),\varepsilon}^\mrm{gw}(t) \\
	& \qquad\qquad \times \varepsilon^{1-\nu} \biggl\lbrack \left(\begin{array}{c} 
 	\varepsilon p_{\mrm{sp}, (k_h,k_z)}^\mrm{gw} \\ U_{\mrm{sp}, \pm,(k_h,k_z)}^\mrm{gw}
 \end{array} \right)  + \mathcal O(\varepsilon^{1-\nu})\biggr\rbrack 
 	,
\end{aligned}\\
& \label{cal:ev-aw}\begin{aligned}
	& \dt U_\varepsilon^\mrm{aw} + \dfrac{1}{\varepsilon}\mathcal L_\varepsilon U_\varepsilon^\mrm{aw} = \sum_{\mathclap{(k_h,k_z) \neq ((0,0),0)}} e^{\mp i \omega^\mrm{aw}_{(k_h,k_z)} \frac{t}{\varepsilon}} \dt \alpha_{\pm,(k_h,k_z),\varepsilon}^\mrm{aw}(t) U_{\pm,(k_h,k_z)}^\mrm{aw}\\
	& = \sum_{\mathclap{(k_h,k_z) \neq ((0,0),0)}} e^{\mp i \omega^\mrm{aw}_{(k_h,k_z)} \frac{t}{\varepsilon}} \dt \alpha_{\pm,(k_h,k_z),\varepsilon}^\mrm{aw}(t) \biggl\lbrack U_{\mrm a,\pm,(k_h,k_z)}^\mrm{aw} + \mathcal O(\varepsilon^{1-\nu})\biggr\rbrack,
\end{aligned}
\end{align}
thanks to \eqref{sm-mf}, \eqref{prt-fq-gw}, and \eqref{prt-bs-gw}. 

%Similarly, recalling \eqref{eps:V-mf-sp}, \eqref{eps:U-sp}, \eqref{eps:U-mf-sp}, and \eqref{eps:U-gw-sp}, one has
%\begin{align}
%& \label{cal:ev-mf-sp}
%\begin{aligned}
%	& \dt U^\mrm{mf}_\mrm{sp} + \dfrac{1}{\varepsilon^\nu}\mathcal L_\mrm{sp} U^\mrm{mf}_\mrm{sp} =\sum_{k_h \neq (0,0)} \dt \alpha_{1,(k_h,k_z),\mrm{sp}}^\mrm{mf}(t) U^\mrm{mf}_{\mrm{sp},1,(k_h,k_z)} \\
%	&\qquad + \sum_{\mrm j = 3,4,~ k_z \in \mathbb Z} \dt \alpha_{\mrm j,((0,0),k_z),\mrm{sp}}^\mrm{mf}(t)  U_{\mrm{sp},\mrm j,((0,0),k_z)}^\mrm{mf} \\
%%		&\qquad + \alpha_{2,((0,0),0),\mrm{sp}}^\mrm{mf}(t) \underbrace{U_{\mrm{sp},2,((0,0),0)}^\mrm{mf}}_{=0} \\
%		&\qquad +  \sum_{k_z\neq 0} \dt \alpha_{2,((0,0),k_z),\mrm{sp}}^\mrm{mf}(t) \varepsilon^{1-\nu} U_{\mrm{sp},2,((0,0),k_z)}^\mrm{mf},
%\end{aligned}\\
%& \label{cal:ev-gw-sp}
%\begin{aligned}
%	& \dt U^\mrm{gw}_\mrm{sp} + \dfrac{1}{\varepsilon^\nu} \mathcal L_\mrm{sp} U^\mrm{gw}_\mrm{sp} = \sum_{\mathclap{k_h\neq (0,0), k_z \neq 0}} e^{\mp i \frac{\omega^\mrm{gw}_{\mrm{sp},(k_h,k_z)}}{\varepsilon^{1-\nu}} \frac{t}{\varepsilon^{\nu}}}  \dt \alpha_{\pm, (k_h,k_z),\mrm{sp}}^\mrm{gw}(t) \varepsilon^{1-\nu} U_{\mrm{sp},\pm,(k_h,k_z)}^\mrm{gw}.
%\end{aligned}
%\end{align}

On the other hand, one can check from Proposition \ref{prop:wave-bases-sp}, $ \bigl\lbrace (U^\mrm{mf}_{\mrm{sp}, \mrm j, (k_h,k_z)})_{\mrm j = 1,2,3,4}, U^\mrm{gw}_{\mrm{sp},\pm, (k_h,k_z)} \bigr\rbrace $ forms a orthogonal basis and satisfies the soundproof property. 
Denote by the projection operators to $\spn\lbrace U^\mrm{mf}_{\mrm{sp},\mrm{j},(k_h,k_z)} \rbrace $ and $\spn\lbrace U^\mrm{gw}_{\mrm{sp},\pm,(k_h,k_z)} \rbrace $,  defined as 
\begin{equation}\label{def:project-sp}
\begin{aligned}
	\mathcal P^\mrm{mf}_{\mrm{sp}	,1,(k_h,k_z)}(\cdot) := & \proj_{\spn\lbrace U_{\mrm{sp},1,(k_h,k_z)}^\mrm{mf} \rbrace}(\cdot),  && k_h \neq (0,0), \\
	\mathcal P^\mrm{mf}_{\mrm{sp}	,2,((0,0),k_z)}(\cdot) := & \proj_{\spn\lbrace U_{\mrm{sp},2,((0,0),k_z)}^\mrm{mf} \rbrace}(\cdot),  && k_z \neq 0, \\
	\mathcal P^\mrm{mf}_{\mrm{sp}	,\mrm j,((0,0),k_z)}(\cdot) := & \proj_{\spn\lbrace U_{\mrm{sp},\mrm j,((0,0),k_z)}^\mrm{mf} \rbrace}(\cdot),  && \mrm j = 3,4,\\
	\mathcal P^\mrm{gw}_{\mrm{sp},\pm,(k_h,k_z)}(\cdot) := & \proj_{\spn\lbrace U^\mrm{gw}_{\mrm{sp},\pm,(k_h,k_z)}\rbrace}(\cdot), && k_h \neq (0,0), k_z \neq 0. 
\end{aligned}
\end{equation}

Now we are ready to filter out the acoustic waves in \eqref{eq:full-101} by projections. In the following, we always assume $ \lvert k_h\lvert, \lvert k_z\lvert \leq K $, and the restrictions on $ (k_h,k_z) $ as in \eqref{def:project-sp} apply.

First, thanks to \eqref{def:ac-flw}, \eqref{nonlinearity:prjtn}, \eqref{eps:acoustic-resonance}, \eqref{eps:nonlinearity}, \eqref{cal:ev-mf}, \eqref{cal:ev-gw}, and \eqref{cal:ev-aw}, one can calculate that, recalling $ 0 \leq 2 \nu \leq 1 $, 
\begin{align*}
 	& \mathcal P_{\mrm{sp}, \mrm j, (k_h,k_z) }^\mrm{mf} \mathcal P_\mrm{rd} \mathcal N(T_K U) =  \mathcal P_{\mrm{sp}, \mrm j, (k_h,k_z) }^\mrm{mf} \mathcal N_\mrm{sp}(T_K \mathcal P_\mrm{rd} (\mathfrak U^\mrm{mf}+\mathfrak U^\mrm{gw}))\\
% 	& \qquad + \underbrace{\mathcal P_{\mrm{sp}, \mrm j, (k_h,k_z) }^\mrm{mf} \mathcal P_\mrm{rd} \mathcal B(T_K U_\varepsilon^\mrm{mf} + T_k U_\varepsilon^\mrm{gw},T_K U_\varepsilon^\mrm{aw}) + \mathcal P_{\mrm{sp}, \mrm j, (k_h,k_z) }^\mrm{mf} \mathcal P_\mrm{rd} \mathcal B(T_K U_\varepsilon^\mrm{aw},T_K U_\varepsilon^\mrm{mf} + T_k U_\varepsilon^\mrm{gw})}_{\text{oscillate in time with rate $ \mathcal O(\frac{1}{\varepsilon}) $}} \\
 	& \qquad + \mathcal O(\varepsilon^{1-\nu}) + \text{oscillation in time with rate $ \mathcal O(\frac{1}{\varepsilon}) $}, \\
 	& \mathcal P_{\mrm{sp}, \pm, (k_h,k_z) }^\mrm{gw} \mathcal P_\mrm{rd} \mathcal N(T_K U) =  \mathcal P_{\mrm{sp}, \pm, (k_h,k_z) }^\mrm{gw} \mathcal N_\mrm{sp}(T_K \mathcal P_\mrm{rd} (\mathfrak U^\mrm{mf}+\mathfrak U^\mrm{gw}))\\
 	& \qquad + \mathcal O(\varepsilon^{1-\nu}) + \text{oscillation in time with rate $ \mathcal O(\frac{1}{\varepsilon}) $}, \\
 	& \mathcal P_\mrm{rd} \mathcal M(T_K U) = 0,\\
 	& \mathcal P_{\mrm{sp}, \mrm j, (k_h,k_z) }^\mrm{mf} \mathcal P_\mrm{rd}(\dt U + \dfrac{1}{\varepsilon}\mathcal L_\varepsilon U) = \dt \alpha^\mrm{mf}_{\mrm{j},(k_h,k_z),\varepsilon} U^\mrm{mf}_{\mrm{sp},\mrm{j},(k_h,k_z)}, \mrm j \neq 2, \\
 	& \mathcal P_{\mrm{sp}, 2, (k_h,k_z) }^\mrm{mf} \mathcal P_\mrm{rd}(\dt U + \dfrac{1}{\varepsilon}\mathcal L_\varepsilon U) = \dt \alpha^\mrm{mf}_{2,(k_h,k_z),\varepsilon} \varepsilon^{1-\nu}U^\mrm{mf}_{\mrm{sp},\mrm{2},(k_h,k_z)}, \\
 	& \mathcal P_{\mrm{sp}, \pm, (k_h,k_z) }^\mrm{gw} \mathcal P_\mrm{rd} (\dt U + \dfrac{1}{\varepsilon}\mathcal L_\varepsilon U) = e^{\mp i \frac{\omega^\mrm{gw}_{\mrm{sp},(k_h,k_z)}}{\varepsilon^{1-\nu}} \frac{t}{\varepsilon^\nu} } \dt \alpha^\mrm{gw}_{\pm,(k_h,k_z),\varepsilon} \varepsilon^{1-\nu} U^\mrm{gw}_{\mrm{sp},\pm,(k_h,k_z)} \\
 	& \qquad + \mathcal O(\varepsilon^{2-3\nu}).
 \end{align*}
 In particular, recalling $ \mathfrak U^\mrm{mf}$ and $ \mathfrak U^\mrm{gw} $ in \eqref{eps:mf} and \eqref{eps:gw}, similar calculation as in \eqref{cal:ev-mf}--\eqref{cal:ev-aw} for $ (\dt + \dfrac{1}{\varepsilon^\nu} \mathcal L_\mrm{sp}) (T_K\mathcal P_\mrm{rd} (\mathfrak U^\mrm{mf} + \mathfrak U^\mrm{gw})) $ yields that
 \begin{equation}\label{proj:mg-101}
 \begin{gathered}
 	\sum_{\mathclap{\substack{\mrm j = 1,2,3,4, \mrm j' = +, -\\\lvert k_h\lvert,\lvert k_z\lvert \leq K}}}(\mathcal P^\mrm{mf}_{\mrm{sp},\mrm j, (k_h,k_z)} + \mathcal P^\mrm{gw}_{\mrm{sp},\mrm j', (k_h,k_z)}) \mathcal P_\mrm{rd}(\dt U + \dfrac{1}{\varepsilon} \mathcal L_\varepsilon U) \\
 	= (\dt + \dfrac{1}{\varepsilon^\nu} \mathcal L_\mrm{sp}) (T_K \mathcal P_\mrm{rd} (\mathfrak U^\mrm{mf} + \mathfrak U^\mrm{gw})) + \mathcal C_K \mathcal O(\varepsilon^{2-3\nu}).
 	\end{gathered}
 \end{equation}
 
 Therefore, denote by
 \begin{equation}
 	\mathcal P_{\mrm{sp},K}^{\mrm{mf+gw}}:=\sum_{\mathclap{\substack{\mrm j = 1,2,3,4, \mrm j' = +, -\\\lvert k_h\lvert,\lvert k_z\lvert \leq K}}} (\mathcal P^\mrm{mf}_{\mrm{sp},\mrm j, (k_h,k_z)} + \mathcal P^\mrm{gw}_{\mrm{sp},\mrm j', (k_h,k_z)}).
 \end{equation}
Applying $$ \mathcal P_{\mrm{sp},K}^{\mrm{mf+gw}} \mathcal P_\mrm{rd} $$
  to \eqref{eq:full-101} yields, since $ 0 < 2\nu < 1 $,
 \begin{equation}\label{eq:proj-mg-101}
 	\begin{gathered}
 	 (\dt + \dfrac{1}{\varepsilon^\nu} \mathcal L_\mrm{sp}) (T_K \mathcal P_\mrm{rd} (\mathfrak U^\mrm{mf} +  \mathfrak U^\mrm{gw})) % + \left( \begin{array}{c}
% 		0 \\ \nabla_h p' \\ \dz p'
% 	\end{array}\right) \\
% 	& \qquad 
 	 + \mathcal P_{\mrm{sp},K}^{\mrm{mf+gw}} \mathcal N_\mrm{sp}( T_K \mathcal P_\mrm{rd} (\mathfrak U^\mrm{mf} +  \mathfrak U^\mrm{gw}) ) \\
 	  \qquad =  \mathcal C_K\mathcal O(\varepsilon^{1-\nu}) +  \mathcal C_K\mathcal O(\varepsilon^{\mu - \sigma}) 
% 	& \qquad \qquad 
 	+ \text{oscillation in time with rate $ \mathcal O(\frac{1}{\varepsilon})$}
% 	& \qquad\qquad 
 	+ Err.
 	\end{gathered} 
 \end{equation}
 On the other hand, with similar calculation as in \eqref{cal:ev-mf} and \eqref{cal:ev-gw}, one can conclude that
 $$
 	T_K(\dt + \dfrac{1}{\varepsilon^\nu}\mathcal L_\mrm{sp}) U_\mrm{sp} = (\dt + \dfrac{1}{\varepsilon^\nu}\mathcal L_\mrm{sp}) T_K U_\mrm{sp}.
 $$
 Consequently, applying $ \mathcal P_{\mrm{sp},K}^{\mrm{mf+gw}} $ to \eqref{eq:sp-101} yields 
 \begin{equation}\label{eq:proj-sp-101}
 	 (\dt + \dfrac{1}{\varepsilon^\nu}\mathcal L_\mrm{sp}) T_K U_\mrm{sp} %+ \left( \begin{array}{c}
% 		0 \\ \nabla_h p'_\mrm{sp} \\ \dz p'_\mrm{sp}
% 	\end{array}\right) 
 	+ \mathcal P_{\mrm{sp},K}^{\mrm{mf+gw}} \mathcal N_\mrm{sp}(T_K U_\mrm{sp}) = Err.
 \end{equation}
% Here, $ p' $ and $ p'_\mrm{sp} $ are the scaler pressures acting as Lagrangian multipliers due to the soundproof property. Indeed, $ \bigl( 0 , \nablah p', \dz p' \bigr)^\top $ and $ \bigl( 0 , \nablah p'_\mrm{sp}, \dz p'_mrm{sp} \bigr)^\top $ cancel the acoustic components of $ \mathcal N_\mrm{sp}(T_K \mathcal P_\mrm{rd} (U^\mrm{mf} +  U^\mrm{gw})) $ and $ \mathcal N_\mrm{sp}(T_K U_\mrm{sp}) $ in the above equations,  respectively. 
Here, although not exactly the same expression as before, $ Err $ satisfies \eqref{err:truncation}.
 
 Then, after subtracting \eqref{eq:proj-mg-101} with \eqref{eq:proj-sp-101}, and taking the $ L^2 $-inner product of the resultant equations with $  2(T_K \mathcal P_\mrm{rd} (\mathfrak U^\mrm{mf} +  \mathfrak U^\mrm{gw}) - T_K U_\mrm{sp})  $, with similar calculation as in section \ref{subsec:snpf-ap}, we arrive at the estimate
 \begin{equation}\label{eq:proj-mg-sp-101}
 	\begin{gathered}
 	\dfrac{d}{dt} \norm{T_K \mathcal P_\mrm{rd} (\mathfrak U^\mrm{mf} +  \mathfrak U^\mrm{gw}) - T_K U_\mrm{sp}}{L^2}^2 \leq \mathcal C \norm{T_K \mathcal P_\mrm{rd} (\mathfrak U^\mrm{mf} +  \mathfrak U^\mrm{gw}) - T_K U_\mrm{sp}}{L^2}^2\\
 	+ \mathcal C_K\mathcal O(\varepsilon^{2-2\nu}) +  \mathcal C_K\mathcal O(\varepsilon^{2\mu - 2\sigma}) \\
 	\qquad \qquad 
 	+ \text{oscillation in time with rate $ \mathcal O(\frac{1}{\varepsilon})$}
% 	& \qquad\qquad 
 	+ Err,
 	\end{gathered}
 \end{equation}
 where we use the fact that
 $$
 \begin{gathered}
 \int \biggl\lbrace \text{oscillation in time with rate $ \mathcal O(\frac{1}{\varepsilon})$}\biggr\rbrace \cdot \underbrace{(T_K \mathcal P_\mrm{rd} (\mathfrak U^\mrm{mf} +  \mathfrak U^\mrm{gw}) - T_K U_\mrm{sp})}_{\text{oscillation at rate $ \mathcal O(\frac{1}{\varepsilon^{\nu}})$}}  \idx \\
  = \text{oscillation in time with rate $ \mathcal O(\frac{1}{\varepsilon})$}.
 \end{gathered}
 $$
 We would like to emphasize that it is important that we get an estimate with coefficient $ \mathcal C $ independent of $ K $ on the right hand side of \eqref{eq:proj-mg-sp-101}. Otherwise when applying Gr\"onwall's inequality, below, it would arrive at an estimate with uncontrollable $ Err $. This is possible thanks to the soundproof property of $ T_K \mathcal P_\mrm{rd} (\mathfrak U^\mrm{mf} +  \mathfrak U^\mrm{gw}) - T_K U_\mrm{sp} $ and cancellation when applying integration by parts, as it is done in section \ref{subsec:snpf-ap}. 
 
 Then integrating \eqref{eq:proj-mg-sp-101} in time yields, since $ 2-2\nu > 1 $, for $ 0< t \leq T_\mrm{\sigma, mg} < \min\lbrace T_\sigma,T_\mrm{sp}\rbrace $ with some $ T_\mrm{\sigma, mg} \in (0,\infty) $,
 \begin{equation}\label{est:proj-mg-sp-102}
 	\begin{gathered}
 		\norm{T_K \mathcal P_\mrm{rd} (\mathfrak U^\mrm{mf} +  \mathfrak U^\mrm{gw})(t) - T_K U_\mrm{sp}(t)}{L^2}^2 \\ \leq \norm{T_K \mathcal P_\mrm{rd} (\mathfrak U^\mrm{mf} +  \mathfrak U^\mrm{gw})(0) - T_K U_\mrm{sp}(0)}{L^2}^2\\
 		+ \int_0^t \mathcal C \norm{T_K \mathcal P_\mrm{rd} (\mathfrak U^\mrm{mf} +  \mathfrak U^\mrm{gw})(s) - T_K U_\mrm{sp}(s)}{L^2}^2 \,ds\\
 		 + C_K (\mathcal O(\varepsilon^{2\mu - 2\sigma}) + \mathcal O(\varepsilon) ) + Err.
 	\end{gathered}
 \end{equation}
 
 We would like to remind readers that $ \mathfrak U^\mrm{gw} $ and $ \mathfrak U^\mrm{mf} $ as in \eqref{eps:gw} and \eqref{eps:mf}, thanks to \eqref{sm-mf}, \eqref{prt-fq-gw}, and \eqref{prt-bs-gw}, satisfy
 \begin{equation}\label{proj:mg-102}
 	T_K \mathcal P_\mrm{rd} \mathfrak U^\mrm{mf} = T_K \mathcal P_\mrm{rd} U^\mrm{mf}_\varepsilon \qquad \text{and} \qquad T_K \mathcal P_\mrm{rd} \mathfrak U^\mrm{gw} = T_K \mathcal P_\mrm{rd} U^\mrm{gw}_\varepsilon + \mathcal C_K \mathcal O(\varepsilon^{2-3\nu}),
 \end{equation}
 and thus, since $ 4-6\nu = 1 + 3(1-2\nu) > 1 $,
 \begin{equation}\label{proj:mg-103}
 	\norm{T_K \mathcal P_\mrm{rd} (\mathfrak U^\mrm{mf} +  \mathfrak U^\mrm{gw}) - T_K \mathcal P_\mrm{rd} (U^\mrm{mf}_\varepsilon +  U^\mrm{gw}_\varepsilon)}{L^2}^2 = \mathcal C_\mrm{K} \mathcal O(\varepsilon^{4-6\nu}) \leq 
 	\mathcal C_\mrm{K} \mathcal O(\varepsilon).
 \end{equation}
 Consequently, after choosing appropriate initial data for $  U_\mrm{sp} $ which carries the initial mean flows and internal waves, one can derive from \eqref{est:proj-mg-sp-102} that 
 \begin{equation}\label{est:proj-mg-sp-103}
 	\norm{T_K \mathcal P_\mrm{rd} (U^\mrm{mf}_\varepsilon +  U^\mrm{gw}_\varepsilon)(t) - T_K U_\mrm{sp}(t)}{L^2}^2 \leq C_K (\mathcal O(\varepsilon^{2\mu - 2\sigma}) + \mathcal O(\varepsilon) ) + Err,
 \end{equation}
 after applying Gr\"onwall's inequality and \eqref{proj:mg-103}.
 We remind readers that $ Err $ satisfies \eqref{err:truncation}. Thus from \eqref{est:proj-mg-sp-103}, one can conclude Theorem \ref{thm:sp-ill-prepared}.

%\pagebreak

\subsection{Remarks}

In section \ref{subsec:f-s-eq}, we introduce the dimension reduction operator $ \mathcal P_\mrm{rd} $ in \eqref{def:dr-prjtn}, which is used in section \ref{subsec:s-a-i-d} to prove Theorem \ref{thm:sp-ill-prepared}; that is, the asymptotic behavior of the $ \widetilde{\mathcal H}(\widetilde{\mathcal H}_\mrm{sp}), v(v_\mrm{sp}), w(w_\mrm{sp}) $ components. However, the asymptotic behavior of the $ \widetilde q $ component is not discussed.

In the case of well-prepared initial data, i.e., in Theorem \ref{thm:sp-approximation}, we choose initially $ \widetilde q $ and $ \varepsilon p_{\mrm{sp}}$ (equivalently $ \widetilde q_\mrm{in} $ and $ \varepsilon p_{\mrm{ms},\mrm{in}} $) close. In particular, since $ \int p _\mrm{sp} \idx = 0 $ in the soundproof system, the well-prepared initial data should satisfy that $ \int \widetilde q_\mrm{in}\idx $ is close to zero, which is {\bf not} the case for the ill-prepared initial data. In particular, $  \int \widetilde q\idx = 0 $ is not a conservative property for the full system \eqref{eq:rf-EEq-dcp-01}. 

That is, the $ \widetilde q $ component is {\bf nontrivial} in both the slow waves and fast waves in the case of ill-prepared initial data (see, for instance, \eqref{eps:acoustic-resonance}). However, these nontrivial waves do not have influence on the $ \widetilde{\mathcal H}(\widetilde{\mathcal H}_\mrm{sp}), v(v_\mrm{sp}), w(w_\mrm{sp}) $ components of the mean flows and internal waves of the solutions to \eqref{eq:rf-EEq-dcp-01} (\eqref{eq:speq-dcp-01}, respectively). In particular, there is no $ \widetilde q $ component in the solutions to \eqref{eq:speq-dcp-01}. This is why our asymptotic analysis works and has to be done after applying the dimension reduction $ \mathcal P_\mrm{rd} $ to system \eqref{eq:rf-EEq-dcp-01}, in the case of ill-prepared initial data. %Nevertheless, the asymptotic behavior of the $ \widetilde q $ component remains interesting and will be discussed in the future. 

\section{Appendix}

Finally, although it is {straightforward}, we would like to record the representation of the waves decomposition of the full compressible system. With the Fourier representations \eqref{def:frr-epsn}, we calculate the mean flow part first. When $ k_h \neq (0,0) $, noticing that $ \lvert U_{1,(k_h,k_z)}^\mrm{mf}\rvert^2  = \varsigma := \int_0^1 \cos^2(k_z z) \,dz = \begin{cases}
\frac 1 2 & \text{if} ~ k_z \neq 0 \\ 1 & \text{if} ~ k_z = 0
\end{cases} $, 
\begin{align*} 
	& \dfrac{1}{\lvert U_{1,(k_h,k_z)}^\mrm{mf}\rvert^2} \int U \cdot \overline{U_{1,(k_h,k_z)}^\mrm{mf}}^c \idx =  \dfrac{V_{(k_h,k_z)} \cdot k_h^\perp}{\lvert k_h\lvert}.
	\end{align*}
	When $ k_h = (0,0) $, noticing that $ \lvert U_{2,((0,0),k_z)}^\mrm{mf}\rvert^2 = \varsigma + \dfrac{ k_z^2}{\eta^2}(1-\varsigma) $, $ \lvert U_{3,((0,0),k_z)}^\mrm{mf}\rvert^2 = \lvert U_{4,((0,0),k_z)}^\mrm{mf}\rvert^2 = \varsigma$, 
	\begin{align*}
		& \dfrac{1}{\lvert U_{2,((0,0),k_z)}^\mrm{mf}\rvert^2} \int U \cdot \overline{U_{2,((0,0),k_z)}^\mrm{mf}}^c \idx = \dfrac{\eta^2 Q_{((0,0),k_z)} \varsigma +  k_z \eta H_{((0,0),k_z)}(1-\varsigma) }{\eta^2 \varsigma +  k_z^2(1-\varsigma) }, \\
		& \dfrac{1}{\lvert U_{3,((0,0),k_z)}^\mrm{mf}\rvert^2} \int U \cdot \overline{U_{3,((0,0),k_z)}^\mrm{mf}}^c \idx = (V_{((0,0),k_z)})_1, \\
		& \dfrac{1}{\lvert U_{4,((0,0),k_z)}^\mrm{mf}\rvert^2} \int U \cdot \overline{U_{4,((0,0),k_z)}^\mrm{mf}}^c \idx = (V_{((0,0),k_z)})_2
	\end{align*}
	Therefore, the mean flow projection of $ U $ is given by
	\begin{equation}\label{def:U-mf}
		\begin{aligned}
			& U^\mrm{mf}_\varepsilon := \mathcal P_\varepsilon^\mrm{mf} U = \sum_{k_h \in 2\pi \mathbb Z^2\backslash \lbrace(0,0)\rbrace, k_z \in 2\pi \mathbb N} \dfrac{V_{(k_h,k_z)} \cdot k_h^\perp}{\lvert k_h\lvert} U_{1,(k_h,k_z)}^\mrm{mf} \\
			& \qquad + \sum_{k_z \in 2\pi \mathbb N} \biggl( \dfrac{\eta^2 Q_{((0,0),k_z)}\varsigma+  k_z \eta H_{((0,0),k_z)}(1-\varsigma) }{\eta^2 \varsigma +  k_z^2 (1-\varsigma) } U_{2,((0,0),k_z)}^\mrm{mf} \\
			& \qquad\qquad + (V_{((0,0),k_z)})_1 U_{3,((0,0),k_z)}^\mrm{mf}
			 + (V_{((0,0),k_z)})_2 U_{4,((0,0),k_z)}^\mrm{mf} \biggr) . 
		\end{aligned}
	\end{equation}

Next, we calculate the internal wave part of $ U $. Notice that for $ k_h \neq (0,0), k_z \neq 0 $, 
\begin{align*} 
	& \lvert U_{\pm,(k_h,k_z)}^\mrm{gw}\rvert^2 =  \dfrac{1}{2} \biggl( 1 + \dfrac{ \lvert k_h\rvert^2}{(\omega_{(k_h,k_z)}^\mrm{gw})^2}\biggr) \\
	& \qquad + \dfrac{1}{2} \biggl( \dfrac{ \lvert k_h\rvert^2}{(\omega_{(k_h,k_z)}^\mrm{gw})^2} - 1 \biggr)^2 \biggl( \dfrac{\eta^2 + (\omega_{(k_h,k_z)}^\mrm{gw})^2}{( k_z)^2} \biggr)\\
	& = 1 + \dfrac{\eta^2}{(k_z)^2}\biggl( \dfrac{ \lvert k_h\rvert^2 }{(\omega_{(k_h,k_z)}^\mrm{gw})^2} - 1 \biggr)^2,
\end{align*}
where $ \omega_{(k_h,k_z)}^\mrm{gw} $ is given by \eqref{def:gw-fqcy}. Here we have used the fact that $ \int U_{+,(k_h,k_z)}^\mrm{gw} \cdot \overline{U_{-,(k_h,k_z)}^\mrm{gw}}^c \idx = 0 $, which yields
\begin{gather*}
%	\dfrac{1}{2} + \dfrac{1}{2} \dfrac{\eta^2}{( k_z)^2}\biggl( \dfrac{ \lvert k_h\rvert^2 }{(\omega_{(k_h,k_z)}^\mrm{gw})^2} - 1 \biggr)^2 - \dfrac{1}{2} \dfrac{ \lvert k_h\rvert^2}{(\omega_{(k_h,k_z)}^\mrm{gw})^2} \\ - \dfrac{1}{2} \dfrac{(\omega_{(k_h,k_z)}^\mrm{gw})^2}{( k_z)^2}\biggl( \dfrac{ \lvert k_h\rvert^2 }{(\omega_{(k_h,k_z)}^\mrm{gw})^2} - 1 \biggr)^2 
	\dfrac{1}{2} \biggl( 1 - \dfrac{ \lvert k_h\rvert^2}{(\omega_{(k_h,k_z)}^\mrm{gw})^2}\biggr) + \dfrac{1}{2} \biggl( \dfrac{ \lvert k_h\rvert^2}{(\omega_{(k_h,k_z)}^\mrm{gw})^2} - 1 \biggr)^2 \biggl( \dfrac{\eta^2 - (\omega_{(k_h,k_z)}^\mrm{gw})^2}{( k_z)^2} \biggr)
	= 0.
\end{gather*}
In addition,
\begin{align*}
	& \dfrac{1}{\lvert U_{\pm,(k_h,k_z)}^\mrm{gw}\rvert^2} \int U \cdot \overline{U_{\pm,(k_h,k_z)}^\mrm{gw}}^c \idx = \dfrac{1}{2 + \dfrac{2\eta^2}{( k_z)^2}\biggl( \dfrac{ \lvert k_h\rvert^2 }{(\omega_{(k_h,k_z)}^\mrm{gw})^2} - 1 \biggr)^2}  \\
	& \qquad \times \biggl\lbrack Q_{(k_h,k_z)} \pm \dfrac{1}{\omega^\mrm{gw}_{(k_h,k_z)}} V_{(k_h,k_z)} \cdot k_h \\
	& \qquad\qquad +  \dfrac{H_{(k_h,k_z)} \eta \mp i W_{(k_h,k_z)} \omega^\mrm{gw}_{(k_h,k_z)}}{ k_z} \biggl( \dfrac{\lvert k_h\rvert^2}{(\omega^\mrm{gw}_{(k_h,k_z)})^2}-1\biggr) \biggr\rbrack.
%	, \\
%	& \dfrac{1}{\lvertU_{-,(k_h,k_z)}^\mrm{gw}\rvert^2} \int U \cdot \overline{U_{-,(k_h,k_z)}^\mrm{gw}}^c \idx = \dfrac{1}{2 + \dfrac{2\eta^2}{(2\pi k_z)^2}\biggl( \dfrac{(2\pi)^2 \lvert k_h\rvert^2 }{(\omega_{(k_h,k_z)}^\mrm{gw})^2} - 1 \biggr)^2} \\
%	& \qquad \times \biggl\lbrack Q_{(k_h,k_z)} - \dfrac{2\pi}{\omega^\mrm{gw}_{(k_h,k_z)}} V_{(k_h,k_z)} \cdot k_h \\
%	& \qquad\qquad + \dfrac{H_{(k_h,k_z)} \eta + W_{(k_h,k_z)} \omega^\mrm{gw}_{(k_h,k_z)}}{2\pi k_z} \biggl( \dfrac{(2\pi)^2\lvert k_h\rvert^2}{(\omega^\mrm{gw}_{(k_h,k_z)})^2}-1\biggr) \biggr\rbrack
\end{align*}
Therefore, the internal wave projection of $ U $ is given by
\begin{equation}\label{def:U-gw}
\begin{aligned}
	& U^\mrm{gw}_\varepsilon := \mathcal P_\varepsilon^\mrm{gw} U= \sum_{k_h \in 2\pi \mathbb Z^2 \backslash \lbrace (0,0)\rbrace , k_z \in 2\pi \mathbb N^+} \dfrac{1}{2 + \dfrac{2\eta^2}{( k_z)^2}\biggl( \dfrac{ \lvert k_h\rvert^2 }{(\omega_{(k_h,k_z)}^\mrm{gw})^2} - 1 \biggr)^2}  \\
	& \qquad \times \biggl\lbrack Q_{(k_h,k_z)} \pm \dfrac{1}{\omega^\mrm{gw}_{(k_h,k_z)}} V_{(k_h,k_z)} \cdot k_h \\
	& \qquad\qquad + \dfrac{H_{(k_h,k_z)} \eta \mp i W_{(k_h,k_z)} \omega^\mrm{gw}_{(k_h,k_z)}}{ k_z} \biggl( \dfrac{\lvert k_h\rvert^2}{(\omega^\mrm{gw}_{(k_h,k_z)})^2}-1\biggr) \biggr\rbrack U_{\pm,(k_h,k_z)}^\mrm{gw}.
	\end{aligned}
\end{equation}

The calculation of the acoustic wave  part of $ U $ is similar for $ k_z \neq 0 $, which is
\begin{align*}
	& \dfrac{1}{\lvert U_{\pm,(k_h,k_z)}^\mrm{aw}\rvert^2} \int U \cdot \overline{U_{\pm,(k_h,k_z)}^\mrm{aw}}^c \idx = \dfrac{1}{2 + \dfrac{2\eta^2}{( k_z)^2}\biggl( \dfrac{( \lvert k_h\rvert^2 }{(\omega_{(k_h,k_z)}^\mrm{aw})^2} - 1 \biggr)^2}  \\
	& \qquad \times \biggl\lbrack Q_{(k_h,k_z)} \pm \dfrac{1}{\omega^\mrm{aw}_{(k_h,k_z)}} V_{(k_h,k_z)} \cdot k_h \\
	& \qquad\qquad + \dfrac{H_{(k_h,k_z)} \eta \mp i W_{(k_h,k_z)} \omega^\mrm{aw}_{(k_h,k_z)}}{ k_z} \biggl( \dfrac{\lvert k_h\rvert^2}{(\omega^\mrm{aw}_{(k_h,k_z)})^2}-1\biggr) \biggr\rbrack.
\end{align*}
On the other hand, when $ k_z = 0 $, $ k_h\neq (0,0) $, we have $ \lvert U_{\pm,(k_h,0)}^\mrm{aw}\rvert^2 =1 + \dfrac{\lvert k_h\rvert^2}{(\omega^\mrm{aw}_{(k_h,0)})^2}  $, and
\begin{align*}
	& \dfrac{1}{\lvert U_{\pm,(k_h,0)}^\mrm{aw}\rvert^2} \int U \cdot \overline{U_{\pm,(k_h,0)}^\mrm{aw}}^c \idx = \dfrac{Q_{(k_h,0)} \pm \dfrac{k_h \cdot V_{(k_h,0)}}{\omega^\mrm{aw}_{(k_h,0)}}}{1 + \dfrac{\lvert k_h\rvert^2}{(\omega^\mrm{aw}_{(k_h,0)})^2}}.
\end{align*}
Consequently, the perturbed acoustic wave projection of $ U $ is given by
\begin{equation}\label{def:U-aw}
	\begin{aligned}
		& U^\mrm{aw}_\varepsilon := \mathcal P_\varepsilon^\mrm{aw} U = \sum_{k_h \in 2\pi \mathbb Z^2, k_z \in 2\pi \mathbb N^+}\dfrac{1}{2 + \dfrac{2\eta^2}{( k_z)^2}\biggl( \dfrac{ \lvert k_h\rvert^2 }{(\omega_{(k_h,k_z)}^\mrm{aw})^2} - 1 \biggr)^2}  \\
	& \qquad \times \biggl\lbrack Q_{(k_h,k_z)} \pm \dfrac{1}{\omega^\mrm{aw}_{(k_h,k_z)}} V_{(k_h,k_z)} \cdot k_h \\
	& \qquad\qquad + \dfrac{H_{(k_h,k_z)} \eta \mp i W_{(k_h,k_z)} \omega^\mrm{aw}_{(k_h,k_z)}}{ k_z} \biggl( \dfrac{\lvert k_h\rvert^2}{(\omega^\mrm{aw}_{(k_h,k_z)})^2}-1\biggr) \biggr\rbrack U_{\pm,(k_h,k_z)}^\mrm{aw} \\
	& \quad + \sum_{k_h \in 2\pi \mathbb Z^2\backslash \lbrace(0,0) \rbrace} \dfrac{Q_{(k_h,0)} \pm \dfrac{ k_h \cdot V_{(k_h,0)}}{\omega^\mrm{aw}_{(k_h,0)}}}{1 + \dfrac{\lvert k_h\rvert^2}{(\omega^\mrm{aw}_{(k_h,0)})^2}}U_{\pm,(k_h,0)}^\mrm{aw}.
	\end{aligned}
\end{equation}
Here $ \omega^\mrm{aw}_{(k_h,k_z)} $ is given as \eqref{def:aw-fqcy}.

\section*{Acknowledgement}

D. Bresch is supported by the SingFlows project grant ANR-18-CE40-0027.
X.L. and R.K. acknowledge support of Deutsche Forschungsgemeinschaft through CRC 1114 "Scaling Cascades in Complex Systems", project number 235221301, Projects A02 "Multiscale data and asymptotic model assimilation for atmospheric flows" and C06 ``Multi-scale structure of atmospheric vortices'', and of Einstein Stiftung Berlin through their Einstein Fellowship Program. X.L. would like to thank the Isaac Newton Institute for Mathematical Sciences, Cambridge, for support and hospitality during the programme``{\it Mathematical aspects of turbulence: where do we stand?}'' where part of the work on this paper was undertaken. This work was supported in part by EPSRC grant no EP/R014604/1. X.L.'s work was partially supported by a grant from the Simons Foundation, during his visit to the Isaac Newton Institute for Mathematical Sciences.

\section*{Declarations}

The authors declare that there is no conflict of interest concerning this work. \\

\noindent
Data sharing is not applicable to this article as no datasets were generated or analysed during the current study.

\bibliographystyle{plain}
%\bibliography{library}

\end{document}